\documentclass{amsart}

\usepackage{amsfonts}
\usepackage{graphicx}
\usepackage{hyperref}
\usepackage[ngerman, english]{babel}
\usepackage{datetime}
\usepackage{enumitem} 
\usepackage{mathtools}
\usepackage{cite}
\usepackage{url}
\usepackage[a4paper,margin=20ex]{geometry}
\usepackage{xpatch}
\usepackage{amssymb}
\usepackage{amsfonts}
\usepackage{amsmath}
\usepackage{microtype}
\usepackage{hyperref}
\usepackage[dvipsnames]{xcolor}
\usepackage{tikz}
\usepackage{pgfplots}

\hypersetup{colorlinks,breaklinks,
	citecolor=[rgb]{0.2,0.58,0.2},
	urlcolor=[rgb]{0,0.5,0.5},
	linkcolor=[rgb]{0,0,1}}

\xpatchcmd{\proof}
{\itshape}
{\bfseries}
{}
{}

\newcommand\blfootnote[1]{%
  \begingroup
  \renewcommand\thefootnote{}\footnote{#1}%
  \addtocounter{footnote}{-1}%
  \endgroup
}

\makeatletter
\newcommand*{\rom}[1]{\expandafter\@slowromancap\romannumeral #1@}
\makeatother
\def\Xint#1{\mathchoice
{\XXint\displaystyle\textstyle{#1}}%
{\XXint\textstyle\scriptstyle{#1}}%
{\XXint\scriptstyle\scriptscriptstyle{#1}}%
{\XXint\scriptscriptstyle\scriptscriptstyle{#1}}%
\!\int}
\def\XXint#1#2#3{{\setbox0=\hbox{$#1{#2#3}{\int}$ }
\vcenter{\hbox{$#2#3$ }}\kern-.582\wd0}}

\def\dashint{\Xint-}

\newtheorem*{defin}{Definition}
\newtheorem{thm}{Theorem}[section]
\newtheorem{cor}[thm]{Corollary}
\newtheorem{lem}[thm]{Lemma}
\newtheorem{rem}[thm]{Remark}

\newtheorem{prop}[thm]{Proposition}

\numberwithin{equation}{section}

\title[Improved Sobolev regularity for nonlocal equations with VMO coefficients]{Improved Sobolev regularity for linear nonlocal equations with VMO coefficients}
\author{Simon Nowak}
\address{Universit\"at Bielefeld, Fakult\"at f\"ur Mathematik, Postfach 100131, D-33501 Bielefeld, Germany}
\email{simon.nowak@uni-bielefeld.de}
\begin{document}

\maketitle
\begin{abstract}
This work is concerned with both higher integrability and differentiability for linear nonlocal equations with possibly very irregular coefficients of VMO-type or even coefficients that are merely small in BMO. In particular, such coefficients might be discontinuous. While for corresponding local elliptic equations with VMO coefficients such a gain of Sobolev regularity along the differentiability scale is unattainable, it was already observed in previous works that gaining differentiability in our nonlocal setting is possible under less restrictive assumptions than in the local setting. In this paper, we follow this direction and show that under assumptions on the right-hand side that allow for an arbitrarily small gain of integrability, weak solutions $u \in W^{s,2}$ in fact belong to $W^{t,p}_{loc}$ for \emph{any} $s \leq t < \min\{2s,1\}$, where $p>2$ reflects the amount of integrability gained. In other words, our gain of differentiability does not depend on the amount of integrability we are able to gain. This extends numerous results in previous works, where either continuity of the coefficient was required or only an in general smaller gain of differentiability was proved.
\end{abstract}
\pagestyle{headings}

\section{Introduction} 
\subsection{Nonlocal equations} \label{setting}
We study the Sobolev regularity of weak solutions to linear nonlocal integro-differential equations of the form \blfootnote{Supported by SFB 1283 of the German Research Foundation.}
\begin{equation} \label{nonlocaleq}
L_A u =  f \text{ in } \Omega \subset \mathbb{R}^n,
\end{equation}
where $\Omega \subset \mathbb{R}^n$ is a domain (= open set) and $A:\mathbb{R}^n \times \mathbb{R}^n \to \mathbb{R}$ is a coefficient. In addition, for some fixed parameter $s \in (0,1)$ the nonlocal operator $L_A$ is formally given by
\begin{equation} \label{no}
L_A u(x) := p.v. \int_{\mathbb{R}^n} \frac{A(x,y)}{|x-y|^{n+2s}} (u(x)-u(y))dy, \quad x \in \Omega.
\end{equation}
Throughout the paper, for the sake of simplicity we assume that $n>2s$.
Furthermore, we require that the coefficient $A$ is measurable and that there exists some constant $\Lambda \geq 1$ such that
\begin{equation} \label{eq1}
\Lambda^{-1} \leq A(x,y) \leq \Lambda \text{ for almost all } x,y \in \mathbb{R}^n.
\end{equation}
Moreover, we assume that $A$ is symmetric, that is,
\begin{equation} \label{symmetry}
A(x,y)=A(y,x) \text{ for almost all } x,y \in \mathbb{R}^n.
\end{equation}
We define $\mathcal{L}_0(\Lambda)$ as the class of all such measurable coefficients $A$ that satisfy (\ref{eq1}) and (\ref{symmetry}). \par
Building on the results and techniques from our previous work \cite{MeV}, the aim of this paper is to show that under appropriate regularity assumptions on $A$ and $f$, weak solutions to (\ref{nonlocaleq}), which are initially assumed to belong to the fractional Sobolev space $W^{s,2}(\mathbb{R}^n)$,
in fact belong to higher-order spaces $W^{t,p}_{loc}(\Omega)$ for some $p>2$ and \emph{any} $s \leq t<\min\{2s,1\}$. For the relevant definitions of these spaces, we refer to section \ref{fracSob}. \par
Concerning our precise notion of weak solutions, denoting by $W^{s,2}_c(\Omega)$ the set of all functions that belong to $W^{s,2}(\mathbb{R}^n)$ and are compactly supported in $\Omega$, we have the following definition.
\begin{defin}
	Given $f \in L^\frac{2n}{n+2s}_{loc}(\Omega)$, we say that $u \in W^{s,2}(\mathbb{R}^n)$ is a weak solution of the equation $L_A u = f$ in $\Omega$, if 
	\begin{equation} \label{weaksolx1}
	\int_{\mathbb{R}^n} \int_{\mathbb{R}^n} \frac{A(x,y)}{|x-y|^{n+2s}} (u(x)-u(y))(\varphi(x)-\varphi(y))dydx = \int_{\Omega} f \varphi dx \quad \forall \varphi \in W^{s,2}_c(\Omega).
	\end{equation}
\end{defin}

\subsection{VMO coefficients} \label{VMOco}
Before stating our main results, we need to recall our notion of coefficients with vanishing mean oscillation which was introduced in \cite{MeV}.

\begin{defin}
Let $\delta>0$ and $A \in \mathcal{L}_0(\Lambda)$. We say that $A$ is $\delta$-vanishing in a ball $B \subset \mathbb{R}^n$, if for any $r>0$ and all $x_0,y_0 \in B$ with $B_r(x_0) \subset B$ and $B_r(y_0) \subset B$, we have $$ \dashint_{B_r(x_0)} \dashint_{B_r(y_0)} |A(x,y)-\overline A_{r,x_0,y_0}|dydx \leq \delta ,$$
where $\overline A_{r,x_0,y_0}:= \dashint_{B_r(x_0)} \dashint_{B_r(y_0)} A(x,y)dydx$. \par
Moreover, we say that $A$ is $(\delta,R)$-BMO in a domain $\Omega \subset \mathbb{R}^n$ and for some $R>0$, if for any $z \in \Omega$ and any $0<r\leq R$ with $B_r(z) \Subset \Omega$, $A$ is $\delta$-vanishing in $B_r(z)$. \par
Finally, we say that $A$ is VMO in $\Omega$, if for any $\delta>0$, there exists some $R>0$ such that $A$ is $(\delta,R)$-BMO in $\Omega$.
\end{defin}
Let us briefly put the above definition into a more classical context. In case $A$ belongs to the classical space of functions with vanishing mean oscillation $\textnormal{VMO}(\mathbb{R}^{2n})$ (see e.g.\ \cite[Section 2.1.1]{PS}, \cite{DiF} or \cite{Sarason}), then $A$ is also VMO in $\mathbb{R}^n$. However, our assumption that $A$ is VMO in $\Omega$ is more general, in the sense that we essentially only assume $A$ to be of vanishing mean oscillation in some arbitrarily small open neighbourhood of the diagonal in $\Omega \times \Omega$, while away from the diagonal in $\Omega \times \Omega$ and outside of $\Omega \times \Omega$ $A$ is not required to possess any regularity at all. In particular, any coefficient $A$ that is continuous in an open neighbourhood of the diagonal in $\Omega \times \Omega$ is VMO in $\Omega$. 
Nevertheless, continuity close to the diagonal is not essential in order for a coefficient to be VMO. \par Indeed, the class of discontinuous VMO functions is actually rather rich.
For instance, assuming that $\Omega$ contains the origin, if for some $\alpha \in (0,1)$ we have
\begin{equation} \label{ex}
	\begin{aligned}
		A(x,y)=
		\begin{cases}
			\textnormal{sin} \left (|\textnormal{log}(|x|+|y|)|^\alpha \right )+2 & \text{ if } x \neq 0 \text{ or } y \neq 0 \\
			0 & \text{ if } x=y=0
		\end{cases}
	\end{aligned}
\end{equation}
or
\begin{equation} \label{ex1}
	\begin{aligned}
		A(x,y)=
		\begin{cases}
			\textnormal{sin} \left (\textnormal{log}|\textnormal{log}(|x|+|y|)| \right )+2 & \text{ if } x \neq 0 \text{ or } y \neq 0 \\
			0 & \text{ if } x=y=0
		\end{cases}
	\end{aligned}
\end{equation}
in an open neighbourhood of $\textnormal{diag}(\Omega \times \Omega)$, then $A$ is VMO in $\Omega$. However, in both cases $A$ is discontinuous at $x=y=0$.
\subsection{Main results} \label{mr}
We are now in the position to state our main results.
\begin{thm} \label{mainint5z}
	Let $\Omega \subset \mathbb{R}^n$ be a domain, $s \in (0,1)$ and $\Lambda \geq 1$. If $A \in \mathcal{L}_0(\Lambda)$ is VMO in $\Omega$, 
	then for any weak solution $u \in W^{s,2}(\mathbb{R}^n)$
	of the equation
	$$
	L_A u = f \text{ in } \Omega,
	$$
	any $p \in (2,\infty)$ and any $s \leq t <\min\{2s,1\}$, we have the implication $$f \in L^\frac{np}{n+(2s-t)p}_{loc}(\Omega) \implies u \in W^{t,p}_{loc}(\Omega).$$
\end{thm}
If we are only interested in arriving at the conclusion that $u \in W^{t,p}_{loc}(\Omega)$ for some fixed $t$ and some fixed $p$, then it suffices for $A$ to be small in BMO, as our second main result indicates, in which we also state an explicit estimate on the solution.

\begin{thm} \label{mainint5}
	Let $\Omega \subset \mathbb{R}^n$ be a domain, $s \in (0,1)$, $\Lambda \geq 1$ and $R>0$. Moreover, fix some $p \in (2,\infty)$ and some $s<t<\min \{2s,1\}$. Then there exists some small enough $\delta=\delta(p,n,s,t,\Lambda)>0$, such that if $A \in \mathcal{L}_0(\Lambda)$ is $(\delta,R)$-BMO in $\Omega$,
	then for any weak solution $u \in W^{s,2}(\mathbb{R}^n)$
	of the equation
	$$
	L_A u = f \text{ in } \Omega,
	$$
	we have the implication $$f \in L^\frac{np}{n+(2s-t)p}_{loc}(\Omega) \implies u \in W^{t,p}_{loc}(\Omega).$$ In addition, for all relatively compact bounded open sets ${\Omega^\prime} \Subset {\Omega^{\prime \prime}} \Subset \Omega$, we have the estimate
	\begin{equation} \label{Wstest}
		[u]_{W^{t,p}(\Omega^\prime)} \leq C \left ([u]_{W^{s,2}(\mathbb{R}^n)} + ||f||_{L^{\frac{np}{n+(2s-t)p}}(\Omega^{\prime \prime})} \right ),
	\end{equation}
	where $C=C(n,s,t,\Lambda,R,p,\Omega^\prime,\Omega^{\prime \prime})>0$.
\end{thm}

We stated Theorem \ref{mainint5z} and Theorem \ref{mainint5} in terms of the higher integrability exponent $p$ at which we arrive. Since in some circumstances it might be more natural to instead prescribe the integrability of the source function $f$, we also state the following reformulation of Theorem \ref{mainint5z}.
\begin{thm} \label{mainint5zx}
	Let $\Omega \subset \mathbb{R}^n$ be a domain, $s \in (0,1)$, $\Lambda \geq 1$, fix some $s \leq t < \min \{2s,1\}$ and let $f \in L^q_{loc}(\Omega)$ for some $q \in \left (\frac{2n}{n+2(2s-t)},\infty \right )$. In addition, assume that $A \in \mathcal{L}_0(\Lambda)$ is VMO in $\Omega$. Then for any weak solution $u \in W^{s,2}(\mathbb{R}^n)$
	of the equation
	$L_A u = f \text{ in } \Omega,$ we have
	$$u \in \begin{cases} W^{t,\frac{nq}{n-(2s-t)q}}_{loc}(\Omega) , & \text{ if } q<\frac{n}{2s-t} \\ W^{t,p}_{loc}(\Omega) \text{ for any } p \in (1,\infty), & \textnormal{ if } q \geq \frac{n}{2s-t}. \end{cases}$$
\end{thm}

Since for any $s < t < \min \{2s,1\}$ we have $\frac{2n}{n+2(2s-t)} < 2$, Theorem \ref{mainint5zx} in particular implies the following higher differentiability result for nonlocal equations with right-hand side in $L^2$.
\begin{thm} \label{higherdiff}
	Let $\Omega \subset \mathbb{R}^n$ be a domain, $s \in (0,1)$, $\Lambda \geq 1$ and $f \in L^2_{loc}(\Omega)$. In addition, assume that $A \in \mathcal{L}_0(\Lambda)$ is VMO in $\Omega$. Then for any weak solution $u \in W^{s,2}(\mathbb{R}^n)$
	of the equation
	$L_A u = f \text{ in } \Omega,$ we have $u \in W^{t,2}_{loc}(\Omega)$ for any $s < t < \min \{2s,1\}$.
\end{thm}

\begin{rem} \label{mainrem} \normalfont
Actually, the conclusions of Theorem \ref{mainint5z}, Theorem \ref{mainint5zx} and Theorem \ref{higherdiff} also remain valid for a class of coefficients $A$ that in general might not be VMO, including in particular irregular coefficients that are translation invariant inside of $\Omega$. \par More precisely, our approach is flexible enough in order to include the case when $A \in \mathcal{L}_0(\Lambda)$ satisfies $A(x,y)=a(x-y)$ for all $x,y \in \Omega$ and some measurable function $a: \mathbb{R}^n \to \mathbb{R}$, but is not required to satisfy any additional regularity assumption. For a more elaborate discussion regarding this extension of our main results, we refer to Remark \ref{endremark}.
\end{rem}

\subsection{Local elliptic equations with VMO coefficients}
From the point of view of the regularity theory for local elliptic equations, our main results can be considered to be somewhat surprising. In order to illustrate this at first glance surprising nature of our main results, let us briefly consider local second-order elliptic equations in divergence form of the type
\begin{equation} \label{localeq}
\textnormal{div}(B \nabla u)= f \quad \text{in } \Omega,
\end{equation}
where the matrix of coefficients $B=\{b_{ij}\}_{i,j=1}^n$ is assumed to be uniformly elliptic and bounded. As it is for instance rigorously established in \cite{GKV}, the equation (\ref{localeq}) can be thought of as a local analogue of the nonlocal equation (\ref{nonlocaleq}) corresponding to the limit case $s=1$. Therefore, it might be intuitive to guess that the regularity properties of solutions to the nonlocal equation (\ref{nonlocaleq}) should in some sense correspond to the ones of the equation (\ref{localeq}). However, it turns out that in the context of higher regularity, this is \emph{not true} at all.  \par A classical fact (see e.g.\ \cite{Morrey,Simader}) is that if the coefficients $b_{ij}$ are continuous in $\Omega$ and if $f \in L^\frac{np}{n+p}_{loc}(\Omega)$ for some $p>2$, then weak solutions $u \in W^{1,2}_{loc}(\Omega)$ of the equation (\ref{localeq}) belong to $W^{1,p}_{loc}(\Omega)$. While for equations with general measurable coefficients such a gain of regularity is not achievable, it was nevertheless realized later (see \cite{DiF}) that the above assertion remains true if the continuity assumption on the coefficients is relaxed to assuming that the coefficients belong to the space of functions with vanishing mean oscillation $\textnormal{VMO}(\Omega)$ (see also e.g.\ \cite{Kin,AM,DongKim1,BD} for some more general developments). In addition, if one is only interested in obtaining $W^{1,p}_{loc}$ regularity for some fixed $p$, then similar to our Theorem \ref{mainint5}, in more recent years it was observed that it suffices for $B$ to be small in BMO, see \cite{ByunR,ByunLp}. However, in contrast to our main results, the results mentioned above do not yield any differentiability gain. \par And indeed, in order to gain any amount of differentiability along the Sobolev scale in the setting of local equations, a corresponding amount of differentiability has to be imposed on the coefficients, which can already be observed in one-dimensional examples (see e.g.\ \cite[section 1]{selfimpro}). Thus, in the setting of local elliptic equations with VMO or even continuous coefficients in general \emph{no differentiability gain at all} is attainable. In contrast, our main results show that in the setting of nonlocal equations with VMO coefficients, the differentiabilty of weak solutions improves quite significantly.
Let us give some further illustrations of these improved regularizing effects of nonlocal equations contained in our main results. \par In fact, in the case when $s \leq 1/2$, we are able to almost match the optimal Calder\'on-Zygmund-type Sobolev regularity for the fractional Laplacian, which corresponds to the case when the coefficient $A$ is constant. Namely, it is known that for the weak solution of the Dirichlet problem
$$
	\begin{cases} \normalfont
		(-\Delta)^s u = f & \text{ in } \Omega \\
		u=0 & \text{ a.e. in } \mathbb{R}^n \setminus \Omega,
	\end{cases}
$$ we have
$u \in W^{2s,p}_{loc}(\Omega)$ whenever $f \in L^p(\Omega)$ for some $p \in \left [2,\infty \right )$ (see \cite{Warma}), while our main results show that despite the presence of a general VMO coefficient $A$ in (\ref{nonlocaleq}), for $s \leq 1/2$ weak solutions of (\ref{nonlocaleq}) still belong to $W^{t,p}_{loc}(\Omega)$ for any $t<2s$ whenever $f \in L^p_{loc}(\Omega)$. This is in sharp contrast to the setting of local second-order equations, since weak solutions $u \in W^{1,2}_{loc}(\Omega)$ to the Poisson equation $\Delta u= f$ in $\Omega$ belong to $W^{2,p}_{loc}(\Omega)$ whenever $f \in L^p_{loc}(\Omega)$, gaining a full weak derivative, while as mentioned above, in the presence of VMO coefficients in (\ref{localeq}) in general not even a gain of fractional differentiability can be expected. \par 
In the case when $s > 1/2$, our main results only yield differentiability for any $t<1$, so that in this case we are no longer able to almost match the optimal Sobolev regularity for the fractional Laplacian. However, this seems natural to us, since we do not expect that the differentiability of solutions to local second-order equations can be exceeded by solutions to corresponding nonlocal equations of lower order. Nevertheless, for $s \geq 1/2$ our main results in particular show that weak solutions to nonlocal equations with VMO coefficients of the type (\ref{nonlocaleq}) almost share the amount differentiability that weak solutions to local equations with VMO coefficients of the type (\ref{localeq}) possess, despite the fact that the order of such nonlocal equations is lower.

\subsection{Previous related results} \label{pr}
By now, there is a substantial amount of works concerning the regularity theory for weak solutions to nonlocal equations of the type (\ref{nonlocaleq}). \par This is especially true concerning regularity results of purely nonlocal type, in the sense that the obtained results do not have analogues in the regularity theory of local elliptic equations.
This line of results was started in the papers \cite{selfimpro} and \cite{Schikorra}, where it was demonstrated that in the case of general bounded measurable coefficients $A \in \mathcal{L}_0(\Lambda)$, weak solutions to nonlocal equations of the type (\ref{nonlocaleq}) are slightly higher differentiable and higher integrable, provided the right-hand side satisfies $f \in L^q_{loc}$ for some $q>\frac{2n}{n+2s}$. Our main results show that under the additional assumption that $A$ is VMO, the conclusions of the results in \cite{selfimpro,Schikorra} can be improved to gaining larger amounts of differentiability and integrability. \par
Concerning results on higher Sobolev regularity for nonlocal equations of the type (\ref{nonlocaleq}), in \cite{MSY} Mengesha, Schikorra and Yeepo proved results similar to our Theorem \ref{mainint5z} in the case when $\Omega=\mathbb{R}^n$ and under the assumption that the mapping $x \mapsto A(x,y)$ is globally H\"older continuous for some arbitrarily small H\"older exponent. Since this H\"older continuity assumption on $A$ in particular does not include discontinuous coefficients of VMO-type like (\ref{ex}) and (\ref{ex1}), in \cite[p.\ 10]{MSY} the authors raised the question if the regularity gain they obtained remains valid for coefficients that merely belong to VMO. Therefore, one of the main achievements of the present paper is that our main results confirm this conjecture to be true, even establishing the desired regularity in the slightly more general case when the coefficient is merely assumed to be small in BMO. Moreover, in contrast to \cite{MSY} we are also able to include translation invariant coefficients that do not satisfy any smoothness assumption, see Remark \ref{endremark}. In addition, we argue on a completely different set of techniques in comparison to the ones applied in \cite{MSY}. Namely, while the key ingredient in \cite{MSY} is given by commutator estimates, our approach is based on a delicate interplay between comparison estimates and so-called dual pairs (see section \ref{app}). \par
Furthermore, in \cite{MeV} we proved weaker versions of the main results in the present paper, in the sense that the differentiability gain obtained in \cite{MeV} depends on $n,s$ and in particular the amount of integrability that we are able to gain, while in our main results stated in section \ref{mr} an arbitrarily small gain of integrability suffices in order to gain differentiability in the full range $s < t < \min \{2s,1\}$. For this reason, the amount of differentiability gained in \cite{MeV} only matches the one in this paper in the case when a very large amount of integrability is prescribed on the right-hand side $f$, while in general the differentiability gain in this work exceeds the one obtained in \cite{MeV} by a very substantial amount. \par This is probably illustrated best in the setting of our Theorem \ref{higherdiff}: For $f \in L^2_{loc}(\Omega)$, \cite[Theorem 1.3]{MeV} only implies that $u \in W^{t,2}_{loc}(\Omega)$ for any $t$ in the restricted range
$$
	s<t < t_{n,s} := \begin{cases} \frac{ns+4s^2}{n+2s} , & \text{ if } s \leq 1/2 \\ \frac{ns+4s-4s^2}{n+2-2s}, & \text{ if } s > 1/2. \end{cases}
$$
In particular, e.g.\ for $n=2$ and $s=1/2$, \cite[Theorem 1.3]{MeV} yields differentiability for any $t<2/3$, while in this case our Theorem \ref{higherdiff} yields differentiability for any $t<1$. In higher dimensions, the improvement in differentiability gain becomes even more visible. In fact, for any $s \in (0,1)$ and any fixed $\varepsilon>0$, there exists some large enough $n=n(s,\varepsilon)$ such that $t_{n,s} < s+\varepsilon$, so that the gain of differentiability in \cite{MeV} is in general very small in the case when $f$ merely belongs to $L^2_{loc}(\Omega)$. On the other hand, for $f \in L^2_{loc}(\Omega)$ our Theorem \ref{higherdiff} implies that $u \in W^{t,2}_{loc}(\Omega)$ in the whole range $s<t<\min \{2s,1\}$, independently of $n$. \par 
Moreover, in \cite{MeN} for $p \in (2,\infty)$ it was proved that weak solutions $u$ to (\ref{nonlocaleq}) belong to $W^{s,p}_{loc}(\Omega)$ whenever $f \in L^\frac{np}{n+sp}_{loc}(\Omega)$ and $A$ is continuous in $\Omega \times \Omega$, which corresponds to the case of no differentiability gain as in the setting of local equations.
\par Also, in the case when $f \in L^2_{loc}(\Omega)$ and $A \in C^s(\Omega \times \Omega)$, by using difference quotients, in \cite{Cozzi} it was shown that weak solutions to (\ref{nonlocaleq}) belong to $W^{t,2}_{loc}(\Omega)$ for any $t < 2s$, which also follows from our Theorem \ref{higherdiff} in the case when $s \leq 1/2$.
In other words, in this case we not only do not need $C^s$ regularity of the coefficient $A$, but not even continuity of $A$ in order to achieve this higher differentiability result. In fact, it is sufficient for $A$ to be VMO in $\Omega$. \par 
More results regarding Sobolev regularity for nonlocal equations are for example proved in \cite{BL,Grubb,KassMengScott,MP,YLKS,MSc,DongKim,Auscher,Me}, while various results on H\"older regularity are proved in \cite{BLS,Fall,Fall1,MeH,NonlocalGeneral,FracLap,CSa,CCV,finnish,Kassmann,Silvestre,Peral,Stinga,CK,Chaker,CDF}. Furthermore, for some regularity results concerning nonlocal equations similar to (\ref{nonlocaleq}) in the case when the right-hand side is merely a measure, we refer to \cite{mdata}.

\subsection{Some remaining open questions and possible extensions} \label{oq}
First of all, while as we discussed in section \ref{pr} the differentiability gain in \cite{MeV} is in general substantially smaller than the gain we achieve in our main results, the main results in \cite{MeV} hold also for certain nonlinear generalizations of the equation (\ref{nonlocaleq}), while in this paper and also in \cite{MSY} only linear equations are considered. Thus, a natural question is if the improved differentiability gain in the present paper remains valid for nonlinear equations. \par 
In addition, in \cite{MSY} the lower bound we imposed on $A$ in (\ref{eq1}) is only assumed to hold at the diagonal, so that another naturally arising question is if the lower bound on $A$ can be relaxed to hold only at the diagonal also in the case when $A$ is merely VMO. \par
Furthermore, in \cite{AFLY}, in the case of the fractional Laplacian, that is, in the special case when $A$ is constant, a global regularity result corresponding to our Theorem \ref{mainint5zx} was proved under the additional restriction that $q<\frac{1}{t-s}$, which is sharp when dealing with regularity up to the boundary. In view of this global regularity result, another interesting question is to what extent the conclusions of our main results, which deal with local regularity, remain valid up to the boundary. \par
Moreover, we believe that our approach is flexible enough in order to generalize our main results to include so-called local weak solutions as considered e.g.\ in \cite{BL}, \cite{BLS} or \cite{MeH}, essentially only assuming that $u \in W^{s,2}_{loc}(\Omega)$ and the finiteness of the nonlocal tails of $u$. However, since including this slightly more general notion of solutions would require a revision of the previous work \cite{MeV} and most notably \cite{selfimpro}, we decided not to insist on this point. \par
Finally, another feature of our approach is that it also enables us to prove local $W^{t,p}$ estimates in the case when the right-hand side $f$ in (\ref{nonlocaleq}) is replaced by the fractional Laplacian or even by sums of more general nonlocal operators, see Theorem \ref{mainright} and Remark \ref{genrem}.

\subsection{Approach} \label{app}
Before commencing with the technical part of the paper, in this section we give a heuristic summary of our approach, in particular since we believe that the techniques displayed in this work have the potential to be useful in a large variety of situations involving nonlocal equations. \par
As mentioned, in the previous paper \cite{MeV}, we proved weaker versions of the main results in the present paper, gaining only a restricted amount of differentiability that depends on the amount of integrability we are able to gain. This was achieved by introducing ideas that on the one hand allow to prove suitable comparison estimates in our nonlocal setting, and on the other hand allow to combine various highly nontrivial covering techniques introduced in the papers \cite{CaffarelliPeral,selfimpro}.
Our approach in this paper essentially combines the techniques implemented in \cite{MeV} with some novel insights that enable us to gain differentiability independently of the integrability gain. Since an in-depth heuristic description of the philosophy of the approach from \cite{MeV} was already given in \cite[Section 1.5]{MeV}, here we focus on emphasizing the main novelties of the approach used in this work compared to the one applied in \cite{MeV}. \par 
The objects at the heart of the approach from \cite{MeV} are certain fractional gradients given by so-called dual pairs. Namely, for some fixed $\theta \in \left (0,\frac{1}{2} \right)$, we define a Borel measure $\mu$ on $\mathbb{R}^{2n}$ as follows. For any function $u:\mathbb{R}^n \to \mathbb{R}$ and $(x,y) \in \mathbb{R}^{2n}$ with $x \neq y$, we define the function 
\begin{equation} \label{Udef}
	U(x,y):=\frac{|u(x)-u(y)|}{|x-y|^{s+\theta}}.
\end{equation}
In addition, for any measurable set $E \subset \mathbb{R}^{2n}$, set
\begin{equation} \label{mudef}
\mu(E):= \int_{E} \frac{dxdy}{|x-y|^{n-2\theta}}.
\end{equation}
For any domain $\Omega \subset \mathbb{R}^n$, we then clearly have $u \in W^{s,2}(\Omega)$ if and only if $u \in L^2(\Omega)$ and $U \in L^2(\Omega \times \Omega,\mu)$,
so that in some sense the function $U$ and the measure $\mu$ are in duality.
Regarding larger exponents, by a simple computation, for any $p > 2$ and $\widetilde s:=s+\theta \left (1-\frac{2}{p} \right )>s$ we have 
\begin{equation} \label{muequiv}
u \in W^{\widetilde s,p}(\Omega) \quad \text{if and only if} \quad u \in L^p(\Omega) \text{ and } U \in L^p(\Omega \times \Omega,\mu).
\end{equation}
Therefore, a key feature of this approach to fractional-type gradients is that by proving higher integrability of the gradient-type function $U$ with respect to the measure $\mu$, we do not only gain regularity along the integrability scale of fractional Sobolev spaces, but also a substantial amount of higher differentiability! In \cite{MeV}, this property of such dual pairs of the type $(U,\mu)$ was then exploited by proving that in the restricted range $0<\theta < \min \{s,1-s\}$, we have $U \in L^p_{loc}(\Omega \times \Omega,\mu)$, which in turn then also gives some higher differentiability as indicated above. However, the amount of differentiability gained in this fashion is in general strictly smaller than the amount we gain in our main results, since a small amount of integrability gain also only yields a small gain along the differentiability scale. In the present paper, we overcome this issue by considering also fractional gradients and dual pairs of \emph{higher order}. More precisely, the key idea is to iteratively replace the function $U$ by fractional gradient-type functions of the type
\begin{equation} \label{Udefa}
	U_{\alpha}(x,y):=\frac{|u(x)-u(y)|}{|x-y|^{\alpha+\theta_\alpha}}
\end{equation}
and the above measure $\mu$ by measures of the form
\begin{equation} \label{mudefa}
	\mu_{\alpha}(E):= \int_{E} \frac{dxdy}{|x-y|^{n-2\theta_\alpha}},
\end{equation}
where $s \leq \alpha<\min\{2s,1\}$ and $\theta_\alpha := s+\theta-\alpha$. In a similar way as above, for any $p > 2$ and $\widetilde \alpha:=\alpha+\theta \left (1-\frac{2}{p} \right )>\alpha$, we have 
\begin{equation} \label{muequiv2}
	u \in W^{\widetilde \alpha,p}(\Omega) \quad \text{if and only if} \quad u \in L^p(\Omega) \text{ and } U_\alpha \in L^p(\Omega \times \Omega,\mu_\alpha).
\end{equation}
With these notions in place, let us now briefly sketch the further approach and in particular the iteration argument that leads to achieving $W^{t,p}_{loc}$ regularity for any $s<t<\min\{2s,1\}$. \par
First, we observe that instead of directly proving the desired regularity for nonlocal equations of the type $L_A u=f$, for technical reasons it is more appropriate for us to first focus on proving regularity for equations of the type $L_A u=(-\Delta)^s g$, where $(-\Delta)^s$ denotes the fractional Laplacian. This is because once we are able to transfer a sufficient amount of regularity from $g$ to $u$, in view of the known $H^{2s,p}$ estimates for the fractional Laplacian, we can then first transfer regularity from $f$ to some solution $g$ of $(-\Delta)^s g=f$ and then from $g$ to weak solutions $u$ of (\ref{nonlocaleq}).
Thus, we focus on proving that for weak solutions to $L_A u = (-\Delta)^s g$, for any $s<t<\min\{2s,1\}$ we have the implication
\begin{equation} \label{imp}
g \in W^{t,p}_{loc} \implies u \in W^{t,p}_{loc}.
\end{equation}
Instead of proving this implication directly, roughly speaking we focus on proving implications of the type
\begin{equation} \label{imp1}
G_\alpha \in L^p_{loc}(\Omega \times \Omega,\mu_\alpha) \implies U_\alpha \in L^p_{loc}(\Omega \times \Omega,\mu_\alpha)
\end{equation}
for any $\alpha \in [s,\min\{2s,1\})$, where $G_\alpha$ is defined in the same way as $U_\alpha$ with $u$ replaced by $g$. Since $\theta_\alpha$ decreases as $\alpha$ increases, this exactly leads to the implication (\ref{imp}) for any $s<t<\min\{2s,1\}$. \par
In order to prove the implication (\ref{imp1}), we make use of a covering argument implemented in detail in \cite{MeV}. The main idea is to cover the level sets $\left \{\mathcal{M}(U_\alpha^2) > \lambda^2 \right \}$ of the maximal function of $U_\alpha$ by dyadic cubes in order to show that these level sets decay sufficiently fast with respect to $\mu_\alpha$, which in view of standard measure-theoretic arguments then implies the desired implication (\ref{imp1}). However, since 
the above level sets are subsets of $\mathbb{R}^{2n}$ instead of $\mathbb{R}^{n}$, in our setup we have to run an exit time argument in $\mathbb{R}^{2n}$ instead of $\mathbb{R}^{n}$ in order to cover the level set of $U$ by Calder\'on-Zygmund cubes in $\mathbb{R}^{2n}$, which leads to rather severe technical difficulties. In particular, since close to the diagonal the information given by the equation can be used much more efficiently, an additional cover of the diagonal in terms of balls is constructed. However, since a large part of this technical covering argument works almost in exactly the same way as the one applied in \cite{MeV}, as indicated before, in this paper we primarily focus on the nontrivial modifications necessary in order to prove the implication (\ref{imp1}) in the higher-order case when $\alpha>s$. \par 
Namely, probably the most crucial complication in contrast to \cite{MeV} arises in the arguments applied in order to control the measures of the balls in the mentioned additional diagonal cover. In \cite{MeV}, the central tool in order to achieve this is given by a comparison estimate. More precisely, in \cite{MeV} the function $U$ was locally approximated in $L^2(\mu)$ by a corresponding function $V$, which is given as in (\ref{Udef}) with $u$ replaced by a weak solution $v$ of the corresponding homogeneous equation $L_{A_0} v=0$ with locally "frozen" coefficient $A_0$. Equivalently, it was proved that the difference $w:=u-v$ is small in $W^{s,2}$ whenever $g$ is small in $W^{s,2}$, which can be shown by testing the equation with $w$ itself. The mentioned covering argument then essentially allows to transfer regularity from $v$ to $u$. More precisely, in \cite{MeH} it was shown that such weak solutions $v$ to homogeneous equations with locally constant coefficients belong to $C^\beta$ for any $0<\beta < \min \{2s,1\}$, which suffices in order to transfer enough regularity from $v$ to $u$ in order to obtain the desired result. \par 
In contrast, proving such a comparison estimate for higher-order fractional gradients is more involved, since in this case the order of the gradient-type function no longer matches the order of the equation already in $L^2$. We resolve this issue as follows. In order to show that $U_\alpha$ is close to $V_\alpha$ in $L^2(\mu_\alpha)$ or equivalently, that $w=u-v$ is small in $W^{\alpha,2}$ whenever $g$ is small in $W^{\alpha,2}$, roughly speaking we additionally assume that $w$ satisfies an estimate of the form 
\begin{equation} \label{Wa2}
	[w]_{W^{\alpha,2}} \lesssim [w]_{W^{s,2}} + [g]_{W^{\alpha,m}} + \text{ tail terms}
\end{equation}
for some $m>2$.
This additional estimate then essentially allows to reduce the problem of proving the smallness of $w$ in $W^{\alpha,2}$ to showing the smallness of $w$ in $W^{s,2}$, which was already done in \cite{MeV}. In addition, while in \cite{MeV} it was necessary to locally freeze the coefficient $A$, since in view of the Sobolev embedding the main results in \cite{MeV} already imply a $C^\beta$ estimate for any $0<\beta < \min \{2s,1\}$ in the case when $A$ is merely VMO, in our situation freezing the coefficient is no longer necessary. However, in order to arrive at our main results, it then still remains to remove the assumption that the estimate (\ref{Wa2}) holds. \par 
We achieve this as follows. Since in the case when $\alpha=s$ the estimate (\ref{Wa2}) holds trivially, in this case (which corresponds to \cite{MeV}) we already achieve some higher differentiability or more precisely, we obtain that the implication (\ref{imp}) holds for some small enough $t_1>s$. But since due to the linearity of the equation, $w=u-v$ also satisfies the equation $L_A w = (-\Delta)^s g$, the estimate (\ref{Wa2}) is therefore also satisfied for $\alpha=t_1$. Thus, through the procedure we sketched above, we obtain that $u$ and $w$ satisfy the implication (\ref{imp1}) for $\alpha=t_1$, leading to the estimate (\ref{imp}) for some $t_2>t_1$, exceeding the amount of differentiability obtained in \cite{MeV}. Iterating this procedure finitely many times then indeed leads to the estimate (\ref{imp}) in the full range $s<t < \min \{2s,1\}$.

\subsection{Brief outline of the paper}
The paper is organized as follows. In section \ref{fracSob}, we define the fractional Sobolev spaces $W^{s,p}$ and mention some of their properties that we use throughout the paper. In section \ref{dual}, we then further discuss the notion of fractional gradients given by dual pairs introduced in the previous section \ref{app}. \par 
The rest of the paper is then devoted to the proof of our main results. In section \ref{ce} we implement the approximation argument for higher-order fractional gradients mentioned in section \ref{app}. In section \ref{gl}, we turn to proving certain good-$\lambda$ inequalities, both at the diagonal and far away from the diagonal. These good-$\lambda$ inequalities then allow to carry out an adaptation of the covering argument from \cite[Section 7]{MeV} for higher-order fractional gradients. Since the covering argument needed in our setting follows very closely the steps in \cite[Section 7]{MeV}, in section \ref{cover} we only explain the required adaptations in order to arrive at the desired level set estimate. In section \ref{ape}, this level set estimate is then used along with some delicate iteration arguments in order to prove a priori estimates for weak solutions. Finally, in section \ref{pmr} these a priori estimates are then combined with smoothing techniques in order to arrive at our main results.

\subsection{Some notation}
For convenience, let us fix some notation which we use throughout the paper. By $C,c$ and $C_i,c_i$, $i \in \mathbb{N}_0$, we always denote positive constants, while dependences on parameters of the constants will be shown in parentheses. As usual, by
$$ B_r(x_0):= \{x \in \mathbb{R}^n \mid |x-x_0|<r \}$$
we denote the open euclidean ball with center $x_0 \in \mathbb{R}^n$ and radius $r>0$. We also set $B_r:=B_r(0)$. In addition, by 
$$ Q_r(x_0):= \{x \in \mathbb{R}^n \mid |x-x_0|_\infty <r/2\}$$
we denote the open cube with center $x_0 \in \mathbb{R}^n$ and sidelength $r>0$.
Moreover, if $E \subset \mathbb{R}^n$ is measurable, then by $|E|$ we denote the $n$-dimensional Lebesgue-measure of $E$. If $0<|E|<\infty$, then for any $u \in L^1(E)$ we define
$$ \overline u_{E}:= \dashint_{E} u(x)dx := \frac{1}{|E|} \int_{E} u(x)dx.$$
As indicated in section \ref{app}, throughout this paper, we often consider integrals and functions on $\mathbb{R}^{2n}=\mathbb{R}^{n} \times \mathbb{R}^{n}$. Instead of dealing with the usual euclidean balls in $\mathbb{R}^{2n}$, for this purpose it is more convenient for us to use the balls generated by the norm
$$ ||(x_0,y_0)||:=\max\{|x_0|,|y_0|\}, \quad (x_0,y_0) \in \mathbb{R}^{2n}.$$
These balls with center $(x_0,y_0) \in \mathbb{R}^{2n}$ and radius $r>0$ are denoted by $\mathcal{B}_r(x_0,y_0)$ and are of the form
$$\mathcal{B}_r(x_0,y_0)=B_r(x_0) \times B_r(y_0).$$
In the case when $x_0=y_0$ we also write $\mathcal{B}_r(x_0):=\mathcal{B}_r(x_0,x_0),$
we call such balls diagonal balls. We also set $\mathcal{B}_r:=\mathcal{B}_r(0)$. Similarly, for $x_0,y_0 \in \mathbb{R}^n$ and $r>0$ we define $\mathcal{Q}_r(x_0,y_0):=Q_r(x_0) \times Q_r(y_0)$ and $\mathcal{Q}_r(x_0):=\mathcal{Q}_r(x_0,x_0)$ and also $\mathcal{Q}_r:=\mathcal{Q}_r(0)$.
\section{Fractional Sobolev spaces} \label{fracSob}
\begin{defin}
	Let $\Omega \subset \mathbb{R}^n$ be a domain. For $p \in [1,\infty)$ and $s \in (0,1)$, we define the fractional Sobolev space
	$$W^{s,p}(\Omega):=\left \{u \in L^p(\Omega) \mathrel{\Big|} \int_{\Omega} \int_{\Omega} \frac{|u(x)-u(y)|^p}{|x-y|^{n+sp}}dydx<\infty \right \}$$
	with norm
	$$ ||u||_{W^{s,p}(\Omega)} := \left (||u||_{L^p(\Omega)}^p + [u]_{W^{s,p}(\Omega)}^p \right )^{1/p} ,$$
	where
	$$ [u]_{W^{s,p}(\Omega)}:=\left (\int_{\Omega} \int_{\Omega} \frac{|u(x)-u(y)|^p}{|x-y|^{n+sp}}dydx \right )^{1/p} .$$
	In addition, we define the corresponding local fractional Sobolev spaces by
	$$ W^{s,p}_{loc}(\Omega):= \left \{ u \in L^p_{loc}(\Omega) \mid u \in W^{s,p}(\Omega^\prime) \text{ for any domain } \Omega^\prime \Subset \Omega \right \}.$$
	Also, we define the space 
	$$W^{s,p}_0(\Omega):= \left \{u \in W^{s,2}(\mathbb{R}^n) \mid u = 0 \text{ in } \mathbb{R}^n \setminus \Omega \right \}.$$
\end{defin}

We use the following fractional Poincar\'e inequality, see \cite[Section 4]{Mingione}.
\begin{lem} \label{Poincare} (fractional Poincar\'e inequality)
	Let $s \in (0,1)$, $p \in [1,\infty)$, $r>0$ and $x_0 \in \mathbb{R}^n$. For any $u \in L^p(B_r(x_0))$, we have
	$$ \int_{B_r(x_0)} \left | u(x)- \overline u_{B_r(x_0)} \right |^p dx \leq C r^{sp} \int_{B_r(x_0)} \int_{B_r(x_0)} \frac{|u(x)-u(y)|^p}{|x-y|^{n+sp}}dydx,$$
	where $C=C(s,p)>0$.
\end{lem}

\begin{prop} \label{Sobemb}
Let $\Omega \subset \mathbb{R}^n$ be a Lipschitz domain, $s \in (0,1)$ and $p \in [1,\infty)$.
\begin{itemize}
	\item If $sp<n$, then we have the continuous embedding
	$$ W^{s,p}(\Omega) \hookrightarrow L^\frac{np}{n-sp}(\Omega).$$
	\item If $sp=n$, then for any $q \in [1,\infty)$ we have the continuous embedding
	$$ W^{s,p}(\Omega) \hookrightarrow L^q(\Omega).$$
	\item If $sp>n$, then we have the continuous embedding
	$$ W^{s,p}(\Omega) \hookrightarrow C^{s-\frac{n}{p}}(\Omega).$$
\end{itemize}
In addition, if $sp>n$ and $\Omega=B_r(x_0)$ for some $r>0$ and some $x_0 \in \mathbb{R}^n$, then for any $u \in W^{s,p}(B_r(x_0))$, we have 
\begin{equation} \label{hemb}
	[u]_{C^{s-\frac{n}{p}}(B_r(x_0))} \leq C [u]_{W^{s,p}(B_r(x_0))},
\end{equation}
where $C=C(n,s,p)>0$.
\end{prop}
\begin{proof}
The above three embeddings follow from \cite[Theorem 6.7, Theorem 6.10, Theorem 8.2]{Hitch}. Let us now prove (\ref{hemb}). Define $u_r(x):=u(rx+x_0)$. Applying the third of the above embeddings to $\widetilde u_r:= u_r-\overline {\left (u_r \right)}_{B_1} \in W^{s,p}(B_1)$ and then using the fractional Poincar\'e inequality (Lemma \ref{Poincare}), along with changes of variables leads to
\begin{align*}
	r^{s-\frac{n}{p}} [u]_{C^{s-\frac{n}{p}}(B_r(x_0))} = & [u_r]_{C^{s-\frac{n}{p}}(B_1)} \\
	= & [\widetilde u_r]_{C^{s-\frac{n}{p}}(B_1)} \\
	\leq & C_1 \left ( \int_{B_1} \int_{B_1} \frac{|\widetilde u_r(x)-\widetilde u_r(y)|^{p}}{|x-y|^{n+sp}}dydx + \int_{B_1} |\widetilde u_r(x)|^p dx \right )^\frac{1}{p} \\
	\leq & C \left (\int_{B_1} \int_{B_1} \frac{|u_r(x)-u_r(y)|^{p}}{|x-y|^{n+sp}}dydx \right )^\frac{1}{p} \\
	= & C r^{s-\frac{n}{p}} \left ( \int_{B_r(x_0)} \int_{B_r(x_0)} \frac{|u(x)-u(y)|^{p}}{|x-y|^{n+sp}}dydx \right )^\frac{1}{p},
\end{align*}
where $C_1$ and $C$ depend only on $n,s$ and $p$. Since the factor $r^{s-\frac{n}{p}}$ cancels out on both sides, the proof is finished.
\end{proof}

For the following Lemma, we refer to \cite[Lemma 2.4]{MeV}.
\begin{lem} \label{SobPoincare} (fractional Sobolev-Poincar\'e inequality)
Let $s \in (0,1)$, $p \in [1,\infty)$, $r>0$ and $x_0 \in \mathbb{R}^n$. In addition, let 
$$ q \in \begin{cases} 
\left [1,\frac{np}{n-sp} \right ], & \text{if } sp<n \\
[1,\infty), & \text{if } sp \geq n.
\end{cases}$$
Then for any $u \in W^{s,p}(B_r(x_0))$, we have
	$$ \left (\dashint_{B_r(x_0)} \left | u(x)- \overline u_{B_r(x_0)} \right |^q dx \right )^\frac{1}{q} \leq C r^{s} \left ( \dashint_{B_r(x_0)} \int_{B_r(x_0)} \frac{|u(x)-u(y)|^p}{|x-y|^{n+sp}}dydx \right )^\frac{1}{p},$$
	where $C=C(n,s,p,q)>0$.
\end{lem}

For $p \in (1,\infty)$ and $s \in (0,2)$, denote by $H^{s,p}(\Omega)$ the standard Bessel potential spaces on $\Omega$, see e.g.\ \cite[Section 2]{MeV}. The following embedding result follows from \cite[Theorem 2.5]{Tr}, where it is given in the more general context of Besov and Triebel-Lizorkin spaces.
\begin{prop} \label{BesselTr}
	Let $1 < p_0 < p < p_1 < \infty$, $s \in (0,2)$, $s_0,s_1 \in (0,1)$ and assume that $\Omega \subset \mathbb{R}^n$ is a smooth domain. If $ s_0 - \frac{n}{p_0} = s - \frac{n}{p} = s_1 - \frac{n}{p_1}, $ then
		$$ W^{s_0,p_0}(\Omega) \hookrightarrow H^{s,p}(\Omega) \hookrightarrow W^{s_1,p_1}(\Omega).$$
\end{prop}
Unlike the first-order Sobolev spaces $W^{1,p}(\Omega)$ on a bounded domain $\Omega \subset \mathbb{R}^n$, the fractional Sobolev spaces $W^{s,p}(\Omega)$ are not contained in each other as the integrability exponent $p$ decreases. Nevertheless, the following result essentially shows that the mentioned inclusions are almost true.
\begin{prop} \label{Sobcont}
Let $1< p_0 \leq p<\infty$, $s \in (0,1)$ and assume that $\Omega \subset \mathbb{R}^n$ is a smooth bounded domain. Then for any $t \in (s,1)$, we have
$$ W^{t,p}(\Omega) \hookrightarrow W^{s,p_0}(\Omega).$$
In addition, if $\Omega=B_r(x_0)$ for some $r>0$ and some $x_0 \in \mathbb{R}^n$, then for any $u \in W^{t,p}(B_r(x_0))$, we have 
\begin{equation} \label{diffemb}
[u]_{W^{s,p_0}(B_r(x_0))} \leq Cr^{\frac{n}{p_0}-\frac{n}{p}+t-s} [u]_{W^{t,p}(B_r(x_0))},
\end{equation}
where $C=C(n,s,t,p,p_0)>0$.
\end{prop}

\begin{proof}
By \cite[Proposition 2.6]{MeV}, for $0<\varepsilon <\min  \left \{t-s,\frac{2n}{p},2n \left (1-\frac{1}{p_0} \right ) \right \}$, we have $ W^{t,p}(\Omega) \hookrightarrow W^{t-\varepsilon,p_0}(\Omega).$
Since by \cite[Proposition 2.1]{Hitch}, we also have $ W^{t-\varepsilon,p_0}(\Omega) \hookrightarrow W^{s,p_0}(\Omega) ,$ we arrive at the embedding
$ W^{t,p}(\Omega) \hookrightarrow W^{s,p_0}(\Omega).$ \par 
In order to prove (\ref{diffemb}), set $u_r(x):=u(rx+x_0)$. By using the above embedding with respect to $\widetilde u_r:=u_r-\overline {\left (u_r \right)}_{B_1}$ and then the fractional Poincar\'e inequality (Lemma \ref{Poincare}), along with changing variables we conclude that
\begin{align*}
& \left (\int_{B_r(x_0)} \int_{B_r(x_0)} \frac{|u(x)-u(y)|^{p_0}}{|x-y|^{n+sp_0}}dydx \right )^\frac{1}{p_0} \\ = & r^{\frac{n}{p_0}-s} \left (\int_{B_1} \int_{B_1} \frac{|\widetilde u_r(x)-\widetilde u_r(y)|^{p_0}}{|x-y|^{n+sp_0}}dydx \right )^\frac{1}{p_0} \\
\leq & C_1 r^{\frac{n}{p_0}-s} \left (\int_{B_1} \int_{B_1} \frac{|\widetilde u_r(x)-\widetilde u_r(y)|^{p}}{|x-y|^{n+tp}}dydx + \int_{B_1} |\widetilde u_r(x)|^p dx \right )^\frac{1}{p}\\
\leq & C_2 r^{\frac{n}{p_0}-s} \left (\int_{B_1} \int_{B_1} \frac{|u_r(x)-u_r(y)|^{p}}{|x-y|^{n+tp}}dydx \right )^\frac{1}{p} \\
= & C_2 r^{\frac{n}{p_0}-\frac{n}{p}+t-s} \left (\int_{B_r(x_0)} \int_{B_r(x_0)} \frac{|u(x)-u(y)|^{p}}{|x-y|^{n+tp}}dydx \right )^\frac{1}{p},
\end{align*}
where $C_1$ and $C_2$ depend only on $n,p,p_0,s$ and $t$. This proves (\ref{diffemb}).
\end{proof}

\section{Fractional gradients on $\mathbb{R}^{2n}$} \label{dual}
\subsection{Basic properties of dual pairs} \label{dm}
Fix some $t \in (0,1)$ and some $\theta \in \left (0,\frac{1}{2} \right)$. We define a Borel measure $\mu_\theta$ on $\mathbb{R}^{2n}$ as follows. For any function $u:\mathbb{R}^n \to \mathbb{R}$ and $(x,y) \in \mathbb{R}^{2n}$ with $x \neq y$, we define the function 
\begin{equation} \label{Udef1}
	U_{t,\theta}(x,y):=\frac{|u(x)-u(y)|}{|x-y|^{t+\theta}}.
\end{equation}
For any measurable set $E \subset \mathbb{R}^{2n}$, set
\begin{equation} \label{mudef1}
	\mu_{\theta}(E):= \int_{E} \frac{dxdy}{|x-y|^{n-2\theta}}.
\end{equation}
The following Lemma follows by a straightforward computation, see \cite[Lemma 3.1]{MeV}.
\begin{lem} \label{Sobolevrelate}
Let $p \geq 2$ and set $t_\theta:=t+\theta \left (1-\frac{2}{p} \right )$. Then we have
$$ u \in W^{t_\theta,p}(\Omega) \quad \text{if and only if} \quad u \in L^p(\Omega) \text{ and } U_{t,\theta} \in L^p(\Omega \times \Omega,\mu_\theta)$$
and $$||U_{t,\theta}||_{L^p(\Omega \times \Omega,\mu_\theta)}=[u]_{W^{t_\theta,p}(\Omega)}.$$
\end{lem}

The next Proposition contains some further important properties of the measure $\mu_\theta$ which we use frequently throughout the paper, usually without explicit reference. For a proof, we refer to \cite[Proposition 3.2]{MeV}.
\begin{prop} \label{doublingmeasure}
\begin{enumerate} [label=(\roman*)]
\item For all $r>0$ and $x_0 \in \mathbb{R}^n$, we have
$$ \mu_\theta(\mathcal{B}_r(x_0))= \mu_\theta(\mathcal{B}_r) =cr^{n+2\theta},$$
where $c=c(n,\theta)>0$.
\item (volume doubling property) For any $(x_0,y_0) \in \mathbb{R}^{2n}$, any $r>0$ and any $M >0$, we have
$$ \mu_\theta(\mathcal{B}_{Mr}(x_0,y_0)) = M^{n+2 \theta} \mu_\theta(\mathcal{B}_{r}(x_0,y_0)).$$
\end{enumerate}
\end{prop}

We will also frequently use the following relation between fractional gradients of different order.

\begin{lem} \label{Urel}
	Let $0<s \leq t<1$, $p \geq 2$, $\theta \in \left (0,\frac{1}{2} \right)$ and set $\theta_{t} := s+\theta-t$. Then for any $r >0$, any $x_0 \in \mathbb{R}^n$ and any $u \in W^{t,p}(B_r(x_0))$, we have
	$$ \dashint_{\mathcal{B}_{r}(x_0)} U_{s,\theta}^p d\mu_{\theta} \leq C \dashint_{\mathcal{B}_{r}(x_0)} U_{t,\theta_{t}}^p d\mu_{\theta_{t}},$$
	where $C=C(n,s,t,\theta,p)>0$.
\end{lem}

\begin{proof}
	Using that $tp+\theta_t(p-2)-sp-\theta(p-2)=2(t-s)$, we have
	\begin{align*}
		\dashint_{\mathcal{B}_{r}(x_0)} U_{s,\theta}^p d\mu_{\theta} = & C_1 r^{-n-2\theta} \int_{B_{r}(x_0)} \int_{B_{r}(x_0)} \frac{|u(x)-u(y)|^{p}}{|x-y|^{n+sp+\theta(p-2)}}dydx \\
		= & C_1 r^{-n-2\theta} \int_{B_{r}(x_0)} \int_{B_{r}(x_0)} \frac{|u(x)-u(y)|^{p}}{|x-y|^{n+tp+\theta_t(p-2)}}|x-y|^{2(t-s)} dydx \\
		\leq & C_2 r^{-n-2\theta +2(t-s)} \int_{B_{r}(x_0)} \int_{B_{r}(x_0)} \frac{|u(x)-u(y)|^{p}}{|x-y|^{n+tp+\theta_t(p-2)}}dydx \\
		= & C_2 r^{-n-2\theta_t} \int_{B_{r}(x_0)} \int_{B_{r}(x_0)} \frac{|u(x)-u(y)|^{p}}{|x-y|^{n+tp+\theta_t(p-2)}}dydx \\
		= & C \dashint_{\mathcal{B}_{r}(x_0)} U_{t,\theta_t}^p d\mu_{\theta_t},
	\end{align*}
	where all constants depend only on $n,s,t,\theta$ and $p$.
\end{proof}

\subsection{The Hardy-Littlewood maximal function} \label{HL}

Another tool we use is the Hardy-Littlewood maximal function with respect to the measure $\mu_\theta$.

\begin{defin}
	Let $F \in L^1_{loc}(\mathbb{R}^{2n},\mu_\theta)$. We define the Hardy-Littlewood maximal function \newline $\mathcal{M} F: \mathbb{R}^{2n} \to [0,\infty]$ of $F$ with respect to $\mu_\theta$ by 
	$$ \mathcal{M} (F)(x,y) := \mathcal{M}_{\mu_\theta} (F)(x,y) := \sup_{\rho>0} \dashint_{\mathcal{B}_\rho(x,y)} |F|d\mu_\theta ,$$ where 
	$$ \dashint_{\mathcal{B}_\rho(x,y)} |F|d\mu_\theta := \frac{1}{\mu_\theta(\mathcal{B}_\rho(x,y))} \int_{\mathcal{B}_\rho(x,y)} |F|d\mu_\theta.$$
	Moreover, for any open set $E \subset \mathbb{R}^{2n}$, we define
	$$ \mathcal{M}_{E} (F) := \mathcal{M} \left(F \chi_{E} \right ), $$
	where $\chi_{E}$ is the characteristic function of $E$.
	In addition, for any $r>0$ we define $$ \mathcal{M}_{\geq r} (F)(x,y) := \sup_{\rho \geq r} \dashint_{\mathcal{B}_\rho(x,y)} |F|d\mu_\theta .$$ and $$ \mathcal{M}_{\geq r,E} (F) := \mathcal{M}_{\geq r} \left(F \chi_{E} \right ). $$
\end{defin}

The following result shows that the Hardy-Littlewood maximal function is well-behaved in the context of $L^p$ spaces. Since in view of Proposition \ref{doublingmeasure} $\mu_\theta$ is a doubling measure with doubling constant $2^{n+2\theta}$, the result follows directly from \cite[Chapter 1, Section 3, Theorem 1]{Stein}.

\begin{prop} \label{Maxfun} 
	Let $E$ be an open subset of $\mathbb{R}^{2n}$.
	\begin{enumerate}[label=(\roman*)]
		\item (weak p-p estimates) If $F \in L^p(E,\mu_\theta)$ for some $p \geq 1$ and $\lambda>0$, then 
		$$
		\mu_\theta \left ( \{x \in E \mid \mathcal{M}_E(F)(x) > \lambda \} \right ) \leq \frac{C}{\lambda^p} \int_{E} |F|^p d\mu_\theta, 
		$$
		where $C$ depends only on $n,\theta$ and $p$.
		\item (strong p-p estimates) If $F \in L^p(E,\mu_\theta)$ for some $p \in (1,\infty]$, then 
		$$
		||\mathcal{M}_E (F)||_{L^p(E,\mu_\theta)} \leq C ||F||_{L^p(E,\mu_\theta)}, 
		$$
		where $C$ depends only on $n$, $\theta$ and $p$.
	\end{enumerate}
\end{prop}

The following result is a direct consequence of the Lebesgue differentiation theorem with respect to $\mu_\theta$, see \cite[Corollary 3.5]{MeV}.
\begin{prop} \label{maxdominate}
	Let $F \in L_{loc}^1(\mathbb{R}^{2n},\mu_\theta)$. Then for almost every $(x,y) \in \mathbb{R}^{2n}$, we have 
	$$
	|F(x,y)| \leq \mathcal{M}(F)(x,y). 
	$$
	In addition, for any open set $E \subset \mathbb{R}^{2n}$ and any $p \in [1,\infty]$, we have 
	$$
	||F||_{L^p(E,\mu_\theta)} \leq ||\mathcal{M}_E (F)||_{L^p(E,\mu_\theta)}. 
	$$
\end{prop}

\section{An approximation argument} \label{ce}
From now on, we fix some $s \in (0,1)$ and some parameter 
\begin{equation} \label{theta}
\theta \in (0,\min\{s,1-s\})
\end{equation} to be chosen later. In addition, for the fractional gradients $U_{s,\theta}$, $V_{s,\theta}$ and $G_{s,\theta}$ of functions $u,v,g:\mathbb{R}^n \to \mathbb{R}$ and the measure $\mu_\theta$, we are going to use the abbreviated notation $$U:=U_{s,\theta}, \quad V:=V_{s,\theta}, \quad G:=G_{s,\theta}, \quad \mu:=\mu_\theta.$$
\begin{defin}
	Given $g \in W^{s,2}(\mathbb{R}^n)$, we say that $u \in W^{s,2}(\mathbb{R}^n)$ is a weak solution of the equation $L_A u = (-\Delta)^s g$ in $\Omega$, if 
	\begin{align*}
		& \int_{\mathbb{R}^n} \int_{\mathbb{R}^n} \frac{A(x,y)}{|x-y|^{n+2s}} (u(x)-u(y))(\varphi(x)-\varphi(y))dydx \\
		= & C_{n,s} \int_{\mathbb{R}^n} \int_{\mathbb{R}^n} \frac{g(x)-g(y)}{|x-y|^{n+2s}} (\varphi(x)-\varphi(y))dydx \quad \forall \varphi \in W^{s,2}_0(\Omega).
	\end{align*}
\end{defin}
In addition, we also need the following definition.
\begin{defin}
	Let $\Omega$ be a domain and consider functions $h \in W^{s,2}(\mathbb{R}^n)$ and $f \in L^\frac{2n}{n+2s}(\Omega)$. We say that $v \in W^{s,2}(\mathbb{R}^n)$ is a weak solution of the Dirichlet problem $$ \begin{cases} \normalfont
		L_{A} v = f & \text{ in } \Omega \\
		v = h & \text{ a.e. in } \mathbb{R}^n \setminus \Omega,
	\end{cases} $$ if we have $v = h \text{ a.e. in } \mathbb{R}^n \setminus \Omega$ and
	$$
	\int_{\mathbb{R}^n} \int_{\mathbb{R}^n} \frac{A(x,y)}{|x-y|^{n+2s}} (u(x)-u(y))(\varphi(x)-\varphi(y))dydx = \int_{\Omega} f \varphi dx \quad \forall \varphi \in W^{s,2}_0(\Omega).
	$$
\end{defin}

The following comparison estimate follows from \cite[Proposition 5.1]{MeV} by taking $A=\widetilde A$ and $f=\widetilde f =0$.
\begin{prop} \label{appplxy}
	Let $x_0 \in \mathbb{R}^n$, $r>0$, $g \in W^{s,2}(\mathbb{R}^n)$ and $A \in \mathcal{L}_0(\Lambda)$. Moreover, let $u \in W^{s,2}(\mathbb{R}^n)$ be a weak solution of the equation
	\begin{equation} \label{neq5}
		L_A u = (-\Delta)^s g \text{ in } B_{2r}(x_0),
	\end{equation}
	and let $v \in W^{s,2}(\mathbb{R}^n)$ be the unique weak solution of the Dirichlet problem
	\begin{equation} \label{constcof3}
		\begin{cases} \normalfont
			L_A v = 0 & \text{ in } B_{2r}(x_0) \\
			v = u & \text{ a.e. in } \mathbb{R}^n \setminus B_{2r}(x_0).
		\end{cases}
	\end{equation}
	Then the function $w:=u-v \in W^{s,2}_0(B_{2r}(x_0))$
	satisfies 
	\begin{align*}
		& \int_{\mathbb{R}^n} \int_{\mathbb{R}^n} \frac{(w(x)-w(y))^2}{|x-y|^{n+2s}}dydx \\
		\leq & C \mu(\mathcal{B}_r(x_0)) \left (\sum_{k=1}^\infty 2^{-k(s-\theta)} \left ( \dashint_{\mathcal{B}_{2^kr}(x_0)} G^2 d\mu \right )^\frac{1}{2}\right )^2 ,
	\end{align*}
	where $C=C(n,s,\theta,\Lambda)>0$.
\end{prop}

We continue by fixing some further notation and some assumptions which we will use throughout the rest of this paper. From now on, we fix some $\Lambda \geq 1$, some $\delta>0$ to be chosen small enough, some coefficient $A \in \mathcal{L}_0(\Lambda)$ that is $\delta$-vanishing in $B_{5n}$ and some $p \in (2,\infty)$. Moreover, we fix another number $q \in [2,p)$ and define
\begin{equation} \label{qstar}
q_\alpha^\star:=\begin{cases} 
\frac{nq}{n-\alpha q}, & \text{if } n>\alpha q \\
2p, & \text{if } n \leq \alpha q.
\end{cases}
\end{equation}
In addition, we fix a number $m$ in the range
\begin{equation} \label{mdef}
2 \leq m < \min \left \{\frac{2(n-s)}{n-2s},p \right \}
\end{equation}
and define 
\begin{equation} \label{q0}
q_0:=\max \{m,q\}.
\end{equation}
Furthermore, we fix some function $g \in W^{s,2}(\mathbb{R}^n)$ and a weak solution $u \in W^{s,2}(\mathbb{R}^n)$ 
of the equation 
\begin{equation} \label{helpeqx}
L_A u = (-\Delta)^s g \text { in } B_{5n}
\end{equation}
and set
\begin{equation} \label{tailcontrol}
\begin{aligned}  
\lambda_0:= M_0 & \Bigg (\sum_{k=1}^\infty 2^{-k(s-\theta)} \left (\dashint_{\mathcal{B}_{2^k 5n}} U^2 d\mu \right )^\frac{1}{2} + \left (\dashint_{\mathcal{B}_{5n}} U_\alpha^{2} d\mu_\alpha \right )^\frac{1}{2} \\ & + \delta^{-1} \sum_{k=1}^\infty 2^{-k(s-\theta)} \left (\dashint_{\mathcal{B}_{2^k 5n}} G^2 d\mu \right )^\frac{1}{2} +\left (\dashint_{\mathcal{B}_{5n}} G_\alpha^{q_0} d\mu_\alpha \right )^\frac{1}{q_0} \Bigg ),
\end{aligned}
\end{equation}
where $M_0 \geq 1 $ remains to be chosen large enough.
From now on, we also fix some number $$\alpha \in [s,\min\{2s,1\})$$ and assuming that $\theta>\alpha-s$, we define a corresponding parameter by
\begin{equation} \label{modtheta}
\theta_\alpha := s+\theta-\alpha>0
\end{equation} 
with associated gradient-type functions
$$U_\alpha(x,y):=U_{\alpha,\theta_\alpha}(x,y)=\frac{|u(x)-u(y)|}{|x-y|^{\alpha+\theta_\alpha}}, \quad G_\alpha(x,y):=G_{\alpha,\theta_\alpha}(x,y)=\frac{|g(x)-g(y)|}{|x-y|^{\alpha+\theta_\alpha}} $$ and with associated measure
$$ \mu_\alpha(E):=\mu_{\theta_\alpha}(E)= \int_{E} \frac{dxdy}{|x-y|^{n-2\theta_\alpha}}, \quad E \subset \mathbb{R}^{2n} \text{ measurable}.$$ In addition, from this point on we assume that for any $x_0 \in \mathbb{R}^n$, $r>0$ such that $B_r(x_0) \subset B_{5n}$, and any weak solution $u_0 \in W^{s,2}(\mathbb{R}^n)$ of $L_A u_0=(-\Delta)^s g$ in $B_{r}(x_0)$, we have a higher differentiability estimate of the form
\begin{equation} \label{hdest}
\begin{aligned}
& \left (\frac{1}{\mu_\alpha(\mathcal{B}_r(x_0))}\int_{B_{r/2}(x_0)} \int_{B_{r/2}(x_0)} \frac{(u_0(x)-u_0(y))^2}{|x-y|^{n+2\alpha}}dydx \right )^\frac{1}{2} \\ \leq & C \Bigg ( \sum_{k=1}^\infty 2^{-k(s-\theta)} \left ( \dashint_{\mathcal{B}_{2^kr}(x_0)} U_0^2 d\mu \right )^\frac{1}{2} + \left ( \dashint_{\mathcal{B}_{r}(x_0)} G_\alpha^{m} d\mu_\alpha \right )^{\frac{1}{m}} \\ & + \sum_{k=1}^\infty 2^{-k(s-\theta)} \left ( \dashint_{\mathcal{B}_{2^kr}(x_0)} G^2 d\mu \right )^\frac{1}{2} \Bigg ),
\end{aligned}
\end{equation}
where $C=C(n,s,\alpha,\theta,\Lambda,m)>0$ and 
$$ U_0(x,y):=\frac{|u_0(x)-u_0(y)|}{|x-y|^{s+\theta}}.$$

\begin{lem} \label{apppl}
	Let $M>0$, $x_0 \in B_{\frac{\sqrt{n}}{2}}$, $r \in \left (0,\frac{\sqrt{n}}{2} \right )$ and $\lambda \geq \lambda_0$.
	Then for any $\varepsilon_0 >0$, there exists some small enough $\delta = \delta (\varepsilon_0,n,s,\alpha,\theta,\Lambda,m,M) \in (0,1)$, such that
	under the assumptions that
\begin{equation} \label{condU}
 {\mathcal{M}}_{\mathcal{B}_{5n}} (U_\alpha^2)(x_0) \leq  M\lambda^2, \quad  {\mathcal{M}}_{\mathcal{B}_{5n}} (G_\alpha^{q_0})(x_0) \leq M\lambda^{q_0} \delta^{q_0} ,
\end{equation}
	for the unique weak solution $v \in W^{s,2}(\mathbb{R}^n)$ of the Dirichlet problem
\begin{equation} \label{constcof3x}
\begin{cases} \normalfont
L_{A} v = 0 & \text{ in } B_{6r}(x_0) \\
v = u & \text{ a.e. in } \mathbb{R}^n \setminus B_{6r}(x_0)
\end{cases}
\end{equation}
	and the function 
	\begin{equation} \label{W}
	W_\alpha(x,y):=\frac{|u(x)-v(x)-u(y)+v(y)|}{|x-y|^{\alpha+\theta_\alpha}}, \quad (x,y) \in \mathbb{R}^{2n},
	\end{equation}
	we have
		\begin{equation} \label{west}
		\int_{\mathcal{B}_{2r}(x_0)} W_\alpha^2 d \mu_\alpha \leq \varepsilon^2 \lambda^2 \mu_\alpha(\mathcal{B}_r(x_0)) .
		\end{equation}
		\newline Moreover, the function
		$$V_\alpha(x,y):=\frac{|v(x)-v(y)|}{|x-y|^{\alpha+\theta_\alpha}}, \quad (x,y) \in \mathbb{R}^{2n}$$
		satisfies the estimate
		\begin{equation} \label{Vest}
		||V_\alpha||_{L^\infty(\mathcal{B}_{2r}(x_0),d\mu_\alpha)} \leq N_0 \lambda
		\end{equation}
		for some constant $N_0= N_0(n,s,\alpha,\theta,\Lambda,M)>0$.
\end{lem}

\begin{rem} \normalfont
	In the above Lemma and in the rest of this paper, the Hardy-Littlewood maximal function is always considered with respect to the measure $\mu_\alpha$.
\end{rem}

\begin{proof}
	Fix $x_0 \in B_{\frac{\sqrt{n}}{2}}$ and $r \in \left (0,\frac{\sqrt{n}}{2} \right )$.
	Let $l \in \mathbb{N}$ be determined by $2^{l-1} r < \sqrt{n} \leq 2^{l} r$, note that $l \geq 2$. Then for any $k < l$, by (\ref{condU}) we have
	\begin{equation} \label{UG}
		\dashint_{\mathcal{B}_{2^k4r}(x_0)} U_\alpha^2 d\mu_\alpha \leq M\lambda^2, \quad \dashint_{\mathcal{B}_{2^k3 r}(x_0)} G_\alpha^{q_0} d\mu_\alpha \leq M\lambda^{q_0} \delta^{q_0} .
	\end{equation}
	On the other hand, in view of (\ref{tailcontrol}) and the inclusions $$\mathcal{B}_{2^k \sqrt{n}}(x_0) \subset \mathcal{B}_{2^{k+l-1}4r}(x_0) \subset \mathcal{B}_{2^k 4\sqrt{n}}(x_0) \subset \mathcal{B}_{2^k 5n},$$ we have
	\begin{equation} \label{tailst}
	\begin{aligned}
		\sum_{k=l}^\infty 2^{-k(s-\theta)} \left ( \dashint_{\mathcal{B}_{2^k 4r}(x_0)} U^2 d\mu \right )^\frac{1}{2} = & 2^{-(l-1)(s-\theta)} \sum_{k=1}^\infty 2^{-k(s-\theta)} \left ( \dashint_{\mathcal{B}_{2^{k+l-1}4r}(x_0)} U^2 d\mu \right )^\frac{1}{2} \\
		\leq & \sum_{k=1}^\infty 2^{-k(s-\theta)} \left (\frac{\mu(\mathcal{B}_{2^k 5n})}{\mu \left (\mathcal{B}_{2^k \sqrt{n}} \right)} \dashint_{\mathcal{B}_{2^k 5n}} U^2 d\mu \right )^\frac{1}{2} \\
		= & C_1 \sum_{k=1}^\infty 2^{-k(s-\theta)} \left ( \dashint_{\mathcal{B}_{2^k 5n}} U^2 d\mu \right )^\frac{1}{2} \leq C_1 \lambda_0,
	\end{aligned}
	\end{equation}
	where $C_1=C_1(n,\theta)>0$.
	Moreover, by Lemma \ref{Urel} for any $k \leq l-1$ we have
	\begin{align*}
	\dashint_{\mathcal{B}_{2^k4r}(x_0)} U^2 d\mu
	\leq C_2 \dashint_{\mathcal{B}_{2^k4r}(x_0)} U_\alpha^2 d\mu_\alpha,
	\end{align*}
	where $C_2=C_2(n,s,\alpha,\theta)$. 
	Now combining the previous display with (\ref{tailst}), (\ref{UG}) and the facts that $\theta < s$ and $\lambda \geq \lambda_0$, we arrive at
	\begin{equation} \label{Ulambda}
		\begin{aligned}
			& \sum_{k=1}^\infty 2^{-k(s-\theta)} \left ( \dashint_{\mathcal{B}_{2^k4r}(x_0)} U^2 d\mu \right )^\frac{1}{2} \\ \leq & \sum_{k=1}^{l-1} 2^{-k(s-\theta)} \left ( \dashint_{\mathcal{B}_{2^k4r}(x_0)} U^2 d\mu \right )^\frac{1}{2} + \sum_{k=l}^\infty 2^{-k(s-\theta)} \left ( \dashint_{\mathcal{B}_{2^k4r}(x_0)} U^2 d\mu \right )^\frac{1}{2} \\
			\leq & C_2^\frac{1}{2} \sum_{k=1}^{l-1} 2^{-k(s-\theta)} \left ( \dashint_{\mathcal{B}_{2^k4r}(x_0)} U_\alpha^2 d\mu_\alpha \right )^\frac{1}{2} + \sum_{k=l}^\infty 2^{-k(s-\theta)} \left ( \dashint_{\mathcal{B}_{2^k4r}(x_0)} U^2 d\mu \right )^\frac{1}{2} \\
			\leq & C_2^\frac{1}{2} M^\frac{1}{2} \lambda \sum_{k=1}^\infty 2^{-k(s-\theta)}+ C_1\lambda_0 \leq C_3 \lambda, 
		\end{aligned}
	\end{equation}
	where $C_3=C_3(n,s,\alpha,\theta,M)>0$.
	In a similar way as in (\ref{tailst}), we have
	\begin{align*}
		\sum_{k=l}^\infty 2^{-k(s-\theta)} \left ( \dashint_{\mathcal{B}_{2^k3r}(x_0)} G^2 d\mu \right )^\frac{1}{2} 
		\leq C_1 \sum_{k=1}^\infty 2^{-k(s-\theta)} \left ( \dashint_{\mathcal{B}_{2^k 5n}} G^2 d\mu \right )^\frac{1}{2} \leq C_1 \lambda_0 \delta.
	\end{align*}
	Therefore, using Lemma \ref{Urel} along with H\"older's inequality, we obtain
	\begin{equation} \label{Glambda}
		\begin{aligned}
			& \sum_{k=1}^\infty 2^{-k(s-\theta)} \left ( \dashint_{\mathcal{B}_{2^k3r}(x_0)} G^2 d\mu \right )^\frac{1}{2} \\ \leq & C_2^\frac{1}{2} \sum_{k=1}^{l-1} 2^{-k(s-\theta)} \left ( \dashint_{\mathcal{B}_{2^k3r}(x_0)} G_\alpha^2 d\mu_\alpha \right )^\frac{1}{2} + \sum_{k=l}^\infty 2^{-k(s-\theta)} \left ( \dashint_{\mathcal{B}_{2^k3r}(x_0)} G^2 d\mu \right )^\frac{1}{2} \\
			\leq & C_2^\frac{1}{2} \sum_{k=1}^{l-1} 2^{-k(s-\theta)} \left ( \dashint_{\mathcal{B}_{2^k3r}(x_0)} G_\alpha^{q_0} d\mu_\alpha \right )^\frac{1}{q_0} + \sum_{k=l}^\infty 2^{-k(s-\theta)} \left ( \dashint_{\mathcal{B}_{2^k3r}(x_0)} G^2 d\mu \right )^\frac{1}{2} \\
			\leq & C_2^\frac{1}{2} M^\frac{1}{q_0} \lambda \delta \sum_{k=1}^\infty 2^{-k(s-\theta)}+ C_1\lambda_0\delta \leq C_3 \lambda\delta.
		\end{aligned}
	\end{equation}
Since $w:=u-v$ is a weak solution of $L_A w = (-\Delta)^s g$ in $B_{6r}(x_0)$, $w$ satisfies the estimate (\ref{hdest}), which combined with Proposition \ref{appplxy}, H\"older's inequality, (\ref{UG}) and (\ref{Glambda}) yields
		\begin{align*}
			\int_{\mathcal{B}_{2r}(x_0)} W_\alpha^2 d \mu_\alpha \leq & \int_{B_{3r}} \int_{B_{3r}} \frac{(w(x)-w(y))^2}{|x-y|^{n+2\alpha}}dydx
			\\ \leq & C_4 \frac{\mu_\alpha(\mathcal{B}_r(x_0))}{\mu(\mathcal{B}_r(x_0))}\int_{\mathbb{R}^n} \int_{\mathbb{R}^n} \frac{(w(x)-w(y))^2}{|x-y|^{n+2s}}dydx
			\\	& + C_4 \mu_\alpha(\mathcal{B}_r(x_0)) \left ( \dashint_{\mathcal{B}_{6r}(x_0)} G_\alpha^{m} d\mu_\alpha \right )^{\frac{2}{m}} \\ & + C_4 \mu_\alpha(\mathcal{B}_r(x_0)) \left (\sum_{k=1}^\infty 2^{-k(s-\theta)} \left ( \dashint_{\mathcal{B}_{2^k3r}(x_0)} G^2 d\mu \right )^\frac{1}{2}\right )^2 \\
			\leq & C_5 \mu_\alpha(\mathcal{B}_r(x_0)) \left ( \dashint_{\mathcal{B}_{6r}(x_0)} G_\alpha^{q_0} d\mu_\alpha \right )^{\frac{2}{q_0}} \\ & + C_5 \mu_\alpha(\mathcal{B}_r(x_0)) \left (\sum_{k=1}^\infty 2^{-k(s-\theta)} \left ( \dashint_{\mathcal{B}_{2^k3r}(x_0)} G^2 d\mu \right )^\frac{1}{2}\right )^2 \\
			\leq & C_6 \mu_\alpha(\mathcal{B}_r(x_0)) \lambda^2 \delta^2 < \varepsilon^2 \lambda^2 \mu_\alpha(\mathcal{B}_r(x_0)) ,
		\end{align*}
	where all constants depend only on $n,s,\alpha,\theta,\Lambda,m,M$ and the last inequality was obtained by choosing $\delta$ sufficiently small. This proves (\ref{west}). \par
	Let us now proof the estimate (\ref{Vest}). Define
	$$ \theta_0:=\frac{\min \{s,1-s\} + \theta}{2} \in (\theta,\min\{s,1-s\}), \quad p_0:= \frac{n+2\theta_0}{\theta_0-\theta} \in (2,\infty).$$
	Since $A$ is $\delta$-vanishing in $B_{5n}$ and therefore $(\delta,5n)$-BMO in $B_{5n}$, by \cite[Theorem 9.1]{MeV}, after choosing $\delta$ smaller if necessary, we have $v \in W^{s+\theta_0(1-2/p_0),p_0}(B_{4r}(x_0))$ and thus $V_{s,\theta_0} \in L^p(B_{4r}(x_0),\mu_{\theta_0})$. Therefore, \cite[Corollary 8.6]{MeV} yields the estimate
	\begin{align*}
		\left ( \dashint_{\mathcal{B}_{2r}(x_0)} V_{s,\theta_0}^{p_0} d\mu_{\theta_0} \right )^\frac{1}{p_0} \leq C_7 \sum_{k=1}^\infty 2^{-k(s-\theta_0)} \left ( \dashint_{\mathcal{B}_{2^k 4r}(x_0)} V_{s,\theta_0}^2 d\mu_{\theta_0} \right )^\frac{1}{2},
	\end{align*}
where $C_7=C_7(n,s,p_0,\theta_0,\Lambda)>0$,
and therefore
	\begin{align*}
			[v]_{W^{s+\theta_0(1-2/p_0),p_0}(B_{2r}(x_0))} \leq & C_7 \mu_{\theta_0}(\mathcal{B}_r(x_0))^\frac{1}{p_0} \sum_{k=1}^\infty 2^{-k(s-\theta_0)} \left ( \dashint_{\mathcal{B}_{2^k 4r}(x_0)} V_{s,\theta_0}^2 d\mu_{\theta_0} \right )^\frac{1}{2} \\
			\leq & C_{8} r^{\frac{n+2\theta_0}{p_0}-\theta_0+\theta} \sum_{k=1}^\infty 2^{-k(s-\theta)} \left ( \dashint_{\mathcal{B}_{2^k 4r}(x_0)} V^2 d\mu \right )^\frac{1}{2} \\
			= & C_8 \sum_{k=1}^\infty 2^{-k(s-\theta)} \left ( \dashint_{\mathcal{B}_{2^k 4r}(x_0)} V^2 d\mu \right )^\frac{1}{2},
	\end{align*}
where $C_8=C_8(n,s,p_0,\theta_0,\Lambda)$ and we used that $$ \frac{n+2\theta_0}{p_0}-\theta_0+\theta= 0.$$ 
Since $s+\theta_0(1-2/p_0)-\frac{n}{p_0}=s+\theta =\alpha+\theta_\alpha$, combining the previous display with the fractional Sobolev embedding given by (\ref{hemb}) yields
\begin{equation} \label{veq}
\begin{aligned}
[v]_{C^{\alpha+\theta_\alpha}(B_{2r}(x_0))}
\leq & C_{9} [v]_{W^{s+\theta_0(1-2/p_0),p_0}(B_{2r}(x_0))} \\
\leq & C_{10} \sum_{k=1}^\infty 2^{-k(s-\theta)} \left ( \dashint_{\mathcal{B}_{2^k 4r}(x_0)} V^2 d\mu \right )^\frac{1}{2},
\end{aligned}
\end{equation}
where $C_{9}$ and $C_{10}$ depend only on $n,s,p_0,\theta_0,\Lambda$.
Now in view of Proposition \ref{appplxy} along with (\ref{UG}), (\ref{Glambda}) and (\ref{Ulambda}), we have
\begin{align*}
	& \sum_{k=1}^\infty 2^{-k(s-\theta)} \left ( \dashint_{\mathcal{B}_{2^k 4r}(x_0)} V^2 d\mu \right )^\frac{1}{2} \\ \leq & \sum_{k=1}^\infty 2^{-k(s-\theta)} \left ( \dashint_{\mathcal{B}_{2^k 4r}(x_0)} U^2 d\mu \right )^\frac{1}{2} + C_{11} \left (\frac{1}{\mu(\mathcal{B}_r(x_0)} \int_{\mathbb{R}^n} \int_{\mathbb{R}^n} \frac{(w(x)-w(y))^2}{|x-y|^{n+2s}}dydx \right )^\frac{1}{2} \\
	\leq & \sum_{k=1}^\infty 2^{-k(s-\theta)} \left ( \dashint_{\mathcal{B}_{2^k 4r}(x_0)} U^2 d\mu \right )^\frac{1}{2} 
	+ C_{12} \sum_{k=1}^\infty 2^{-k(s-\theta)} \left ( \dashint_{\mathcal{B}_{2^k3r}(x_0)} G^2 d\mu \right )^\frac{1}{2}
	\leq C_{13} \lambda,
\end{align*}
where all constants depend only on $n,s,\alpha,\theta,\Lambda,M$.
Therefore, combining the last display with (\ref{veq}) yields
$$||V_\alpha||_{L^\infty(\mathcal{B}_{2r}(x_0),d\mu_\alpha)} \leq [v]_{C^{\alpha+\theta_\alpha}(B_{2r}(x_0))} \leq N_0 \lambda,$$
for some $N_0=N_0(n,s,\alpha,\theta,\Lambda,M)>0$, which proves the estimate (\ref{Vest}). This finishes the proof.
\end{proof}

\section{Good-$\lambda$ inequalities} \label{gl}
In this section, we prove some good-$\lambda$ inequalities which serve as key ingredients in the covering arguments from \cite[Section 7]{MeV}. Although the proofs of the results in this section are similar to the ones of the corresponding good-$\lambda$ inequalities in \cite[Section 6]{MeV}, since the presence of higher-order fractional gradients requires quite a few adaptations, for the sake of coherence we nevertheless provide most of the details.
\subsection{Diagonal good-$\lambda$ inequalities}
We start by proving good-$\lambda$ inequalities at the diagonal, which are somewhat akin to corresponding ones in the local setting, see e.g.\ \cite{CaffarelliPeral,ByunR}.
\begin{lem} \label{mfuse}
	There is a constant $N_d=N_d(n,s,\alpha,\theta,\Lambda) \geq 1$, such that the following holds. For any $\varepsilon > 0$ and any $\kappa>0$ there exists some small enough $\delta = \delta(\varepsilon,\kappa,n,s,\alpha,\theta,\Lambda,m) \in (0,1)$, 
	such that for any $\lambda \geq \lambda_0$, any $r \in \left (0,\frac{\sqrt{n}}{2} \right )$ and any point $x_0 \in Q_1$
	with
	\begin{equation} \label{ccll}
	\mu_\alpha \left ( \left \{(x,y) \in \mathcal{B}_{r}(x_0) \mid  {\mathcal{M}}_{\mathcal{B}_{5n}} (U_\alpha^2)(x,y) > N_d^2 \lambda^2 \right \} \right ) \geq \kappa \varepsilon \mu_\alpha(\mathcal{B}_r(x_0)),
	\end{equation}
	we have
	\begin{equation} \label{incly}
	\begin{aligned}
	\mathcal{B}_r(x_0) \subset & \left \{ (x,y) \in \mathcal{B}_r(x_0) \mid  {\mathcal{M}}_{\mathcal{B}_{5n}} (U_\alpha^2)(x,y) > \lambda^2 \right \} \\ 
	& \cap \left \{ (x,y) \in \mathcal{B}_r(x_0) \mid  {\mathcal{M}}_{\mathcal{B}_{5n}} \left (G_\alpha^{q_0} \right )(x,y) > \lambda^{q_0} \delta^{q_0} \right \},
	\end{aligned}
	\end{equation}
\end{lem}

\begin{proof}
	Let $\varepsilon_0 >0$ and $M>0$ to be chosen and consider the corresponding $\delta = \delta(\varepsilon_0,n,s,\theta,\Lambda,m,M) \in (0,1)$ given by Lemma $\ref{apppl}$.
	Fix $\varepsilon, \kappa > 0$, $r \in \left (0,\frac{\sqrt{n}}{2} \right )$, $x_0 \in Q_1$ and assume that (\ref{ccll}) holds, but that (\ref{incly}) is false, so that there exists
	a point $(x^\prime,y^\prime) \in \mathcal{B}_r(x_0)$ such that 
	$$ {\mathcal{M}}_{\mathcal{B}_{5n}} (U_\alpha^2)(x^\prime,y^\prime) \leq \lambda^2, \quad  {\mathcal{M}}_{\mathcal{B}_{5n}} \left (G_\alpha^{q_0} \right )(x^\prime,y^\prime) \leq \lambda^{q_0} \delta^{q_0}.$$
	Thus, for any $\rho >0$ we have 
	\begin{equation} \label{lumh}
	\dashint_{\mathcal{B}_\rho(x^\prime,y^\prime)}\chi_{\mathcal{B}_{5n}} U_\alpha^2 d\mu_\alpha \leq \lambda^2, \quad \dashint_{\mathcal{B}_\rho(x^\prime,y^\prime)} \chi_{\mathcal{B}_{5n}} G_\alpha^{q_0} d\mu_\alpha \leq \lambda^{q_0} \delta^{q_0}.
	\end{equation}
	Observe that for any $\rho \geq r$, we have $\mathcal{B}_\rho(x_0) \subset \mathcal{B}_{2\rho}(x^\prime,y^\prime) \subset \mathcal{B}_{3\rho}(x_0)$.
	Together with (\ref{lumh}), we obtain
	\begin{align*}
	\dashint_{\mathcal{B}_\rho(x_0)} \chi_{\mathcal{B}_{5n}} U_\alpha^2 d\mu_\alpha \leq \frac{\mu_\alpha(\mathcal{B}_{2\rho}(x^\prime,y^\prime))}{\mu_\alpha(\mathcal{B}_{\rho}(x_0))} \text{ } \dashint_{\mathcal{B}_{2\rho}(x^\prime,y^\prime)} \chi_{\mathcal{B}_{5n}} U_\alpha^2 d\mu_\alpha \leq & \frac{\mu_\alpha(\mathcal{B}_{3\rho}(x_0))}{\mu_\alpha(\mathcal{B}_\rho(x_0))} \text{ } \dashint_{\mathcal{B}_{2\rho}(x^\prime,y^\prime)} \chi_{\mathcal{B}_{5n}} U_\alpha^2 d\mu_\alpha \\ \leq & 3^{n+2\theta_\alpha} \lambda^2
	\end{align*}
	and similarly
	\begin{align*}
	\dashint_{\mathcal{B}_\rho(x_0)} \chi_{\mathcal{B}_{5n}} G_\alpha^{q_0} d\mu_\alpha \leq \frac{\mu_\alpha(\mathcal{B}_{2\rho}(x^\prime,y^\prime))}{\mu_\alpha(\mathcal{B}_{\rho}(x_0))} \text{ } \dashint_{\mathcal{B}_{2\rho}(x^\prime,y^\prime)} \chi_{\mathcal{B}_{5n}} G_\alpha^{q_0} d\mu_\alpha  \leq 3^{n+2\theta_\alpha} \lambda^{q_0} \delta^{q_0},
	\end{align*}
	so that $U_\alpha$ and $G_\alpha$ satisfy 
	the condition $(\ref{condU})$ with $M=3^{n+2\theta_\alpha}$.
	Therefore, by Lemma \ref{apppl} the weak solution $v \in W^{s,2}(\mathbb{R}^n)$ of the Dirichlet problem
	$$ \begin{cases} \normalfont
	L_{A} v = 0 & \text{ in } B_{6r}(x_0) \\
	v = u & \text{ a.e. in } \mathbb{R}^n \setminus B_{6r}(x_0)
	\end{cases} $$
	satisfies
	\begin{equation} \label{aprox5}
	\int_{\mathcal{B}_{2r}(x_0)} W_\alpha^2 d\mu_\alpha \leq \varepsilon_0^2 \lambda^2 \mu_\alpha(\mathcal{B}_{r}(x_0)) ,
	\end{equation}
	where $W_\alpha$ is given as in (\ref{W}). In addition, also by Lemma $\ref{apppl}$ there exists a constant $N_0=N_0(n,s,\alpha,\theta, \Lambda) >0$ such that 
	\begin{equation} \label{loclinf}
	||V_\alpha||_{L^\infty(\mathcal{B}_{2r}(x_0))}^2 \leq N_0^2 \lambda^2.
	\end{equation}
	Next, we set $N_d := (\max \{ 4 N_0^2, 5^{n+2\theta_\alpha} \})^{1/2} > 1$ and claim that 
	\begin{equation} \label{inclusion}
	\begin{aligned}
	& \left \{ (x,y) \in \mathcal{B}_r(x_0) \mid  {\mathcal{M}}_{\mathcal{B}_{5n}} ( U_\alpha^2 )(x,y) > N_d^2\lambda^2 \right \} \\ \subset & \left \{ (x,y) \in \mathcal{B}_r(x_0) \mid  {\mathcal{M}}_{\mathcal{B}_{2r}(x_0)} ( W_\alpha^2 )(x,y) > N_0^2\lambda^2 \right \}. 
	\end{aligned}
	\end{equation}
	To see this, assume that 
	\begin{equation} \label{menge}
	(x_1,y_1) \in \left \{ x \in \mathcal{B}_r(x_0) \mid  {\mathcal{M}}_{\mathcal{B}_{2r}(x_0)} ( W_\alpha^2 ) (x,y) \leq N_0^2\lambda^2 \right \}. 
	\end{equation}
	For $ \rho < r$, we have $\mathcal{B}_\rho (x_1,y_1) \subset \mathcal{B}_r(x_1,y_1) \subset \mathcal{B}_{2r}(x_0)$, so that together with $(\ref{menge})$ and $(\ref{loclinf})$ we deduce 
	\begin{align*}
	\dashint_{\mathcal{B}_\rho (x_1,y_1)} U_\alpha^2 d\mu_\alpha & \leq 2 \text{ } \dashint_{\mathcal{B}_\rho(x_1,y_1)} \left ( W_\alpha^2+ V_\alpha^2 \right )d\mu_\alpha \\
	& \leq 2 \text{ } \dashint_{\mathcal{B}_\rho(x_1,y_1)} W_\alpha^2 d\mu_\alpha + 2 \text{ } ||V_\alpha||_{L^\infty(\mathcal{B}_\rho(x_1,y_1))}^2 \\
	& \leq 2 \text{ }  {\mathcal{M}}_{\mathcal{B}_{2r}(x_0)} ( W_\alpha^2 )(x_1,y_1) + 2 \text{ } ||V_\alpha||_{L^\infty(\mathcal{B}_{2r}(x_0))}^2 \leq 4 N_0^2\lambda^2. 
	\end{align*}
	On the other hand, for $\rho \geq r$ we have $ \mathcal{B}_\rho (x_1,y_1) \subset \mathcal{B}_{3 \rho}(x^\prime,y^\prime) \subset \mathcal{B}_{5 \rho}(x_1,y_1)$, so that $(\ref{lumh})$ implies 
	\begin{align*}
	\dashint_{\mathcal{B}_\rho(x_1,y_1)} \chi_{\mathcal{B}_{5n}} U_\alpha^2 d\mu_\alpha \leq & \frac{\mu_\alpha(\mathcal{B}_{3 \rho} (x^\prime,y^\prime))}{\mu_\alpha(\mathcal{B}_\rho(x_1,y_1))} \dashint_{\mathcal{B}_{3 \rho} (x^\prime,y^\prime)} \chi_{\mathcal{B}_{5n}} U_\alpha^2 d\mu_\alpha \\ \leq & \frac{\mu_\alpha(\mathcal{B}_{5 \rho} (x_1,y_1))}{\mu_\alpha(\mathcal{B}_\rho(x_1,y_1))} \dashint_{\mathcal{B}_{3 \rho} (x^\prime,y^\prime)} \chi_{\mathcal{B}_{5n}} U_\alpha^2 d\mu_\alpha \leq 5^{n+2\theta_\alpha}\lambda^2.
	\end{align*}
	Thus, we have $$ (x_1,y_1) \in \left \{ (x,y) \in \mathcal{B}_r(x_0,y_0) \mid  {\mathcal{M}}_{\mathcal{B}_{5n}} ( U_\alpha^2 )(x,y) \leq N_d^2 \lambda^2 \right \} ,$$ 
	which implies $(\ref{inclusion})$. 
	Now using $(\ref{inclusion})$, the weak $1$-$1$ estimate from Proposition \ref{Maxfun} and $(\ref{aprox5})$, we conclude that there exists some constant $C=C(n,\theta_\alpha)>0$ such that 
	\begin{align*}
	& \mu_\alpha \left ( \left \{(x,y) \in \mathcal{B}_r(x_0) \mid  {\mathcal{M}}_{\mathcal{B}_{5n}} (U_\alpha^2)(x,y) > N_d^2\lambda^2 \right \} \right ) \\
	\leq & \mu_\alpha \left ( \left \{ (x,y) \in \mathcal{B}_r(x_0) \mid  {\mathcal{M}}_{\mathcal{B}_{2r}(x_0)} ( W_\alpha^2 )(x,y) > N_0^2\lambda^2 \right \} \right ) \\
	\leq & \frac{C}{N_0^2\lambda^2} \int_{\mathcal{B}_{2r}(x_0)} W_\alpha^2 d\mu_\alpha \\
	\leq & \frac{C}{N_0^2} \mu_\alpha(\mathcal{B}_{r}(x_0)) \varepsilon_0^2 < \varepsilon \kappa \mu_\alpha(\mathcal{B}_{r}(x_0)),
	\end{align*}
	where the last inequality is obtained by choosing $\varepsilon_0$ and thus also $\delta$ sufficiently small. This contradicts (\ref{ccll}) and thus finishes our proof.
\end{proof}

\subsection{Off-diagonal reverse H\"older inequalities}
While in the setting of local elliptic equations of the form (\ref{localeq}) proving analogues of the above diagonal good-$\lambda$ inequalities is sufficient in order to establish the desired Sobolev regularity, in our nonlocal setting which involves fractional gradients defined on $\mathbb{R}^{2n}$, it is also necessary to prove an analogue of Lemma \ref{mfuse} on balls that are far away from the diagonal. However, since far away from the diagonal the equation cannot be used very efficiently, in this situation no useful comparison estimates are available. \par In order to bypass this loss of information, as in \cite{MeV} we replace the comparison estimates used in the diagonal setting by certain off-diagonal reverse H\"older inequalities with diagonal correction terms, which in view of an iteration argument in the end will still be sufficiently strong tools in order to deduce the desired regularity. \par
For this reason, in addition to the assumption that $u$ satisfies the estimate (\ref{hdest}), from now on we assume that for any $r>0$, $x_0 \in \mathbb{R}^n$ with $B_{r}(x_0) \subset B_{5n}$, $U_\alpha$ satisfies an estimate of the form
\begin{equation} \label{J}
\begin{aligned}
\left (\dashint_{\mathcal{B}_{r/2}(x_0)} U_\alpha^{q} d\mu_\alpha \right )^{\frac{1}{q}} \leq C_q& \Bigg ( \sum_{k=1}^\infty 2^{-k(s-\theta)} \left ( \dashint_{\mathcal{B}_{2^kr}(x_0)} U^2 d\mu \right )^\frac{1}{2} \\ & + \left (\dashint_{\mathcal{B}_{r}(x_0)} G_\alpha^{q_0} d\mu_\alpha \right )^\frac{1}{q_0} + \sum_{k=1}^\infty 2^{-k(s-\theta)} \left ( \dashint_{\mathcal{B}_{2^kr}(x_0)} G^2 d\mu \right )^\frac{1}{2} \Bigg ),
\end{aligned}
\end{equation}
where $C_q$ depends only on $q,n,s,\alpha,\theta,m$ and $\Lambda$.
\begin{prop} \label{offdiagreverse}
	Let $r>0$, $x_0,y_0 \in \mathbb{R}^n$ and suppose that for some $\gamma \in (0,1]$ we have $\textnormal{dist}(B_r(x_0),B_r(y_0)) \geq \gamma r$.
	Then we have
	\begin{align*}
	& \left (\dashint_{\mathcal{B}_r(x_0,y_0)} U_\alpha^{q_\alpha^\star} d\mu_\alpha \right )^{\frac{1}{q_\alpha^\star}} \\ \leq & C_{nd} \left (\dashint_{\mathcal{B}_r(x_0,y_0)} U_\alpha^{2} d\mu_\alpha \right )^\frac{1}{2} \\
	& + C_{nd} \left ( \frac{r}{\textnormal{dist}(B_r(x_0),B_r(y_0))} \right )^{\alpha+\theta_\alpha} \Bigg ( \sum_{k=1}^\infty 2^{-k(s-\theta)} \left ( \dashint_{\mathcal{B}_{2^kr}(x_0)} U^2 d\mu \right )^\frac{1}{2} \\ & \text{ }+ \left (\dashint_{\mathcal{B}_{2r}(x_0)} G_\alpha^{q_0} d\mu_\alpha \right )^\frac{1}{q_0} + \sum_{k=1}^\infty 2^{-k(s-\theta)} \left ( \dashint_{\mathcal{B}_{2^kr}(x_0)} G^2 d\mu \right )^\frac{1}{2} \Bigg ) \\
	& + C_{nd} \left ( \frac{r}{\textnormal{dist}(B_r(x_0),B_r(y_0))} \right )^{\alpha+\theta_\alpha} \Bigg ( \sum_{k=1}^\infty 2^{-k(s-\theta)} \left ( \dashint_{\mathcal{B}_{2^kr}(y_0)} U^2 d\mu \right )^\frac{1}{2} \\ & \text{ }+ \left (\dashint_{\mathcal{B}_{2r}(y_0)} G_\alpha^{q_0} d\mu_\alpha \right )^\frac{1}{q_0} + \sum_{k=1}^\infty 2^{-k(s-\theta)} \left ( \dashint_{\mathcal{B}_{2^kr}(y_0)} G^2 d\mu \right )^\frac{1}{2} \Bigg ),
	\end{align*}
	where $C_{nd}=C_{nd}(n,s,\alpha,\theta,\Lambda,\gamma,m,q,p) \geq 1$ and $q_\alpha^\star$ is given by (\ref{qstar}).
\end{prop}

\begin{proof}
	Choose points $x_1 \in \overline B_r(x_0)$ and $y_1 \in \overline B_r(y_0)$ such that $\textnormal{dist}(B_r(x_0),B_r(y_0))=|x_1-y_1|$. For any $(x,y) \in \mathcal{B}_r(x_0,y_0)$, we observe that
	\begin{align*}
	|x-y| \leq & |x_1-y_1| + |x_1-x| + |y_1-y| \\
	\leq & \textnormal{dist}(B_r(x_0),B_r(y_0)) + 2 r
	\leq 3 \textnormal{dist}(B_r(x_0),B_r(y_0))/\gamma.
	\end{align*}
	Together with the definition of $\textnormal{dist}(B_r(x_0),B_r(y_0))$, for any $(x,y) \in \mathcal{B}_r(x_0,y_0)$ we obtain
	\begin{equation} \label{ineq99}
	1 \leq \frac{|x-y|}{\textnormal{dist}(B_r(x_0),B_r(y_0))} \leq 3 /\gamma.
	\end{equation}
	Thus, by taking into account the definition of the measure $\mu_\alpha$, we conclude that
	\begin{equation} \label{comparable}
	\frac{c_1 r^{2n}}{\textnormal{dist}(B_r(x_0),B_r(y_0))^{n-2\theta_\alpha}} \leq \mu_\alpha(\mathcal{B}_r(x_0,y_0)) \leq \frac{C_1 r^{2n}}{\textnormal{dist}(B_r(x_0),B_r(y_0))^{n-2\theta_\alpha}},
	\end{equation}
	where $c_1=c_1(n,\gamma,\theta_\alpha) \in (0,1)$ and $C_1=C_1(n,\theta_\alpha) \geq 1$.
	By (\ref{comparable}) and (\ref{ineq99}), we have
	\begin{align*}
	\left (\dashint_{\mathcal{B}_r(x_0,y_0)} U_\alpha^{q_\alpha^\star} d\mu_\alpha \right )^{\frac{1}{q_\alpha^\star}} \leq & \left (\frac{\textnormal{dist}(B_r(x_0),B_r(y_0))^{n-2\theta_\alpha}}{c_1 r^{2n}} \int_{B_r(x_0)}\int_{B_r(y_0)} \frac{|u(x)-u(y)|^{q_\alpha^\star}}{|x-y|^{n-2\theta_\alpha+{q_\alpha^\star}(\alpha+\theta_\alpha)}}dydx \right )^{\frac{1}{q_\alpha^\star}} \\
	\leq & C_2 \textnormal{dist}(B_r(x_0),B_r(y_0))^{-(\alpha+\theta_\alpha)} \left ( \dashint_{B_r(x_0)} \dashint_{B_r(y_0)} |u(x)-u(y)|^{q_\alpha^\star}dydx \right )^{\frac{1}{q_\alpha^\star}},
	\end{align*}
	where $C_2=C_2(n,\gamma,\theta_\alpha)\geq 1$. In view of Minkowski's inequality, we can further estimate the integral on the right-hand side as follows
	\begin{align*}
	\left ( \dashint_{B_r(x_0)} \dashint_{B_r(y_0)} |u(x)-u(y)|^{q_\alpha^\star}dydx \right )^{\frac{1}{q_\alpha^\star}} \leq & \underbrace{\left ( \dashint_{B_r(x_0)} |u(x)- \overline u_{B_r(x_0)}|^{q_\alpha^\star}dydx \right )^{\frac{1}{q_\alpha^\star}}}_{=:I_1} \\
	& + \underbrace{\left ( \dashint_{B_r(y_0)} |u(x)- \overline u_{B_r(y_0)}|^{q_\alpha^\star}dydx \right )^{\frac{1}{q_\alpha^\star}}}_{=:I_2} \\
	& + \underbrace{|\overline u_{B_r(x_0)}-\overline u_{B_r(y_0)}|}_{=:I_3}.
	\end{align*}
	By using the fractional Sobolev-Poincar\'e inequality (Lemma \ref{SobPoincare}) and then the estimate (\ref{J}), for $I_1$ we obtain
	\begin{align*}
	I_1 \leq & C_3 r^\alpha \left ( \frac{1}{r^n} \int_{B_r(x_0)} \int_{B_r(x_0)} \frac{|u(x)-u(y)|^{q}}{|x-y|^{n+\alpha q}}dydx \right )^{\frac{1}{q}} \\
	= & C_4 r^\alpha \left (\frac{r^{2 \theta_\alpha} }{\mu_\alpha(\mathcal{B}_r(x_0))} \int_{B_r(x_0)} \int_{B_r(x_0)} \frac{|u(x)-u(y)|^{q} |x-y|^{(q-2)\theta_\alpha}}{|x-y|^{n-2\theta_\alpha+q(\alpha+\theta_\alpha)}}dydx \right )^{\frac{1}{q}} \\
	\leq & C_5 r^\alpha \left ( r^{q\theta_\alpha} \dashint_{\mathcal{B}_r(x_0)} \frac{|u(x)-u(y)|^{q}}{|x-y|^{q(\alpha+\theta_\alpha)}}d\mu_\alpha \right )^{\frac{1}{q}} \\
	= & C_5 r^{\alpha+\theta_\alpha} \left ( \dashint_{\mathcal{B}_r(x_0)} U_\alpha^q d\mu_\alpha \right )^\frac{1}{q} \\
	\leq & C_q C_5 r^{\alpha+\theta_\alpha} \Bigg ( \sum_{k=1}^\infty 2^{-k(s-\theta)} \left ( \dashint_{\mathcal{B}_{2^kr}(x_0)} U^2 d\mu \right )^\frac{1}{2} \\ & \text{ } + \left (\dashint_{\mathcal{B}_{2r}(x_0)} G_\alpha^{q_0} d\mu_\alpha \right )^\frac{1}{q_0} + \sum_{k=1}^\infty 2^{-k(s-\theta)} \left ( \dashint_{\mathcal{B}_{2^kr}(x_0)} G^2 d\mu \right )^\frac{1}{2} \Bigg )
	\end{align*}
	where $C_3,C_4$ and $C_5$ depend only on $n,s,\alpha,\theta$ and $\theta_\alpha$. In the same way, for $I_2$ we deduce that
	\begin{align*}
	I_2 \leq & C_q C_5 r^{\alpha+\theta_\alpha} \Bigg ( \sum_{k=1}^\infty 2^{-k(s-\theta)} \left ( \dashint_{\mathcal{B}_{2^kr}(y_0)} U^2 d\mu \right )^\frac{1}{2} \\ & \text{ } + \left (\dashint_{\mathcal{B}_{2r}(y_0)} G_\alpha^{q_0} d\mu_\alpha \right )^\frac{1}{q_0} + \sum_{k=1}^\infty 2^{-k(s-\theta)} \left ( \dashint_{\mathcal{B}_{2^kr}(y_0)} G^2 d\mu \right )^\frac{1}{2} \Bigg ).
	\end{align*}
	Finally, by the Cauchy-Schwarz inequality, (\ref{comparable}) and (\ref{ineq99}), for $I_3$ we have
	\begin{align*}
	I_3 \leq & \dashint_{B_r(x_0)} \dashint_{B_r(y_0)} |u(x)-u(y)|dydx \\
	\leq & \left (\dashint_{B_r(x_0)} \dashint_{B_r(y_0)} |u(x)-u(y)|^2 dydx \right )^\frac{1}{2} \\
	\leq & \left (\frac{C_1}{\textnormal{dist}(B_r(x_0),B_r(y_0))^{n-2\theta_\alpha} \mu_\alpha(\mathcal{B}_r(x_0,y_0))} \int_{B_r(x_0)} \int_{B_r(y_0)} |u(x)-u(y)|^2 dydx \right )^\frac{1}{2} \\
	\leq & C_6 \left (\dashint_{\mathcal{B}_r(x_0,y_0)} |u(x)-u(y)|^2 d\mu_\alpha \right )^\frac{1}{2} \\
	\leq & C_7 \textnormal{dist}(B_r(x_0),B_r(y_0))^{\alpha+\theta_\alpha} \left (\dashint_{\mathcal{B}_r(x_0,y_0)} U_\alpha^2 d\mu_\alpha \right )^\frac{1}{2},
	\end{align*}
	where $C_6=C_6(n,\gamma,\theta_\alpha) \geq 1$ and $C_7=C_7(n,\gamma,\theta_\alpha) \geq 1$. The claim now follows by combining the last five displays, so that the proof is finished.
\end{proof}
\subsection{Off-diagonal good-$\lambda$ inequalities}
In what follows, we fix some $\varepsilon \in (0,1)$ to be chosen small enough and set 
\begin{equation} \label{Neps}
N_{\varepsilon,q}:=\frac{C_{nd} C_{s,\theta} C_\alpha N_d 10^{10n}}{ \varepsilon^{1/q_\alpha^\star}},
\end{equation}
where $N_d=N_d(n,s,\alpha,\theta,\Lambda) \geq 1$ is given by Lemma \ref{mfuse},  $C_{nd}=C_{nd}(n,s,\alpha,\theta,\Lambda,\gamma,m,q,p) \geq 1$ is given by Proposition \ref{offdiagreverse} with $\gamma$ to be chosen and 
\begin{equation} \label{Cst}
1 \leq C_{s,\theta}:= \sum_{k=1}^\infty 2^{-k(s-\theta)} < \infty,
\end{equation}
while $C_\alpha=C_\alpha(n,s,\alpha,\theta)>0$ is given by Lemma \ref{Urel} with $t=\alpha$.
Moreover, for all $r \in \left (0,\frac{\sqrt{n}}{2} \right )$ and all $(x_0,y_0) \in \mathcal{Q}_1$ we define
\begin{equation} \label{distfct}
\widetilde \phi(r,x_0,y_0):= \frac{r}{\textnormal{dist}(B_\frac{r}{2}(x_0),B_\frac{r}{2}(y_0))}.
\end{equation}

\begin{lem} \label{mfusenondiag}
	For any $\lambda \geq \lambda_0$, $r \in \left (0,\frac{\sqrt{n}}{2} \right )$ and any point $(x_0,y_0) \in \mathcal{Q}_1$ satisfying $|x_0-y_0| \geq (3\sqrt{n}+1)r$ and
	\begin{equation} \label{ccllz}
	\mu_\alpha \left ( \left \{(x,y) \in \mathcal{B}_{\frac{\sqrt{n}}{2}r}(x_0,y_0) \mid  {\mathcal{M}}_{\mathcal{B}_{5n}} (U_\alpha^2)(x,y) > N_{\varepsilon,q}^2 \lambda^2 \right \} \right ) \geq \varepsilon \mu_\alpha(\mathcal{B}_{\frac{r}{2}}(x_0,y_0)),
	\end{equation}
	we have
	\begin{align*}
	\mathcal{B}_{\frac{r}{2}}(x_0,y_0)
	\subset & \left \{ (x,y) \in \mathcal{B}_{\frac{r}{2}}(x_0,y_0) \mid  {\mathcal{M}}_{\mathcal{B}_{5n}} (U_\alpha^2)(x,y) > \lambda^2 \right \} \\ 
	\cup & \left \{ (x,y) \in \mathcal{B}_{\frac{r}{2}}(x_0,y_0) \mid  {\mathcal{M}}_{\geq r,\mathcal{B}_{5n}} (U_\alpha^2)(x,x) > 3^{n+2\theta_\alpha}N_d^2 \widetilde \phi(r,x_0,y_0)^{-2(\alpha+\theta_\alpha)} \lambda^2 \right \} \\
	\cup & \left \{ (x,y) \in \mathcal{B}_{\frac{r}{2}}(x_0,y_0) \mid  {\mathcal{M}}_{\geq r,\mathcal{B}_{5n}} (U_\alpha^2)(y,y) > 3^{n+2\theta_\alpha}N_d^2 \widetilde \phi(r,x_0,y_0)^{-2(\alpha+\theta_\alpha)} \lambda^2 \right \} \\
	\cup & \left \{ (x,y) \in \mathcal{B}_{\frac{r}{2}}(x_0,y_0) \mid  {\mathcal{M}}_{\geq r,\mathcal{B}_{5n}} (G_\alpha^{q_0})(x,x) > 3^{n+2\theta_\alpha} \widetilde \phi(r,x_0,y_0)^{-q_0(\alpha+\theta_\alpha)} \lambda^{q_0} \right \} \\
	\cup & \left \{ (x,y) \in \mathcal{B}_{\frac{r}{2}}(x_0,y_0) \mid  {\mathcal{M}}_{\geq r,\mathcal{B}_{5n}} (G_\alpha^{q_0})(y,y) > 3^{n+2\theta_\alpha} \widetilde \phi(r,x_0,y_0)^{-q_0(\alpha+\theta_\alpha)} \lambda^{q_0} \right \}.
	\end{align*}
\end{lem}

\begin{proof}
	Assume that (\ref{ccllz}) holds, but that the conclusion is false, so that there exists a point $(x^\prime,y^\prime) \in \mathcal{B}_{\frac{r}{2}}(x_0,y_0)$ such that 
	\begin{align*}
	& {\mathcal{M}}_{\mathcal{B}_{5n}}(U_\alpha^2)(x^\prime,y^\prime) \leq \lambda^2, \\ &  {\mathcal{M}}_{\geq r,\mathcal{B}_{5n}} (U_\alpha^2)(x^\prime,x^\prime) \leq 3^{n+2\theta_\alpha}N_d^2 \widetilde \phi(r,x_0,y_0)^{-2(\alpha+\theta_\alpha)} \lambda^2, \\ &
	 {\mathcal{M}}_{\geq r,\mathcal{B}_{5n}} (U_\alpha^2)(y^\prime,y^\prime) \leq 3^{n+2\theta_\alpha}N_d^2 \widetilde \phi(r,x_0,y_0)^{-2(\alpha+\theta_\alpha)} \lambda^2, \\ &  {\mathcal{M}}_{\geq r,\mathcal{B}_{5n}} (G_\alpha^{q_0})(x^\prime,x^\prime) \leq 3^{n+2\theta_\alpha} \widetilde \phi(r,x_0,y_0)^{-{q_0}(\alpha+\theta_\alpha)} \lambda^{q_0}, \\ &
	 {\mathcal{M}}_{\geq r,\mathcal{B}_{5n}} (G_\alpha^{q_0})(y^\prime,y^\prime) \leq 3^{n+2\theta_\alpha} \widetilde \phi(r,x_0,y_0)^{-{q_0}(\alpha+\theta_\alpha)} \lambda^{q_0}.
	\end{align*}
	Therefore, for any $\rho \geq r$ we have
	\begin{equation} \label{lumhn}
	\dashint_{\mathcal{B}_\rho(x^\prime,y^\prime)}\chi_{\mathcal{B}_{5n}} U_\alpha^2 d\mu_\alpha \leq \lambda^2 ,
	\end{equation}
	\begin{equation} \label{lumhn1}
	\begin{aligned}
	&\dashint_{\mathcal{B}_\rho(x^\prime)} \chi_{\mathcal{B}_{5n}} U_\alpha^2 d\mu_\alpha \leq 3^{n+2\theta_\alpha}N_d^2 \widetilde \phi(r,x_0,y_0)^{-{2}(\alpha+\theta_\alpha)} \lambda^2, \\ & \dashint_{\mathcal{B}_\rho(y^\prime)} \chi_{\mathcal{B}_{5n}} U_\alpha^2 d\mu_\alpha \leq 3^{n+2\theta_\alpha}N_d^2 \widetilde \phi(r,x_0,y_0)^{-{2}(\alpha+\theta_\alpha)}\lambda^2
	\end{aligned}
	\end{equation} 
	and similarly
	\begin{equation} \label{lumhn1x}
	\begin{aligned}
	& \dashint_{\mathcal{B}_\rho(x^\prime)} \chi_{\mathcal{B}_{5n}} G_\alpha^{q_0} d\mu_\alpha \leq 3^{n+2\theta_\alpha}\widetilde \phi(r,x_0,y_0)^{-{q_0}(\alpha+\theta_\alpha)} \lambda^{q_0}, \\ & \dashint_{\mathcal{B}_\rho(y^\prime)} \chi_{\mathcal{B}_{5n}} G_\alpha^{q_0} d\mu_\alpha \leq 3^{n+2\theta_\alpha} \widetilde \phi(r,x_0,y_0)^{-{q_0}(\alpha+\theta_\alpha)}\lambda^{q_0}.
	\end{aligned}
	\end{equation} 
	Since for any $\rho \geq r$ we have $\mathcal{B}_\rho(x_0,y_0) \subset \mathcal{B}_{2\rho}(x^\prime,y^\prime) \subset \mathcal{B}_{3\rho}(x_0,y_0)$, from (\ref{lumhn}) we deduce
	\begin{equation} \label{Uoffdiag}
	\dashint_{\mathcal{B}_\rho(x_0,y_0)} \chi_{\mathcal{B}_{5n}} U_\alpha^2 d\mu_\alpha \leq \frac{\mu_\alpha(\mathcal{B}_{2\rho}(x^\prime,y^\prime))}{\mu_\alpha(\mathcal{B}_\rho(x_0,y_0))} \dashint_{\mathcal{B}_{2\rho}(x^\prime)} \chi_{\mathcal{B}_{5n}} U_\alpha^2 d\mu_\alpha \leq 3^{n+2\theta_\alpha} \lambda^2.
	\end{equation}
	Since for any $\rho \geq r$ we have $\mathcal{B}_\rho(x_0) \subset \mathcal{B}_{2\rho}(x^\prime)$, together with (\ref{lumhn1}) we observe that
	\begin{equation} \label{Ux0}
	\dashint_{\mathcal{B}_\rho(x_0)}\chi_{\mathcal{B}_{5n}} U_\alpha^2 d\mu_\alpha \leq \frac{\mu_\alpha(\mathcal{B}_{2\rho}(x^\prime))}{\mu_\alpha(\mathcal{B}_\rho(x_0))} \dashint_{\mathcal{B}_{2\rho}(x^\prime)}\chi_{\mathcal{B}_{5n}} U_\alpha^2 d\mu_\alpha \leq 6^{n+2\theta_\alpha}N_d^2 \widetilde \phi(r,x_0,y_0)^{-2(\alpha+\theta_\alpha)} \lambda^2 
	\end{equation}
	and similarly by using (\ref{lumhn1x}) instead of (\ref{lumhn1}), we obtain
	\begin{equation} \label{Gx0}
	\dashint_{\mathcal{B}_\rho(x_0)}\chi_{\mathcal{B}_{5n}} G_\alpha^{q_0} d\mu_\alpha \leq \frac{\mu_\alpha(\mathcal{B}_{2\rho}(x^\prime))}{\mu_\alpha(\mathcal{B}_\rho(x_0))} \dashint_{\mathcal{B}_{2\rho}(x^\prime)}\chi_{\mathcal{B}_{5n}} G_\alpha^{q_0} d\mu_\alpha \leq 6^{n+2\theta_\alpha} \widetilde \phi(r,x_0,y_0)^{-{q_0}(\alpha+\theta_\alpha)} \lambda^{q_0}.
	\end{equation}
	By the same reasoning, (\ref{Ux0}) and (\ref{Gx0}) hold also with $x_0$ replaced by $y_0$.
	Next, we claim that 
	\begin{equation} \label{inclusion1}
	\begin{aligned}
	& \left \{ (x,y) \in \mathcal{B}_{\frac{\sqrt{n}}{2}r}(x_0,y_0) \mid  {\mathcal{M}}_{\mathcal{B}_{5n}}( U_\alpha^2 )(x,y) > N_{\varepsilon,q}^2 \lambda^2 \right \} \\ \subset & \left \{ (x,y) \in \mathcal{B}_{\frac{ \sqrt{n}}{2}r}(x_0,y_0) \mid  {\mathcal{M}}_{\mathcal{B}_{\frac{3\sqrt{n}}{2}r}(x_0,y_0)} ( U_\alpha^2 )(x,y) > N_{\varepsilon,q}^2 \lambda^2 \right \}. 
	\end{aligned}
	\end{equation}
	To see this, assume that 
	\begin{equation} \label{menge1}
	(x_1,y_1) \in \left \{ x \in \mathcal{B}_{\frac{\sqrt{n}}{2}r}(x_0,y_0) \mid  {\mathcal{M}}_{\mathcal{B}_{\frac{3\sqrt{n}}{2}r}(x_0,y_0)} ( U_\alpha^2 ) (x,y) \leq N_{\varepsilon,q}^2 \lambda^2 \right \}. 
	\end{equation}
	For $ \rho < \sqrt{n}r$, we have $\mathcal{B}_\rho (x_1,y_1) \subset \mathcal{B}_{\sqrt{n}r}(x_1,y_1) \subset \mathcal{B}_{\frac{3\sqrt{n}}{2}r}(x_0,y_0)$, so that along with $(\ref{menge1})$ we deduce 
	\begin{align*}
	\dashint_{\mathcal{B}_\rho (x_1,y_1)} U_\alpha^2 d\mu_\alpha & \leq  {\mathcal{M}}_{\mathcal{B}_{\frac{3\sqrt{n}}{2}r}(x_0,y_0)} ( U_\alpha^2 )(x_1,y_1) \leq N_{\varepsilon,q}^2 \lambda^2. 
	\end{align*}
	On the other hand, for $\rho \geq \sqrt{n}r$ we have $ \mathcal{B}_\rho (x_1,y_1) \subset \mathcal{B}_{3 \rho}(x^\prime,y^\prime) \subset \mathcal{B}_{5 \rho}(x_1,y_1)$, so that $(\ref{lumhn})$ implies 
	\begin{align*}
	\dashint_{\mathcal{B}_\rho(x_1,y_1)} \chi_{\mathcal{B}_{5n}} U_\alpha^2 d\mu_\alpha \leq \frac{\mu_\alpha(\mathcal{B}_{5 \rho} (x_1,y_1))}{\mu_\alpha(\mathcal{B}_\rho(x_1,y_1))} \dashint_{\mathcal{B}_{3 \rho} (x^\prime,y^\prime)} \chi_{\mathcal{B}_{5n}} U_\alpha^2 d\mu_\alpha \leq 5^{n+2\theta_\alpha} \lambda^2 \leq N_{\varepsilon,q}^2 \lambda^2.
	\end{align*}
	Thus, we have $$ (x_1,y_1) \in \left \{ (x,y) \in \mathcal{B}_r(x_0,y_0) \mid  {\mathcal{M}}_{\mathcal{B}_{5n}}( U_\alpha^2 )(x,y) \leq N_{\varepsilon,q}^2 \lambda^2 \right \} ,$$ 
	which implies $(\ref{inclusion1})$.
As in the proof of Lemma \ref{apppl}, let $l \in \mathbb{N}$ be determined by $2^{l-1} r < \sqrt{n} \leq 2^{l} r$, note that $l \geq 2$. Then for any $k < l$, by (\ref{Ux0}) and (\ref{Gx0}) we have
\begin{equation} \label{UGx}
\begin{aligned}
&\dashint_{\mathcal{B}_{2^k\frac{3\sqrt{n}}{2} r}(x_0)} U_\alpha^2 d\mu_\alpha \leq 6^{n+2\theta_\alpha}\lambda^2 \widetilde \phi(r,x_0,y_0)^{-2(\alpha+\theta_\alpha)}, \\ &\dashint_{\mathcal{B}_{2^k\frac{3\sqrt{n}}{2} r}(x_0)} G_\alpha^{q_0} d\mu \leq 6^{n+2\theta_\alpha}\lambda^{q_0} \delta^{q_0} \widetilde \phi(r,x_0,y_0)^{-2(\alpha+\theta_\alpha)}.
\end{aligned}
\end{equation}
Moreover, in view of (\ref{tailcontrol}), the inclusions $$\mathcal{B}_{2^k \frac{n}{2}}(x_0) \subset \mathcal{B}_{2^{k+l-1} \frac{3\sqrt{n}}{2} r}(x_0) \subset \mathcal{B}_{2^k \frac{3n}{2}}(x_0) \subset \mathcal{B}_{2^k 5n}$$ and the fact that $\widetilde \phi(r,x_0,y_0) \leq 1$, we have
\begin{equation} \label{Ut}
\begin{aligned}
& \sum_{k=l}^\infty 2^{-k(s-\theta)} \left ( \dashint_{\mathcal{B}_{2^k \frac{3\sqrt{n}}{2}r}(x_0)} U^2 d\mu \right )^\frac{1}{2} \\ = & 2^{-(l-1)(s-\theta)} \sum_{k=1}^\infty 2^{-k(s-\theta)} \left ( \dashint_{\mathcal{B}_{2^{k+l-1}\frac{3\sqrt{n}}{2}r}(x_0)} U^2 d\mu \right )^\frac{1}{2} \\
\leq & \sum_{k=1}^\infty 2^{-k(s-\theta)} \left (\frac{\mu(\mathcal{B}_{2^k 5n})}{\mu \left (\mathcal{B}_{2^k \frac{n}{2}} \right)} \dashint_{\mathcal{B}_{2^k 5n}} U^2 d\mu \right )^\frac{1}{2} \\
\leq & {10}^{\frac{n}{2}+\theta} \lambda_0 \leq {10}^{\frac{n}{2}+\theta} \lambda_0 \widetilde \phi(r,x_0,y_0)^{-(\alpha+\theta_\alpha)}.
\end{aligned}
\end{equation}
Together with (\ref{UGx}) and the assumption that $\lambda \geq \lambda_0$, along with using Lemma \ref{Urel}, we obtain
\begin{equation} \label{Ulambdax}
\begin{aligned}
& \sum_{k=1}^\infty 2^{-k(s-\theta)} \left ( \dashint_{\mathcal{B}_{2^k \frac{3\sqrt{n}}{2}r}(x_0)} U^2 d\mu \right )^\frac{1}{2} \\ \leq & C_\alpha \sum_{k=1}^{l-1} 2^{-k(s-\theta)} \left ( \dashint_{\mathcal{B}_{2^k \frac{3\sqrt{n}}{2} r}(x_0)} U_\alpha^2 d\mu_\alpha \right )^\frac{1}{2} + \sum_{k=l}^\infty 2^{-k(s-\theta)} \left ( \dashint_{\mathcal{B}_{2^k \frac{3\sqrt{n}}{2}r}(x_0)} U^2 d\mu \right )^\frac{1}{2} \\
\leq & 6^{\frac{n}{2}+\theta_\alpha} C_{s,\theta}C_\alpha N_d^2 \widetilde \phi(r,x_0,y_0)^{-(\alpha+\theta_\alpha)} \lambda + {10}^{\frac{n}{2}+\theta} \widetilde \phi(r,x_0,y_0)^{-(\alpha+\theta_\alpha)}\lambda_0 \\ \leq & {10}^{n} C_{s,\theta} C_\alpha N_d^2 \widetilde \phi(r,x_0,y_0)^{-(\alpha+\theta_\alpha)} \lambda.
\end{aligned}
\end{equation}
By a similar reasoning as above, (\ref{Ut}) holds also with $U$ replaced by $G$, so that along with Lemma \ref{Urel} and H\"older's inequality, we deduce
\begin{equation} \label{Glambdax}
\begin{aligned}
& \sum_{k=1}^\infty 2^{-k(s-\theta)} \left ( \dashint_{\mathcal{B}_{2^k \frac{3\sqrt{n}}{2}r}(x_0)} G^2 d\mu \right )^\frac{1}{2} \\
\leq & C_\alpha \sum_{k=1}^{l-1} 2^{-k(s-\theta)} \left ( \dashint_{\mathcal{B}_{2^k \frac{3\sqrt{n}}{2} r}(x_0)} G_\alpha^{2} d\mu_\alpha \right )^\frac{1}{2} + \sum_{k=l}^\infty 2^{-k(s-\theta)} \left ( \dashint_{\mathcal{B}_{2^k \frac{3\sqrt{n}}{2} r}(x_0)} G^2 d\mu \right )^\frac{1}{2} \\  
\leq & C_\alpha \sum_{k=1}^{l-1} 2^{-k(s-\theta)} \left ( \dashint_{\mathcal{B}_{2^k \frac{3\sqrt{n}}{2} r}(x_0)} G_\alpha^{q_0} d\mu_\alpha \right )^\frac{1}{q_0} + \sum_{k=l}^\infty 2^{-k(s-\theta)} \left ( \dashint_{\mathcal{B}_{2^k \frac{3\sqrt{n}}{2} r}(x_0)} G^2 d\mu \right )^\frac{1}{2} \\ 
\leq & {10}^{n} C_{s,\theta} C_\alpha \widetilde \phi(r,x_0,y_0)^{-(\alpha+\theta_\alpha)} \lambda.
\end{aligned}
\end{equation}
Again, by the same arguments as above (\ref{Ulambdax}) and (\ref{Glambdax}) also hold for $x_0$ replaced by $y_0$.
	Therefore, together with the weak $\frac{q_\alpha^\star}{2}-\frac{q_\alpha^\star}{2}$ estimate for the Hardy-Littlewood maximal function, Proposition \ref{offdiagreverse} with $\gamma=\frac{1}{3\sqrt{n}}$, (\ref{Uoffdiag}), (\ref{Ulambdax}), (\ref{Gx0}), (\ref{Glambdax}) and taking into account (\ref{Neps}), we arrive at
	\begin{align*}
	& \mu_\alpha  \left (\left \{ (x,y) \in \mathcal{B}_{\frac{\sqrt{n}}{2}r}(x_0,y_0) \mid  {\mathcal{M}}_{\mathcal{B}_{5n}}( U_\alpha^2 )(x,y) > N_{\varepsilon,q}^2 \lambda^2 \right \} \right ) \\ \leq & \mu_\alpha  \left (\left \{ (x,y) \in \mathcal{B}_{\frac{ \sqrt{n}}{2}r}(x_0,y_0) \mid  {\mathcal{M}}_{\mathcal{B}_{\frac{3\sqrt{n}}{2}r}(x_0,y_0)} ( U_\alpha^2 )(x,y) > N_{\varepsilon,q}^2 \lambda^2 \right \} \right ) \\
	\leq & N_{\varepsilon,q}^{-q_\alpha^\star} \lambda^{-q_\alpha^\star} \int_{\mathcal{B}_{\frac{3\sqrt{n}}{2}r}(x_0,y_0)} U_\alpha^{q_\alpha^\star} d\mu_\alpha \\
	\leq & N_{\varepsilon,q}^{-q_\alpha^\star} \lambda^{-q\star}3^{q_\alpha^\star} C_{nd}^{q_\alpha^\star} \mu_\alpha \left(\mathcal{B}_{\frac{3\sqrt{n}}{2}r}(x_0,y_0) \right ) \Bigg [ \left (\dashint_{\mathcal{B}_{\frac{3\sqrt{n}}{2}r}(x_0,y_0)} U_\alpha^{2} d\mu_\alpha \right )^\frac{q_\alpha^\star}{2} \\
	& + \left ( \frac{3\sqrt{n}r/2}{\textnormal{dist} \left (B_{\frac{3\sqrt{n}}{2}r}(x_0),B_{\frac{3\sqrt{n}}{2}r}(y_0) \right )} \right )^{q_\alpha^\star (\alpha+\theta_\alpha)} \Bigg ( \sum_{k=1}^\infty 2^{-k(s-\theta)} \left ( \dashint_{\mathcal{B}_{2^k\frac{3\sqrt{n}}{2}r}(x_0)} U^2 d\mu \right )^\frac{1}{2} \\ & \text{ }+ \left (\dashint_{\mathcal{B}_{3 \sqrt{n} r}(x_0)} G_\alpha^{q_0} d\mu_\alpha \right )^\frac{1}{q_0} + \sum_{k=1}^\infty 2^{-k(s-\theta)} \left ( \dashint_{\mathcal{B}_{2^k\frac{3\sqrt{n}}{2}r}(x_0)} G^2 d\mu \right )^\frac{1}{2} \Bigg )^{q_\alpha^\star} \\
	& + \left ( \frac{3\sqrt{n}r/2}{\textnormal{dist} \left (B_{\frac{3\sqrt{n}}{2}r}(x_0),B_{\frac{3\sqrt{n}}{2}r}(y_0) \right )} \right )^{q_\alpha^\star (\alpha+\theta_\alpha)} \Bigg ( \sum_{k=1}^\infty 2^{-k(s-\theta)} \left ( \dashint_{\mathcal{B}_{2^k \frac{3\sqrt{n}}{2}r}(y_0)} U^2 d\mu \right )^\frac{1}{2} \\ & \text{ }+ \left (\dashint_{\mathcal{B}_{3\sqrt{n}r}(y_0)} G_\alpha^{q_0} d\mu_\alpha \right )^\frac{1}{q_0} + \sum_{k=1}^\infty 2^{-k(s-\theta)} \left ( \dashint_{\mathcal{B}_{2^k \frac{3\sqrt{n}}{2}r}(y_0)} G^2 d\mu \right )^\frac{1}{2} \Bigg )^{q_\alpha^\star} \Bigg ] \\
	\leq & N_{\varepsilon,q}^{-q\star} \lambda^{-q\star} 3^{q_\alpha^\star} C_{nd}^{q_\alpha^\star} \left (3\sqrt{n} \right )^{n+2\theta_\alpha} \mu_\alpha \left(\mathcal{B}_{\frac{r}{2}}(x_0,y_0) \right ) \bigg (3^{(\frac{n}{2}+\theta_\alpha)q_\alpha^\star} \lambda^{q_\alpha^\star} \\
	& +6^{q_\alpha^\star} (9n)^{q_\alpha^\star(\alpha+\theta_\alpha)} \widetilde \phi(r,x_0,y_0)^{q_\alpha^\star (\alpha+\theta_\alpha)} {10}^{n q_\alpha^\star} C_{s,\theta}^{q_\alpha^\star} C_\alpha^{q_\alpha^\star} N_d^{q_\alpha^\star} \widetilde \phi(r,x_0,y_0)^{-q_\alpha^\star(\alpha+\theta_\alpha)} \lambda^{q_\alpha^\star} \bigg ) \\
	< & \varepsilon \mu_\alpha \left(\mathcal{B}_{\frac{r}{2}}(x_0,y_0) \right ),
	\end{align*}
	which contradicts (\ref{ccllz}) and thus finishes the proof.
\end{proof}

Next, we restate the previous Lemma in terms of cubes instead of balls, which is vital in order to make it applicable in the context of Calder\'on-Zygmund cube decompositions as used in the covering argument in \cite[Section 7]{MeV}. In analogy to the quantity $\widetilde \phi(r,x_0,y_0)$ defined in (\ref{distfct}), for any $r \in \left (0,\frac{\sqrt{n}}{2} \right )$ and all $x_0,y_0 \in \mathbb{R}^n$ with $|x_0-y_0|>\sqrt{n}r$, we define the quantity
\begin{equation} \label{phidef}
\phi(r,x_0,y_0):=\frac{r}{\textnormal{dist}(Q_r(x_0),Q_r(y_0))}.
\end{equation}

Since the proof of the following result works almost exactly like the one in \cite[Corollary 6.4]{MeV} by using our Lemma \ref{mfusenondiag} instead of \cite[Lemma 6.3]{MeV} and by replacing in \cite{MeV} the measure $\mu$ by $\mu_\alpha$, the function $U$ by $U_\alpha$ and the parameters $s$ and $\theta$ by $\alpha$ and $\theta_\alpha$, respectively, we omit the proof and instead refer to \cite[Corollary 6.4]{MeV}.

\begin{cor} \label{mfusenondiagcubes}
	For any $\lambda \geq \lambda_0$, $r \in \left (0,\frac{\sqrt{n}}{2} \right )$ and any point $(x_0,y_0) \in \mathcal{Q}_1$ satisfying $|x_0-y_0| \geq (3\sqrt{n}+1)r$ and
	\begin{equation} \label{ccllx64}
	\mu_\alpha \left ( \left \{(x,y) \in \mathcal{Q}_{r}(x_0,y_0) \mid  {\mathcal{M}}_{\mathcal{B}_{5n}}(U_\alpha^2)(x,y) > N_{\varepsilon,q}^2 \lambda^2 \right \} \right ) > \varepsilon \mu_\alpha(\mathcal{Q}_r(x_0,y_0)),
	\end{equation}
	we have
	\begin{align*}
	& \mu_\alpha (\mathcal{Q}_{r}(x_0,y_0)) \\
	\leq & (\sqrt{n})^{n+2\theta_\alpha} \bigg( \mu_\alpha \left (\left \{ (x,y) \in \mathcal{Q}_r(x_0,y_0) \mid  {\mathcal{M}}_{\mathcal{B}_{5n}} (U_\alpha^2)(x,y) > \lambda^2 \right \} \right ) \\ 
	& + \phi(r,x_0,y_0)^{n-2 \theta_\alpha} \mu_\alpha \left (\left \{ (x,y) \in \mathcal{Q}_r(x_0) \mid  {\mathcal{M}}_{\mathcal{B}_{5n}} (U_\alpha^2)(x,y) > N_d^2 \phi(r,x_0,y_0)^{-2(\theta_\alpha+\alpha)} \lambda^2 \right \} \right ) \\
	& + \phi(r,x_0,y_0)^{n-2 \theta_\alpha} \mu_\alpha \left (\left \{ (x,y) \in \mathcal{Q}_r(y_0) \mid  {\mathcal{M}}_{\mathcal{B}_{5n}} (U_\alpha^2)(x,y) > N_d^2 \phi(r,x_0,y_0)^{-2(\theta_\alpha+\alpha)} \lambda^2 \right \} \right ) \\
	& + \phi(r,x_0,y_0)^{n-2 \theta_\alpha} \mu_\alpha \left (\left \{ (x,y) \in \mathcal{Q}_r(x_0) \mid  {\mathcal{M}}_{\mathcal{B}_{5n}} (G_\alpha^{q_0})(x,y) > \phi(r,x_0,y_0)^{-{q_0}(\theta_\alpha+\alpha)} \lambda^{q_0} \right \} \right ) \\
	& + \phi(r,x_0,y_0)^{n-2 \theta_\alpha} \mu_\alpha \left (\left \{ (x,y) \in \mathcal{Q}_r(y_0) \mid  {\mathcal{M}}_{\mathcal{B}_{5n}} (G_\alpha^{q_0})(x,y) > \phi(r,x_0,y_0)^{-{q_0}(\theta_\alpha+\alpha)}\lambda^{q_0} \right \} \right ) \bigg ).
	\end{align*}
\end{cor}

\section{Level set estimates} \label{cover}
By combining the good-$\lambda$ inequalities given by Lemma \ref{mfuse} and Corollary \ref{mfusenondiagcubes} with a technically involved covering argument, it is possible to deduce a level set estimate which will then imply the desired $L^p$ estimate for $U_\alpha$ with respect to $\mu_\alpha$. Since the mentioned covering argument was already implemented in great detail in \cite[Section 7]{MeV} and up to some minor straightforward adjustments the argument needed in our setting works exactly like the one in \cite[Section 7]{MeV}, we omit most of the technical details leading to this level set estimate. \par More precisely, in the arguments of \cite[Section 7]{MeV}, we need to replace the ball $\mathcal{B}_{4n}$ by $\mathcal{B}_{5n}$, the function $U$ by $U_\alpha$, the measure $\mu$ by $\mu_\alpha$ and the parameters $s$ and $\theta$ by $\alpha$ and $\theta_\alpha$, respectively, while the good-$\lambda$ inequatilities given by \cite[Lemma 6.1, Corollary 6.4]{MeV} need to be replaced by our corresponding good-$\lambda$ inequatilities given by Lemma \ref{mfuse} and Corollary \ref{mfusenondiagcubes}. If in addition, we take into account our different definition (\ref{tailcontrol}) of the number $\lambda_0$ in comparison to \cite[Formula (5.10)]{MeV}, we arrive at the following level set estimate, which corresponds to \cite[Corollary 7.8]{MeV}.
\begin{prop} \label{levelsetest1}
	Assume that the estimate (\ref{hdest}) is satisfied in any ball contained in $B_{5n}$ with respect to $\alpha$ and that the estimate (\ref{J}) is satisfied in any ball contained in $B_{5n}$ with respect to $q$. Then there exists some $\varepsilon_0=\varepsilon_0(n,\theta_\alpha) \in (0,1)$, such that the following is true. Let $\varepsilon \in (0,\varepsilon_0]$ and let $\delta=\delta(\varepsilon,n,s,\alpha,\theta,\Lambda,m)>0$ be given by Lemma \ref{mfuse}. Then after choosing the number $M_0=M_0(n,\theta_\alpha)>0$ in (\ref{tailcontrol}) large enough, for any $\lambda \geq \lambda_0$ we have
	\begin{align*}
		& \mu_\alpha \left (\left \{(x,y) \in \mathcal{Q}_{1} \mid  {\mathcal{M}}_{\mathcal{B}_{5n}} (U_\alpha^2)(x,y) > N_{\varepsilon,q}^2 \lambda^2 \right \} \right ) \\ \leq & C \left (\frac{\varepsilon}{\lambda^2} \int_{\mathcal{Q}_1 \cap \left \{ {\mathcal{M}}_{\mathcal{B}_{5n}} (U_\alpha^2) > \lambda^2 \right \}} {\mathcal{M}}_{\mathcal{B}_{5n}} (U_\alpha^2 ) d\mu_\alpha + \frac{1}{\delta^{q_0} \lambda^{q_0}} \int_{\mathcal{Q}_1 \cap \left \{ {\mathcal{M}}_{\mathcal{B}_{5n}} (G_\alpha^{q_0}) > \delta^{q_0} \lambda^{q_0} \right \}} {\mathcal{M}}_{\mathcal{B}_{5n}} (G_\alpha^{q_0} ) d\mu_\alpha \right ),
	\end{align*}
	where $C=C(n,\alpha,\theta_\alpha)>0$.
\end{prop}
We remark that the number $\varepsilon_0$ arises from the restriction \cite[Formula (7.26)]{MeV} adapted to our setting, that is, we have
$$\varepsilon_0=\frac{1}{4 (\sqrt{n})^{n+2\theta_\alpha}c}$$ for some $c=c(n,\theta_\alpha) \geq 1$. In addition, the number $M_0$ needs to be chosen large enough in \cite[Formula (7.7)]{MeV} adapted to our setting. More precisely, in our setting \cite[Formula (7.7)]{MeV} needs to be replaced by
	\begin{align*}
		& \mu_\alpha \left (\left \{(x,y) \in \mathcal{Q}_1 \mid  {\mathcal{M}}_{\mathcal{B}_{5n}}(U_\alpha^2)(x,y) > N_d^2 \lambda^2 \right \} \right ) \\
		& + \mu_\alpha \left (\left \{(x,y) \in \mathcal{Q}_1 \mid  {\mathcal{M}}_{\mathcal{B}_{5n}} \left (G_\alpha^{q_0} \right )(x,y) > \lambda^{q_0} \right \} \right ) \\
		\leq & \frac{C_1}{N_{d}^2 \lambda^2} \int_{\mathcal{B}_{5n}} U_\alpha^2 d\mu_\alpha + \frac{C_1}{\lambda^{q_0}} \int_{\mathcal{B}_{5n}} G_\alpha^{q_0} d\mu_\alpha \\
		\leq & \frac{C_1}{N_{d}^2 \lambda_0^2} \int_{\mathcal{B}_{5n}} U_\alpha^2 d\mu_\alpha + \frac{C_1}{\lambda_0^{q_0}} \dashint_{\mathcal{B}_{5n}} G_\alpha^{q_0} d\mu_\alpha < \kappa \varepsilon \mu_\alpha(\mathcal{Q}_1),
	\end{align*}
where all constants depend only on $n,s,\alpha$ and $\theta$ and the last inequality is obtained by choosing $M_0$ large enough in (\ref{tailcontrol}) and taking account our definition (\ref{tailcontrol}) of $\lambda_0$.

\section{A priori estimates} \label{ape}
In order to establish a priori estimates for weak solutions to the equation $L_{A} u = (-\Delta)^s g$, we need the following standard alternative characterization of the $L^p$ norm which follows from Fubini's theorem in a straightforward way.
\begin{lem} \label{altchar}
	Let $\nu$ be a $\sigma$-finite measure on $\mathbb{R}^n$ and let $h:\Omega \rightarrow [0,+\infty]$ be a $\nu$-measurable function in a domain $\Omega \subset \mathbb{R}^n$. Then for any $0< \beta < \infty$, we have 
	$$ \int_{\Omega} h^\beta d\nu = \beta \int_0^{\infty} \lambda^{\beta-1} \nu \left ( \{x \in \Omega \mid h(x)>\lambda \} \right ) d\lambda.$$
\end{lem}
\begin{prop} \label{aprioriest}
	Let $q \in [2,p)$ and $\widetilde q \in (q_0,q_\alpha^\star)$, where $q_0$ is given by (\ref{q0}). Then there exists some small enough $\delta = \delta(n,s,\alpha,\theta,\Lambda,m,q,\widetilde q) > 0$ such that if $A \in \mathcal{L}_0(\Lambda)$ is $\delta$-vanishing in $\mathcal{B}_{5n}$ and $g \in W^{s,2}(\mathbb{R}^n)$ satisfies $G_\alpha \in L^{\widetilde q}(\mathcal{B}_{5n},\mu_\alpha)$, then for any weak solution $u \in W^{s,2}(\mathbb{R}^n)$ of the equation $L_{A} u = (-\Delta)^s g$ in $B_{5n}$ that satisfies $U_\alpha \in L^{\widetilde q}(\mathcal{B}_{5n},\mu_\alpha)$, the estimate (\ref{hdest}) in any ball contained in $B_{5n}$ with respect to $\alpha$ and (\ref{J}) in any ball contained in $B_{5n}$ with respect to $q$, we have
\begin{align*}
\left (\dashint_{\mathcal{B}_{1/2}} U_\alpha^{\widetilde q} d\mu_\alpha \right )^{\frac{1}{\widetilde q}} \leq & C \Bigg (\sum_{k=1}^\infty 2^{-k(s-\theta)} \left ( \dashint_{\mathcal{B}_{2^k 5n}} U^2 d\mu \right )^\frac{1}{2} \\
	& + \left ( \dashint_{\mathcal{B}_{5n}} G_\alpha^{\widetilde q} d\mu_\alpha \right )^\frac{1}{\widetilde q} + \sum_{k=1}^\infty 2^{-k(s-\theta)} \left ( \dashint_{\mathcal{B}_{2^k 5n}} G^2 d\mu \right )^\frac{1}{2} \Bigg ),
	\end{align*}
	where $C=C(n,s,\alpha,\theta,\Lambda,m,q,\widetilde q,p)>0$.
\end{prop}

\begin{proof}
Let $\varepsilon$ to be chosen small enough and consider the corresponding $\delta = \delta(\varepsilon,n,s,\alpha,\theta,\Lambda,m) > 0$ given by Lemma \ref{mfuse}. Then by using Lemma \ref{altchar} multiple times, first with $\beta=\widetilde q$, $h={\mathcal{M}}_{\mathcal{B}_{5n}} (U_\alpha^2 )^\frac{1}{2}$ and $d\nu=d\mu_\alpha$, then with $\beta=\widetilde q-2$, $h={\mathcal{M}}_{\mathcal{B}_{5n}} (U_\alpha^2 )^\frac{1}{2}$ and $d\nu={\mathcal{M}}_{\mathcal{B}_{5n}} (U_\alpha^2 )d\mu_\alpha$, and also with $\beta=\widetilde q-q_0$,  $h={\mathcal{M}}_{\mathcal{B}_{5n}} (G_\alpha^{q_0} )^\frac{1}{q_0}$ and $d\nu={\mathcal{M}}_{\mathcal{B}_{5n}} (G_\alpha^{q_0} )d\mu_\alpha$, a change of variables, Proposition \ref{levelsetest1} and the definition of $N_{\varepsilon,q}$ from (\ref{Neps}), we obtain
\begin{align*}
& \int_{\mathcal{Q}_1} \left ({\mathcal{M}}_{\mathcal{B}_{5n}} (U_\alpha^2 ) \right )^\frac{\widetilde q}{2} d\mu_\alpha \\ = & \widetilde q \int_0^{\infty} \lambda^{\widetilde q-1} \mu_\alpha \left ( \mathcal{Q}_1 \cap \left \{{\mathcal{M}}_{\mathcal{B}_{5n}} (U_\alpha^2 ) > \lambda^2 \right \} \right ) d\lambda \\
= & \widetilde q N_{\varepsilon,q}^{\widetilde q} \int_0^{\infty} \lambda^{\widetilde q-1} \mu_\alpha \left ( \mathcal{Q}_1 \cap \left \{{\mathcal{M}}_{\mathcal{B}_{5n}} (U_\alpha^2 ) > N_{\varepsilon,q}^2 \lambda^2 \right \} \right ) d\lambda \\
= & \widetilde q N_{\varepsilon,q}^{\widetilde q} \int_0^{\lambda_0} \lambda^{\widetilde q-1} \mu_\alpha \left ( \mathcal{Q}_1 \cap \left \{{\mathcal{M}}_{\mathcal{B}_{5n}} (U_\alpha^2 ) > N_{\varepsilon,q}^2 \lambda^2 \right \} \right ) d\lambda \\ & + \widetilde q N_{\varepsilon,q}^{\widetilde q} \int_{\lambda_0}^{\infty} \lambda^{\widetilde q-1} \mu_\alpha \left ( \mathcal{Q}_1 \cap \left \{{\mathcal{M}}_{\mathcal{B}_{5n}} (U_\alpha^2 ) > N_{\varepsilon,q}^2 \lambda^2 \right \} \right ) d\lambda \\
\leq & \widetilde q N_{\varepsilon,q}^{\widetilde q} \mu_\alpha(\mathcal{Q}_1) \lambda_0^{\widetilde q} \\
&+ C_1 \widetilde q N_{\varepsilon,q}^{\widetilde q} \varepsilon \int_{0}^{\infty} \lambda^{\widetilde q-3} \int_{\mathcal{Q}_1 \cap \left \{ {\mathcal{M}}_{\mathcal{B}_{5n}} (U_\alpha^2) > \lambda^2 \right \}} {\mathcal{M}}_{\mathcal{B}_{5n}} (U_\alpha^2 ) d\mu_\alpha d\lambda \\
&+ C_1 \widetilde q N_{\varepsilon,q}^{\widetilde q}\delta^{-q_0} \int_{0}^{\infty} \lambda^{\widetilde q-q_0-1} \int_{\mathcal{Q}_1 \cap \left \{ {\mathcal{M}}_{\mathcal{B}_{5n}} (G_\alpha^{q_0}) > \delta^{q_0} \lambda^{q_0} \right \}} {\mathcal{M}}_{\mathcal{B}_{5n}} (G_\alpha^{q_0} ) d\mu_\alpha d\lambda \\
= & \widetilde q N_{\varepsilon,q}^{\widetilde q} \mu_\alpha(\mathcal{Q}_1) \lambda_0^{\widetilde q} \\
&+ C_1 \widetilde q C_{nd} C_{s,\theta} N_d 10^{10n} \varepsilon^{1-\widetilde q/q_\alpha^\star} \int_{\mathcal{Q}_1} \left ({\mathcal{M}}_{\mathcal{B}_{5n}} (U_\alpha^2 ) \right )^\frac{\widetilde q}{2} d\mu_\alpha \\
&+ C_1 \widetilde q N_{\varepsilon,q}^{\widetilde q} \delta^{-q_0} \int_{\mathcal{Q}_1} \left ({\mathcal{M}}_{\mathcal{B}_{5n}} (G_\alpha^{q_0} ) \right )^\frac{\widetilde q}{q_0} d\mu_\alpha,
\end{align*}
where $C_1=C_1(n,s,\alpha,\theta) \geq 1$. Next, we set
$$ \varepsilon := \min \left \{ \varepsilon_0, \left (2 C_1 \widetilde q C_{nd} C_{s,\theta} C_\alpha N_d 10^{10n} \right )^{-\frac{q_\alpha^\star}{q_\alpha^\star-\widetilde q}} \right \}, $$
so that $\varepsilon$ is a valid choice in Proposition \ref{levelsetest1} and moreover, we have
$$ C_1 \widetilde q C_{nd} C_{s,\theta} C_\alpha N_d 10^{10n} \varepsilon^{1-\widetilde q/q_\alpha^\star} \leq \frac{1}{2}.$$ Since in addition by assumption we have $U_\alpha \in L^{\widetilde q}(\mathcal{B}_{5n},\mu_\alpha)$, by Proposition \ref{Maxfun} we have $$\int_{\mathcal{Q}_1} \left ({\mathcal{M}}_{\mathcal{B}_{5n}} (U_\alpha^2 ) \right )^\frac{\widetilde q}{2} d\mu_\alpha < \infty,$$ so that we can reabsorb the second to last term on the right-hand side of the first display of the proof in the the left-hand side, which yields
\begin{align*}
\int_{\mathcal{Q}_1} \left ({\mathcal{M}}_{\mathcal{B}_{5n}} (U_\alpha^2 ) \right )^\frac{\widetilde q}{2} d\mu_\alpha
\leq & 2 \widetilde q N_{\varepsilon,q}^{\widetilde q} \mu_\alpha(\mathcal{Q}_1) \lambda_0^{\widetilde q}+ 2 C_1 \widetilde q N_{\varepsilon,q}^{\widetilde q} \delta^{-q_0} \int_{\mathcal{Q}_1} \left ({\mathcal{M}}_{\mathcal{B}_{5n}} (G_\alpha^{q_0} ) \right )^\frac{\widetilde q}{q_0} d\mu_\alpha.
\end{align*}
Now in view of Proposition \ref{maxdominate} and Proposition \ref{Maxfun}, taking into account the definition of $\lambda_0$ from (\ref{tailcontrol}) along with using the estimate (\ref{hdest}) with $u_0=u$ and H\"older's inequality, we obtain
\begin{align*}
& \dashint_{\mathcal{B}_{1/2}} U_\alpha^{\widetilde q} d\mu_\alpha
\leq \frac{1}{\mu_\alpha\left (\mathcal{B}_{1/2} \right )} \int_{\mathcal{Q}_1} \left ({\mathcal{M}}_{\mathcal{B}_{5n}} (U_\alpha^2 ) \right )^\frac{\widetilde q}{2} d\mu_\alpha \\ 
\leq & C_2 \left ( \lambda_0^{\widetilde q}+ \int_{\mathcal{Q}_1} \left ({\mathcal{M}}_{\mathcal{B}_{5n}} (G_\alpha^{q_0} ) \right )^\frac{\widetilde q}{q_0} d\mu_\alpha \right ) \\
\leq & C_3 \Bigg (\sum_{k=1}^\infty 2^{-k(s-\theta)} \left (\dashint_{\mathcal{B}_{2^k 5n}} U^2 d\mu \right )^\frac{1}{2} + \sum_{k=1}^\infty 2^{-k(s-\theta)} \left (\dashint_{\mathcal{B}_{2^k 5n}} G^2 d\mu \right )^\frac{1}{2} + \left (\dashint_{\mathcal{B}_{5n}} U_\alpha^2 d\mu_\alpha \right )^\frac{1}{2} \\ & +\left (\dashint_{\mathcal{B}_{5n}} G_\alpha^{q_0} d\mu_\alpha \right )^\frac{1}{q_0} \Bigg )^{\widetilde q}
+ C_3 \int_{\mathcal{B}_{5n}} G_\alpha^{\widetilde q} d\mu_\alpha \\
\leq & C_4 \Bigg (\sum_{k=1}^\infty 2^{-k(s-\theta)} \left (\dashint_{\mathcal{B}_{2^k 5n}} U^2 d\mu \right )^\frac{1}{2} + \sum_{k=1}^\infty 2^{-k(s-\theta)} \left (\dashint_{\mathcal{B}_{2^k 5n}} G^2 d\mu \right )^\frac{1}{2} + \left (\dashint_{\mathcal{B}_{5n}} G_\alpha^{m} d\mu_\alpha \right )^\frac{1}{m} \\ & +\left (\dashint_{\mathcal{B}_{5n}} G_\alpha^{q_0} d\mu_\alpha \right )^\frac{1}{q_0} \Bigg )^{\widetilde q}
+ C_4 \int_{\mathcal{B}_{5n}} G_\alpha^{\widetilde q} d\mu_\alpha \\
\leq & C_5 \Bigg (\sum_{k=1}^\infty 2^{-k(s-\theta)} \left (\dashint_{\mathcal{B}_{2^k 5n}} U^2 d\mu \right )^\frac{1}{2} + \sum_{k=1}^\infty 2^{-k(s-\theta)} \left (\dashint_{\mathcal{B}_{2^k 5n}} G^2 d\mu \right )^\frac{1}{2} \Bigg )^{\widetilde q} + C_5 \dashint_{\mathcal{B}_{5n}} G_\alpha^{\widetilde q} d\mu_\alpha,
\end{align*}
where we also used that $m \leq q_0 \leq \widetilde q$ and all constants depend only on $n,s,\alpha,\theta,\Lambda,m,q,\widetilde q$ and $p$. This proves the desired estimate with $C=C_5^{1/\widetilde q}$.
\end{proof}

\begin{cor} \label{aprioriestcor}
	Consider some $q \in [2,p)$ and some $\widetilde q \in (q_0,q_\alpha^\star)$. Then there exists some small enough $\delta = \delta(n,s,\alpha,\theta,\Lambda,m,q,\widetilde q) > 0$ such that if $A \in \mathcal{L}_0(\Lambda)$ is $\delta$-vanishing in $B_1$ and $g \in W^{s,2}(\mathbb{R}^n)$ satisfies $G_\alpha \in L^{\widetilde q}(\mathcal{B}_1,\mu_\alpha)$, then for any weak solution $u \in W^{s,2}(\mathbb{R}^n)$ of the equation $L_{A} u = (-\Delta)^s g$ in $B_1$ that satisfies $U_\alpha \in L^{\widetilde q}(\mathcal{B}_1,\mu_\alpha)$, the estimate (\ref{hdest}) in any ball contained in $B_1$ with respect to $\alpha$ and (\ref{J}) in any ball contained in $B_1$ with respect to $q$, we have the estimate
	\begin{equation} \label{1scale}
	\begin{aligned}
	\left (\dashint_{\mathcal{B}_{1/2}} U_\alpha^{\widetilde q} d\mu_\alpha \right )^{\frac{1}{\widetilde q}} \leq & C \Bigg (\sum_{k=1}^\infty 2^{-k(s-\theta)} \left ( \dashint_{\mathcal{B}_{2^k}} U^2 d\mu \right )^\frac{1}{2} \\
	& + \left ( \dashint_{\mathcal{B}_{1}} G_\alpha^{\widetilde q} d\mu_\alpha \right )^\frac{1}{\widetilde q} + \sum_{k=1}^\infty 2^{-k(s-\theta)} \left ( \dashint_{\mathcal{B}_{2^k}} G^2 d\mu \right )^\frac{1}{2} \Bigg ),
	\end{aligned}
	\end{equation}
	where $C=C(n,s,\alpha,\theta,\Lambda,m,q,\widetilde q,p)>0$.
\end{cor}

\begin{proof}
	There exists some small enough $r_1 \in \left (0,1 \right)$ such that for any $z \in B_{1/2}$, we have
	\begin{equation} \label{rz}
	B_{5nr_1}(z) \Subset B_1.
	\end{equation}
	Fix some $z \in B_{1/2}$ and consider the scaled functions $u_z,g_z \in W^{s,2}(\mathbb{R}^n)$ given by
	$$ u_z(x):=u(r_1 x+z), \quad g_z(x):=g(r_1 x+z) $$
	and also
	$$A_z(x,y):= A(r_1 x+z,r_1 y+z).$$
	Since $A$ is $\delta$-vanishing in $\mathcal{B}_1$, we see that $A_z$ clearly is $\delta$-vanishing in $B_{\frac{1}{5nr}}(-z) \supset B_{5n}$.
	Furthermore, in view of (\ref{rz}), $u_z$ is a weak solution of $L_{A_z} u_z = g_z$ in $B_{\frac{1}{5nr_1}}(-z) \supset B_{5n}$. Now fix some $r>0$ and some $x_0 \in \mathbb{R}^n$ such that $B_r(x_0) \subset B_{5n}$. Then again in view of (\ref{rz}), we clearly have $$B_{r_1 r}(r_1 x_0+z) \subset B_1,$$ so that by the assumption that the estimate (\ref{J}) holds for any ball contained in $B_1$, the estimate (\ref{J}) holds with respect to the ball $B_{r_1 r}(r_1 x_0+z)$. Together with changes of variables and taking into account (\ref{modtheta}), by straightforward computations similar to \cite[Formula (8.2)]{MeV} it is now easy to verify that the functions
	\begin{align*}
	(U_\alpha)_z(x,y):=\frac{|u_z(x)-u_z(y)|}{|x-y|^{\alpha+\theta_\alpha}}, \quad (G_\alpha)_z(x,y):=\frac{|g_z(x)-g_z(y)|}{|x-y|^{\alpha+\theta_\alpha}},
	\end{align*}
	\begin{align*}
	U_z(x,y):=\frac{|u_z(x)-u_z(y)|}{|x-y|^{s+\theta}}, \quad G_z(x,y):=\frac{|g_z(x)-g_z(y)|}{|x-y|^{s+\theta}}
	\end{align*}
	satisfy the estimate (\ref{hdest}) in any ball contained in $B_{5n}$ with respect to $\alpha$ and the estimate (\ref{J}) in any ball contained in $B_{5n}$ with respect to $q$. Since in addition the assumption that $U_\alpha \in L^{\widetilde q}(\mathcal{B}_1,\mu_\alpha)$ clearly implies that $(U_\alpha)_z \in L^{\widetilde q} \left (\mathcal{B}_{{\frac{1}{5nr_1}}(-z)},\mu_\alpha \right ) \subset L^{\widetilde q}\left (\mathcal{B}_{5n} ,\mu_\alpha \right )$, by Proposition \ref{aprioriest} we obtain that
	\begin{align*}
	\left (\dashint_{\mathcal{B}_{1/2}} (U_\alpha)_z^{\widetilde q} d\mu_\alpha \right )^{\frac{1}{\widetilde q}} \leq & C_4 \Bigg (\sum_{k=1}^\infty 2^{-k(s-\theta)} \left ( \dashint_{\mathcal{B}_{2^k 5n}} U_z^2 d\mu \right )^\frac{1}{2} \\
	& + \left ( \dashint_{\mathcal{B}_{5n}} (G_\alpha)_z^{\widetilde q} d\mu_\alpha \right )^\frac{1}{\widetilde q} + \sum_{k=1}^\infty 2^{-k(s-\theta)} \left ( \dashint_{\mathcal{B}_{2^k 5n}} G_z^2 d\mu \right )^\frac{1}{2} \Bigg ),
	\end{align*}
	where $C_4=C_4(n,s,\alpha,\theta,\Lambda,m,q,\widetilde q,p)>0$. 
	By combining the last display with another straightforward computation involving changes of variables (cf.\ \cite[Formula (8.3)]{MeV}), we obtain
	\begin{align*}
		\left (\dashint_{\mathcal{B}_{r_1/2}} U_\alpha^{\widetilde q} d\mu_\alpha \right )^{\frac{1}{\widetilde q}} \leq & C_5 \Bigg (\sum_{k=1}^\infty 2^{-k(s-\theta)} \left ( \dashint_{\mathcal{B}_{2^k 5nr_1(z)}} U^2 d\mu \right )^\frac{1}{2} \\
		& + \left ( \dashint_{\mathcal{B}_{5nr_1(z)}} G_\alpha^{\widetilde q} d\mu_\alpha \right )^\frac{1}{\widetilde q} + \sum_{k=1}^\infty 2^{-k(s-\theta)} \left ( \dashint_{\mathcal{B}_{2^k 5nr_1(z)}} G^2 d\mu \right )^\frac{1}{2} \Bigg ),
	\end{align*}
	 where again $C_5=C_5(n,s,\alpha,\theta,\Lambda,m,q,\widetilde q,p)>0$.
	 Since $\left \{B_{r_1/2}(z) \right \}_{z \in B_{1/2}}$ is an open covering of the compact set $\overline B_{1/2}$, there is a finite subcover $\left \{B_{r_1/2}(z_j) \right \}_{j=1}^N$ of $B_{1/2}$. Thus, summing up the above estimates applied with $z=z_j$ over $j=1,...,N$ in essentially the same way as in the last display in the proof of \cite[Corollary 8.3]{MeV} yields the estimate (\ref{1scale}), which finishes the proof.
\end{proof}

In view of another straightforward scaling argument (cf. \cite[Corollary 8.4]{MeV}), we also have the following scaled version of Corollary \ref{aprioriestcor}.
\begin{cor} \label{aprioriestcorscaled}
	Let $r>0$ and $z \in \mathbb{R}^n$ and consider some $q \in [2,p)$ and some $\widetilde q \in (q_0,q_\alpha^\star)$. Then there exists some small enough $\delta = \delta(n,s,\alpha,\theta,\Lambda,m,q,\widetilde q) > 0$ such that if $A \in \mathcal{L}_0(\Lambda)$ is $\delta$-vanishing in $\mathcal{B}_r(z)$ and $g \in W^{s,2}(\mathbb{R}^n)$ satisfies $G_\alpha \in L^{\widetilde q}(\mathcal{B}_r(z),\mu_\alpha)$, then for any weak solution $u \in W^{s,2}(\mathbb{R}^n)$ of the equation $L_{A} u = (-\Delta)^s g$ in $B_r(z)$ that satisfies $U_\alpha \in L^{\widetilde q}(\mathcal{B}_r(z),\mu_\alpha)$, the estimate (\ref{hdest}) in any ball contained in $B_r(z)$ with respect to $\alpha$ and (\ref{J}) in any ball contained in $B_r(z)$ with respect to $q$, we have the estimate
	\begin{align*}
	\left (\dashint_{\mathcal{B}_{r/2}(z)} U_\alpha^{\widetilde q} d\mu_\alpha \right )^{\frac{1}{\widetilde q}} \leq & C \Bigg (\sum_{k=1}^\infty 2^{-k(s-\theta)} \left ( \dashint_{\mathcal{B}_{2^k r(z)}} U^2 d\mu \right )^\frac{1}{2} \\
	& + \left ( \dashint_{\mathcal{B}_{r}(z)} G_\alpha^{\widetilde q} d\mu_\alpha \right )^\frac{1}{\widetilde q} + \sum_{k=1}^\infty 2^{-k(s-\theta)} \left ( \dashint_{\mathcal{B}_{2^k r}(z)} G^2 d\mu \right )^\frac{1}{2} \Bigg ),
	\end{align*}
	where $C=C(n,s,\alpha,\theta,\Lambda,m,q,\widetilde q,p)>0$.
\end{cor}
Next, we use an iteration argument in order to drop the assumption (\ref{J}) and obtain higher integrability all the way up to the exponent $p$.
\begin{prop} \label{aprioriestx}
	Let $r>0$, $z \in \mathbb{R}^n$, $s \in (0,1)$ and $p \in (m,\infty)$, where $m$ satisfies (\ref{mdef}). Then there exists some small enough $\delta = \delta(p,n,s,\alpha,\theta,\Lambda,m) > 0$ such that if $A \in \mathcal{L}_0(\Lambda)$ is $\delta$-vanishing in $\mathcal{B}_{r}(z)$ and $g \in W^{s,2}(\mathbb{R}^n)$ satisfies $G \in L^{p}(\mathcal{B}_{r}(z),\mu)$, then for any weak solution $u \in W^{s,2}(\mathbb{R}^n)$ of the equation $L_{A} u = (-\Delta)^s g$ in $B_{r}(z)$ that satisfies $U_\alpha \in L^{p}(\mathcal{B}_{r}(z),\mu_\alpha)$ and the estimate (\ref{hdest}) in any ball contained in $B_r(z)$, we have
	\begin{equation} \label{estp}
	\begin{aligned}
	\left (\dashint_{\mathcal{B}_{r/2}(z)} U_\alpha^{p} d\mu_\alpha \right )^{\frac{1}{p}} \leq & C \Bigg (\sum_{k=1}^\infty 2^{-k(s-\theta)} \left ( \dashint_{\mathcal{B}_{2^k r}(z)} U^2 d\mu \right )^\frac{1}{2} \\
	& + \left ( \dashint_{\mathcal{B}_{r}(z)} G_\alpha^{p} d\mu_\alpha \right )^\frac{1}{p} + \sum_{k=1}^\infty 2^{-k(s-\theta)} \left ( \dashint_{\mathcal{B}_{2^k r}(z)} G^2 d\mu \right )^\frac{1}{2} \Bigg ),
	\end{aligned}
	\end{equation}
	where $C=C(n,s,\alpha,\theta,\Lambda,m,p)>0$.
\end{prop}

\begin{proof}
Define iteratively a sequence $\{q_i\}_{i=1}^\infty$ of real numbers by
$$ q_1:=2, \quad q_{i+1}:= \min \{(q_i+(q_i)^\star)/2,p\},$$
where as in (\ref{qstar}) we let
\begin{align*}
(q_i)^\star=\begin{cases} 
\frac{nq_i}{n-\alpha q_i}, & \text{if } n>\alpha q_i \\
2p, & \text{if } n \leq \alpha q_i.
\end{cases}
\end{align*}
Since for any $i$ with $n>\alpha q_{i+1}$ we have
$$ \left (q_i+\frac{nq_i}{n-\alpha q_i} \right )/2 -q_i =\frac{nq_{i}}{2(n-\alpha q_{i})} - \frac{q_i}{2} \geq \frac{4s}{2(n-\alpha )}>0,$$
there clearly exists some $i_p \in \mathbb{N}$ such that $q_{i_p} = p$. \newline
Since the estimate (\ref{J}) is trivially satisfied for $q=q_1=2$, and in view of the additional assumption that $U_\alpha \in L^{p}(\mathcal{B}_{r}(z),\mu_\alpha)$ we in particular have $U_\alpha \in L^{q_1}(\mathcal{B}_{r}(z),\mu_\alpha)$, if we choose $\delta$ small enough such that Corollary \ref{aprioriestcorscaled} is applicable with $q=2$ and $\widetilde q = q_2$, then all assumptions of Corollary \ref{aprioriestcorscaled} are satisfied with respect to $q=q_1=2$ and $\widetilde q = q_2\in (\min \{m,q_1\},(q_1)^\star)$, so that we obtain
\begin{equation} \label{q1est}
\begin{aligned}
\left (\dashint_{\mathcal{B}_{r/2}(z)} U_\alpha^{q_2} d\mu_\alpha \right )^{\frac{1}{q_2}} \leq & C \Bigg (\sum_{k=1}^\infty 2^{-k(s-\theta)} \left ( \dashint_{\mathcal{B}_{2^k r(z)}} U^2 d\mu \right )^\frac{1}{2} \\
& + \left ( \dashint_{\mathcal{B}_{r}(z)} G_\alpha^{q_2} d\mu_\alpha \right )^\frac{1}{q_2} + \sum_{k=1}^\infty 2^{-k(s-\theta)} \left ( \dashint_{\mathcal{B}_{2^k r}(z)} G^2 d\mu \right )^\frac{1}{2} \Bigg ),
\end{aligned}
\end{equation}
where $C_1=C_1(n,s,\alpha,\theta,\Lambda,m,p)>0$. If $i_p=2$, then $q_2=p$ and the proof is finished. Otherwise, we observe that since $r$ and $z$ are arbitrary, the estimate (\ref{q1est}) holds also in any ball that is contained in $B_r(z)$, so that that the estimate (\ref{J}) is satisfied with respect to $q=q_2$ in any ball contained in $B_r(z)$. Since also $U_\alpha \in L^{p}(\mathcal{B}_{r}(z),\mu_\alpha) \subset L^{q_2}(\mathcal{B}_{r}(z),\mu_\alpha)$, if we choose $\delta$ smaller if necessary such that Corollary \ref{aprioriestcorscaled} is applicable with $q=q_2$ and $\widetilde q = q_3$, then all assumptions of Corollary \ref{aprioriestcorscaled} are satisfied with respect to $q=q_2$ and $\widetilde q = q_3 = (q_2+(q_2)^\star)/2 \in (q_2,(q_2)^\star)$, so that we obtain the estimate 
\begin{align*}
\left (\dashint_{\mathcal{B}_{r/2}(z)} U_\alpha^{q_3} d\mu_\alpha \right )^{\frac{1}{q_3}} \leq & C_2 \Bigg (\sum_{k=1}^\infty 2^{-k(s-\theta)} \left ( \dashint_{\mathcal{B}_{2^k r(z)}} U^2 d\mu \right )^\frac{1}{2} \\
& + \left ( \dashint_{\mathcal{B}_{r}(z)} G_\alpha^{q_3} d\mu_\alpha \right )^\frac{1}{q_3} + \sum_{k=1}^\infty 2^{-k(s-\theta)} \left ( \dashint_{\mathcal{B}_{2^k r}(z)} G^2 d\mu \right )^\frac{1}{2} \Bigg ),
\end{align*}
where $C_2=C_2(n,s,\alpha,\theta,\Lambda,m,p)>0$.
If $i_p=3$, then $q_3=p$ and the proof is finished. Otherwise, iterating this procedure $i_p-1$ times and using that $q_{i_p}=p$ also leads to the estimate (\ref{estp}).
\end{proof}
Finally, by another delicate iteration argument we also drop the assumption that the estimate (\ref{hdest}) holds, achieving an a priori higher differentiability estimate for any $s<t<\min\{2s,1\}$.
\begin{prop} \label{aprioriestxy}
	Let $r>0$, $z \in \mathbb{R}^n$, $s \in (0,1)$, $s<t<\min\{2s,1\}$ and $p \in (2,\infty)$. Then there exists some small enough $\delta = \delta(p,n,s,t,\Lambda) > 0$ such that if $A \in \mathcal{L}_0(\Lambda)$ is $\delta$-vanishing in $\mathcal{B}_{r}(z)$ and $g$ belongs to $W^{s,2}(\mathbb{R}^n) \cap W^{t,p}(B_r(z))$, then for any weak solution $u \in W^{s,2}(\mathbb{R}^n) \cap W^{t,p}(B_r(z))$ of the equation $L_{A} u = (-\Delta)^s g$ in $B_{r}(z)$, we have
	\begin{equation} \label{estpy}
			[u]_{W^{t,p}(B_{r/2}(z))} \leq C \left ( [u]_{W^{s,2}(\mathbb{R}^n)} + [g]_{W^{t,p}(B_{r}(z))} + [g]_{W^{s,2}(\mathbb{R}^n)} \right ),
	\end{equation}
	where $C=C(n,s,t,\Lambda,p,r)>0$.
\end{prop}

\begin{proof}
Fix some $s <t<\min\{2s,1\}$ and some $p \in (2,\infty)$. All constants in this proof will only depend on $n,s,t,\Lambda,p$ and $r$. First of all, the assumption that $u \in W^{t,p}(B_r(z))$ implies that $U_\alpha=U_{\alpha,\theta_\alpha} \in L^p(B_r(z),\mu_\alpha)$ for any $s \leq \alpha<\min\{2s,1\}$ such that $\alpha + \left (1-\frac{2}{p} \right )\theta_\alpha \leq t$. \par Let $\delta>0$ be to be chosen small enough, fix some $0<\gamma<\min\{2s,1\}-t$ and choose the parameter $\theta$ by $\theta:=\min\{s,1-s\}-\gamma \in (0,\min\{s,1-s\})$, so that in particular $t<s+\theta$. In addition, define sequences of parameters $\{m_k\}_{k \in \mathbb{N}}$ and $\{\varepsilon_k\}_{k \in \mathbb{N}}$ by $$m_k:=\frac{1}{k} \min \left \{\frac{2n-3s}{n-2s},1+\frac{p}{2} \right \} +\left (1-\frac{1}{k} \right ) \min \left \{\frac{2(n-s)}{n-2s},p \right \} \in \left (2, \min \left \{\frac{2(n-s)}{n-2s},p \right \} \right )$$ and $$ \varepsilon_k:=1-\frac{2}{m_k} \in (0,1).$$ In particular, note that as indicated above, for any $k \in \mathbb{N}$ the parameter $m_k$ belongs to the range given by (\ref{mdef}). Define inductively further sequences of parameters $\{t_k\}_{k \in \mathbb{N}_0}$ and $\{\theta_{t_k}\}_{k \in \mathbb{N}_0}$ by $t_0:=s$, $\theta_{t_0}:=\theta$ and
$$ t_k:=t_{k-1}+\frac{\varepsilon_k \theta_{t_{k-1}}}{2}, \quad \theta_{t_{k}}:=s+\theta-t_{k}, \quad k \geq 1.$$
Let $$\varepsilon_\star := \lim_{k \to \infty} \varepsilon_k = 1-2 / \min \left \{\frac{2(n-s)}{n-2s},p \right \}>0.$$ Since the sequence $\{t_k\}_{k \in \mathbb{N}_0}$ is strictly increasing and bounded by $s+\theta$, the limit $t_\star:= \lim_{k \to \infty} t_k$ exists and satisfies
$t_\star=t_\star+\frac{\varepsilon_\star}{2}(s+\theta-t_\star),$ which leads to $t_\star=s+\theta$.
Thus, since we have $t<s+\theta=t_\star$, there exists a non-negative integer $\widetilde k$ such that $t_{\widetilde k} < t$ but $t_{\widetilde k +1} \geq t$. Also, define
$$ \theta_t:=\frac{t-t_{\widetilde k}}{1-2/p}, \quad \widetilde \theta := \theta_t+t_{\widetilde k}-s$$
and note that since $p \geq m_{\widetilde k}$, we have $$\theta_t \leq \frac{t-t_{\widetilde k +1}+\varepsilon_{\widetilde k} \theta_{t_{\widetilde k}}}{\varepsilon_{\widetilde k}} \leq \theta_{t_{\widetilde k}}=s+\theta-t_{\widetilde k},$$
which implies that
$$ 0<\widetilde \theta \leq \theta < \min \{s,1-s\}.$$
Thus, $\widetilde \theta$ also belongs to the range (\ref{theta}) and the relation (\ref{modtheta}) is satisfied for $\theta_\alpha=\theta_t$, $\alpha=t_{\widetilde k}$ and with $\theta$ replaced by $\widetilde \theta$, that is, we have $\theta_t=s+\widetilde \theta - t_{\widetilde k}$.
In addition, observe that $$t_{\widetilde k} + \left (1-\frac{2}{p} \right )\theta_t=t.$$ \par 
If $\widetilde k=0$, then since for $\alpha=t_0=s$, the estimate (\ref{hdest}) is trivially satisfied with $m=2$, by Corollary \ref{aprioriestx} with $\theta_\alpha=\theta_t$ and with $\theta$ replaced by $\widetilde \theta$, for $\delta$ small enough we have
	\begin{align*}
		[u]_{W^{t,p}(B_{r/2}(z))} = & C_1 \left (\dashint_{\mathcal{B}_{r/2}(z)} U_{s,\theta_t}^{p} d\mu_{\theta_t} \right )^{\frac{1}{p}} \\ \leq & C_2 \Bigg (\sum_{k=1}^\infty 2^{-k(s-\widetilde \theta)} \left ( \dashint_{\mathcal{B}_{2^k r}(z)} U_{s,\widetilde \theta}^2 d\mu_{\widetilde \theta} \right )^\frac{1}{2} \\
		& + \left ( \dashint_{\mathcal{B}_{r}(z)} G_{s,\theta_t}^{p} d\mu_{\theta_t} \right )^\frac{1}{p} + \sum_{k=1}^\infty 2^{-k(s-\widetilde \theta)} \left ( \dashint_{\mathcal{B}_{2^k r}(z)} G_{s,\widetilde \theta}^2 d\mu_{\widetilde \theta} \right )^\frac{1}{2} \Bigg ) \\
		\leq & C_3 \left ([u]_{W^{s,2}(\mathbb{R}^n)} + [g]_{W^{t,p}(B_{r}(z))} + [g]_{W^{s,2}(\mathbb{R}^n)} \right ).
	\end{align*}
In the case when $\widetilde k=0$, the proof is finished. If on the other hand $\widetilde k>0$, then for any $x_0 \in B_r(z)$ and any $r^\prime>0$ such that $B_{r^\prime}(x_0) \subset B_r(z)$, using Proposition \ref{Sobcont}, Corollary \ref{aprioriestx} with $p$ replaced by $m_1$ along with Lemma \ref{Urel} yields
	\begin{align*}
	& \left (\frac{1}{\mu_{\theta_{t_1}}(\mathcal{B}_{r^\prime}(x_0))} \int_{B_{r^\prime/2}(x_0)} \int_{B_{r^\prime/2}(x_0)} \frac{(u(x)-u(y))^2}{|x-y|^{n+2t_1}}dydx \right )^\frac{1}{2} \\ \leq & C_4 (r^\prime)^{-\theta_{t_1}-\frac{n}{m_1}+\frac{\varepsilon_1 \theta}{2}} \left ( \int_{B_{r^\prime/2}(x_0)} \int_{B_{r^\prime/2}(x_0)} \frac{|u(x)-u(y)|^{m_1}}{|x-y|^{n+m_1(t_1+ \varepsilon_1 \theta/2)}}dydx \right )^\frac{1}{m_1} \\ = & C_4 (r^\prime)^{-\frac{n}{m_1}-\frac{2\theta}{m_1}} \left ( \int_{B_{r^\prime/2}(x_0)} \int_{B_{r^\prime/2}(x_0)} \frac{|u(x)-u(y)|^{m_1}}{|x-y|^{n+m_1(s+\varepsilon_1 \theta)}}dydx \right )^\frac{1}{m_1} \\ = & C_5 \left (\dashint_{\mathcal{B}_{r^\prime/2}(x_0)} U_{s,\theta}^{m_1} d\mu \right )^{\frac{1}{m_1}} \\ \leq & C_6 \Bigg (\sum_{k=1}^\infty 2^{-k(s-\theta)} \left ( \dashint_{\mathcal{B}_{2^k r^\prime}(x_0)} U^2 d\mu \right )^\frac{1}{2} \\
	& + \left ( \dashint_{\mathcal{B}_{r^\prime}(x_0)} G_{s,\theta}^{m_1} d\mu \right )^\frac{1}{m_1} + \sum_{k=1}^\infty 2^{-k(s-\theta)} \left ( \dashint_{\mathcal{B}_{2^k r^\prime}(x_0)} G^2 d\mu \right )^\frac{1}{2} \Bigg ) \\
	\leq & C_7 \Bigg (\sum_{k=1}^\infty 2^{-k(s-\theta)} \left ( \dashint_{\mathcal{B}_{2^k r^\prime}(x_0)} U^2 d\mu \right )^\frac{1}{2} \\
	& + \left ( \dashint_{\mathcal{B}_{r^\prime}(x_0)} G_{t_1,\theta_{t_1}}^{m_1} d\mu_{\theta_{t_1}} \right )^\frac{1}{m_1} + \sum_{k=1}^\infty 2^{-k(s-\theta)} \left ( \dashint_{\mathcal{B}_{2^k r^\prime}(x_0)} G^2 d\mu \right )^\frac{1}{2} \Bigg )
	\end{align*}
for $\delta$ small enough and any weak solution $u \in W^{s,2}(\mathbb{R}^n)$ of $L_{A} u = (-\Delta)^s g$ in $B_{r}(z)$. Thus, since in addition $C_7$ does not depend on $r$ and $r^\prime$, we conclude that the estimate (\ref{hdest}) is satisfied in any ball contained in $B_r(z)$ with respect to $\alpha=t_1$ and $m=m_1$. Therefore, in the case when $\widetilde k=1$, once again by Corollary \ref{aprioriestx} with $\theta_\alpha=\theta_t$ (which is applicable since $m_1<p$) and with $\theta$ replaced by $\widetilde \theta$, we see that 
\begin{align*}
	[u]_{W^{t,p}(B_{r/2}(z))} = & C_8 \left (\dashint_{\mathcal{B}_{r/2}(z)} U_{t_1,\theta_t}^{p} d\mu_{\theta_t} \right )^{\frac{1}{p}} \\ \leq & C_9 \Bigg (\sum_{k=1}^\infty 2^{-k(s-\widetilde \theta)} \left ( \dashint_{\mathcal{B}_{2^k r}(z)} U_{s,\widetilde \theta}^2 d\mu_{\widetilde \theta} \right )^\frac{1}{2} \\
	& + \left ( \dashint_{\mathcal{B}_{r}(z)} G_{t_1,\theta_t}^{p} d\mu_{\theta_t} \right )^\frac{1}{p} +\sum_{k=1}^\infty 2^{-k(s-\widetilde \theta)} \left ( \dashint_{\mathcal{B}_{2^k r}(z)} G_{s,\widetilde \theta}^2 d\mu_{\widetilde \theta} \right )^\frac{1}{2} \\
	\leq & C_{10} \left ([u]_{W^{s,2}(\mathbb{R}^n)} + [g]_{W^{t,p}(B_{r}(z))} + [g]_{W^{s,2}(\mathbb{R}^n)} \right )
\end{align*}
for $\delta$ small enough,
so that in this case the proof is finished. If $\widetilde k>1$, then since $m_2>m_1$, for any $x_0 \in B_r(z)$ and any $r^\prime>0$ such that $B_{r^\prime}(x_0) \subset B_r(z)$, by Proposition \ref{Sobcont}, Corollary \ref{aprioriestx} with $p$ replaced by $m_2$ and Lemma \ref{Urel}, for any weak solution $u \in W^{s,2}(\mathbb{R}^n)$ of $L_{A} u = (-\Delta)^s g$ in $B_{r}(z)$ and $\delta$ small enough we have 
\begin{align*}
	& \left (\frac{1}{\mu_{\theta_{t_2}}(\mathcal{B}_{r^\prime}(x_0))} \int_{B_{r^\prime/2}(x_0)} \int_{B_{r^\prime/2}(x_0)} \frac{(u(x)-u(y))^2}{|x-y|^{n+2t_2}}dydx \right )^\frac{1}{2} \\ \leq & C_{11} (r^\prime)^{-\theta_{t_2}-\frac{n}{m_2}+\frac{\varepsilon_2 \theta_{t_1}}{2}} \left ( \int_{B_{r^\prime/2}(x_0)} \int_{B_{r^\prime/2}(x_0)} \frac{|u(x)-u(y)|^{m_2}}{|x-y|^{n+m_2(t_2+ \varepsilon_2 \theta_{t_1}/2)}}dydx \right )^\frac{1}{m_2} \\ = & C_{11} (r^\prime)^{-\frac{n}{m_2}-\frac{2\theta}{m_2}} \left ( \int_{B_{r^\prime/2}(x_0)} \int_{B_{r^\prime/2}(x_0)} \frac{|u(x)-u(y)|^{m_2}}{|x-y|^{n+m_2(t_1+\varepsilon_2 \theta_{t_1})}}dydx \right )^\frac{1}{m_2} \\ = & C_{12} \left (\dashint_{\mathcal{B}_{r^\prime/2}(x_0)} U_{t_1,\theta_{t_1}}^{m_2} d\mu_{\theta_{t_1}} \right )^{\frac{1}{m_2}} \\ \leq & C_{13} \Bigg (\sum_{k=1}^\infty 2^{-k(s-\theta)} \left ( \dashint_{\mathcal{B}_{2^k r^\prime}(x_0)} U^2 d\mu \right )^\frac{1}{2} \\
	& + \left ( \dashint_{\mathcal{B}_{r^\prime}(x_0)} G_{t_1,\theta_{t_1}}^{m_2} d\mu_{\theta_{t_1}} \right )^\frac{1}{m_2} + \sum_{k=1}^\infty 2^{-k(s-\theta)} \left ( \dashint_{\mathcal{B}_{2^k r^\prime}(x_0)} G^2 d\mu \right )^\frac{1}{2} \Bigg ) \\
	\leq & C_{14} \Bigg (\sum_{k=1}^\infty 2^{-k(s-\theta)} \left ( \dashint_{\mathcal{B}_{2^k r^\prime}(x_0)} U^2 d\mu \right )^\frac{1}{2} \\
	& + \left ( \dashint_{\mathcal{B}_{r^\prime}(x_0)} G_{t_2,\theta_{t_2}}^{m_2} d\mu_{\theta_{t_2}} \right )^\frac{1}{m_2} + \sum_{k=1}^\infty 2^{-k(s-\theta)} \left ( \dashint_{\mathcal{B}_{2^k r^\prime}(x_0)} G^2 d\mu \right )^\frac{1}{2} \Bigg ) ,
\end{align*}
where $C_{14}$ does not depend on $r$ and $r^\prime$, so that (\ref{hdest}) is satisfied in any ball contained in $B_r(z)$ with respect to $\alpha=t_2$ and $m=m_2$. Thus, if $\widetilde k=2$, again by applying Corollary \ref{aprioriestx} with respect to $\theta_\alpha=\theta_t$ and with $\theta$ replaced by $\widetilde \theta$ we see that the desired estimate (\ref{estpy}) holds, so that in this case the proof is finished. If $\widetilde k>2$, then iterating the above procedure $\widetilde k$ times also leads to the estimate (\ref{estpy}), which finishes the proof.
\end{proof}

We are now able to prove an a priori $W^{t,p}$ estimate for equations of the type $L_A u = (-\Delta)^s g$ in the case when $A$ is small in BMO.
\begin{thm} \label{mainright}
	Let $\Omega \subset \mathbb{R}^n$ be a domain, $s \in (0,1)$, $\Lambda \geq 1$, $s<t<\min\{2s,1\}$, $p \in (2,\infty)$ and $R>0$. Then there exists some small enough $\delta = \delta(p,n,s,t,\Lambda) > 0$ such that if $A \in \mathcal{L}_0(\Lambda)$ is $(\delta,R)$-BMO in $\Omega$ and $g$ belongs to $W^{s,2}(\mathbb{R}^n) \cap W^{t,p}(\Omega)$, then for any weak solution $u \in W^{s,2}(\mathbb{R}^n) \cap W^{t,p}(\Omega)$ of the equation $L_{A} u = (-\Delta)^s g$ in $\Omega$ and any relatively compact domain $\Omega^\prime \Subset \Omega$, we have
	\begin{equation} \label{estpym}
		[u]_{W^{t,p}(\Omega^\prime)} \leq C \left ( [u]_{W^{s,2}(\mathbb{R}^n)} + [g]_{W^{t,p}(\Omega)} + [g]_{W^{s,2}(\mathbb{R}^n)} \right ),
	\end{equation}
	where $C=C(n,s,t,\Lambda,R,p,\Omega^\prime,\Omega)>0$.
\end{thm}

\begin{proof}
Fix a relatively compact bounded domain ${\Omega^\prime} \Subset \Omega$ and let $\delta=\delta(p,n,s,t,\Lambda)>0$ be given by Proposition \ref{aprioriestxy}. There exists some $r \in (0,R)$ such that for any $z \in \Omega^\prime$, we have $B_r(z) \Subset \Omega$. Since $A$ is $(\delta,R)$-BMO in $\Omega$, for any $z \in \Omega^\prime$ we conclude that $A$ is $\delta$-vanishing in $B_{r}(z)$. Also, since $u \in W^{t,p}(\Omega)$, we have $u \in W^{t,p}(B_{r}(z))$ for any $z \in \Omega^\prime$. Therefore, by Proposition \ref{aprioriestxy}, for any $z \in \Omega^\prime$ we obtain the estimate
\begin{equation} \label{ap}
	[u]_{W^{t,p}(B_{r/2}(z))}
	\leq C_1 \left ([u]_{W^{s,2}(\mathbb{R}^n)} + [g]_{W^{t,p}(B_{r}(z))} + [g]_{W^{s,2}(\mathbb{R}^n)} \right ) ,
\end{equation}
where $C_1=C_1(n,s,t,\Lambda,p,r)$. \par 
Since $\left \{B_{r/2}(z) \right \}_{z \in {\Omega^\prime}}$ is an open covering of $\overline {\Omega^\prime}$ and $\overline {\Omega^\prime}$ is compact, there exists a finite subcover $\left \{B_{r/2}(z_i) \right \}_{i=1}^N$ of $\overline {\Omega^\prime}$ and hence of $\Omega^\prime$. Now summing over $i=1,...,N$ and using the estimate (\ref{ap}) for $z=z_i$ ($i=1,...,N$) yields
\begin{equation} \label{cest}
	\begin{aligned}
		[u]_{W^{t,p}(\Omega^\prime)} \leq & \sum_{i=1}^N [u]_{W^{t,p}(B_{r/2}(z_i))} \\
		\leq & \sum_{i=1}^N C_2 \left ([u]_{W^{s,2}(\mathbb{R}^n)} + [g]_{W^{t,p}(B_{r}(z))} + [g]_{W^{s,2}(\mathbb{R}^n)} \right )\\
		\leq & C_{2}N \left ([u]_{W^{s,2}(\mathbb{R}^n)} + [g]_{W^{t,p}({\Omega})} + [g]_{W^{s,2}(\mathbb{R}^n)} \right ),
	\end{aligned}
\end{equation}
where $C_{2}=C_{2}(n,s,t,\Lambda,p,r)>0$.
Since $N$ depends only on $\Omega^\prime$ and ${\Omega}$, while $r$ depends only on $R,\Omega^\prime$ and $\Omega$, this proves the estimate (\ref{estpym}), so that the proof is finished.
\end{proof}

\begin{rem} \label{genrem} \normalfont
Since it might be useful in some applications, we remark that the statement of Proposition \ref{mainright} can be generalized to the setting of a right-hand side that is given by a more general nonlocal operator or even by sums of more general nonlocal operators. For some $l \in \mathbb{N}$ and $i=1,...,l$, consider measurable functions $D_i:\mathbb{R}^n \times \mathbb{R}^n \to \mathbb{R}$ such that
\begin{equation} \label{upb}
\sum_{i=1}^l |D_i(x,y)| \leq \Lambda \text{ for almost all } x,y \in \mathbb{R}^n.
\end{equation}
In addition, fix functions $g_i \in W^{s,2}(\mathbb{R}^n) \cap W^{t,p}(\Omega)$ and let $u \in W^{s,2}(\mathbb{R}^n) \cap W^{t,p}(\Omega)$ be a weak solution of the more general nonlocal equation $L_A u =\sum_{i=1}^l L_{D_i} g_i$ in $\Omega$, that is, assume that
\begin{align*}
	& \int_{\mathbb{R}^n} \int_{\mathbb{R}^n} \frac{A(x,y)}{|x-y|^{n+2s}} (u(x)-u(y))(\varphi(x)-\varphi(y))dydx \\
	= & \sum_{i=1}^l \int_{\mathbb{R}^n} \int_{\mathbb{R}^n}  \frac{D_i(x,y)}{|x-y|^{n+2s}}(g_i(x)-g_i(y)) (\varphi(x)-\varphi(y))dydx \quad \forall \varphi \in W^{s,2}_0(\Omega).
\end{align*}
Then the following is true. For $s,t$ and $p$ as in Theorem \ref{mainright}, 
there exists some small enough $\delta = \delta(p,n,s,t,\Lambda) > 0$ such that if $A \in \mathcal{L}_0(\Lambda)$ is $(\delta,R)$-BMO in $\Omega$ for some $R>0$, then for any relatively compact domain $\Omega^\prime \Subset \Omega$, we have the a priori estimate
\begin{equation} \label{estpymd}
	[u]_{W^{t,p}(\Omega^\prime)} \leq C \left ( [u]_{W^{s,2}(\mathbb{R}^n)} + \sum_{i=1}^l \hspace{0.3mm} [g_i]_{W^{t,p}(\Omega)} + \sum_{i=1}^l \hspace{0.3mm} [g_i]_{W^{s,2}(\mathbb{R}^n)} \right ),
\end{equation}
where $C=C(n,s,t,\Lambda,R,p,\Omega^\prime,\Omega)>0$. \par
This is true since the statement of our comparison estimate given by Proposition \ref{appplxy} remains valid for weak solutions $u$ of such equations of the form $L_A u =\sum_{i=1}^l L_{D_i} g_i$, which can be easily seen by using the bound (\ref{upb}) in the estimation of the appropriately adapted integral $I_2$ in \cite[Proposition 5.1]{MeV}, while the adaptations required to account for the summation over $i=1,...,l$ are straightforward and do not change the proofs in any conceptually significant way.
\end{rem}

\section{Proofs of the main results} \label{pmr}
We are now in the position to prove our main results.

\begin{proof}[Proof of Theorem \ref{mainint5}]
	Fix relatively compact bounded domains ${\Omega^\prime} \Subset \Omega_0 \Subset {\Omega^{\prime \prime}} \Subset \Omega$, where we assume that $\Omega_0$ is a smooth domain.
	Let $\delta=\delta(p,n,s,t,\Lambda)>0$ be given by Theorem \ref{mainright} and let $\{\psi_m\}_{m=1}^\infty$ be a sequence of standard mollifiers in $\mathbb{R}^{n}$ with the properties 
	\begin{equation} \label{molprop}
		\psi_m \in C_0^\infty(B_{1/m}), \quad \psi_m \geq 0, \quad \int_{\mathbb{R}^{n}} \psi_m(x)dx=1 \quad \text{for all } m \in \mathbb{N}.
	\end{equation}
	For any $m \in \mathbb{N}$ and $x \in \Omega_m := \left \{x \in \Omega \mid \textnormal{dist}(x, \partial \Omega) > 1/m \right \}$, we now define
	$$ f_m(x):= \int_{\Omega} f(y)\psi_m(x-y)dy.$$
	Next, observe that there exists some large enough $m_0 \in \mathbb{N}$, such that $\Omega^{\prime \prime} \subset \Omega_m$ for all $m \geq m_0$.
	Since $f \in L^\frac{np}{n+(2s-t)p}_{loc}(\Omega)$ and ${\Omega^{\prime \prime}} \Subset \Omega$, by standard properties of mollifiers we have
	\begin{equation} \label{appf}
		f_m \xrightarrow{m \to \infty} f \text{ in } L^\frac{np}{n+(2s-t)p}({\Omega^{\prime \prime}})
	\end{equation}
	and $f_m \in L^\infty({\Omega^{\prime \prime}})$ for any $m \geq m_0$. In addition, for any $m \geq m_0$, by \cite[Proposition 4.1]{MeN} there exists a unique weak solution $u_m \in W^{s,2}(\mathbb{R}^n)$ of the Dirichlet problem 
	\begin{equation} \label{constcof3k}
		\begin{cases} \normalfont
			L_{A} u_m = f_m & \text{ in } \Omega^{\prime \prime} \\
			u_m = u & \text{ a.e. in } \mathbb{R}^n \setminus \Omega^{\prime \prime}.
		\end{cases}
	\end{equation}
	Since $w_m:=u-u_m \in W_0^{s,2}(\Omega^{\prime \prime})$ is a weak solution of the equation $L_A w_m = f-f_m$ in $\Omega^{\prime \prime}$, in view of using $w_m$ itself as a test function in this equation, along with H\"older's inequality and the fractional Sobolev inequality (see \cite[Theorem 6.5]{Hitch}), we obtain
	\begin{equation} \label{basest}
	\begin{aligned}
		[w_m]_{W^{s,2}(\mathbb{R}^n)}^2 \leq & \Lambda^{-1} \int_{\mathbb{R}^n} \int_{\mathbb{R}^n} A(x,y) \frac{(w_m(x)-w_m(y))^2}{|x-y|^{n+2s}}dydx \\ = & \Lambda^{-1} \int_{\Omega^{\prime \prime}} (f-f_m)w_m dx \\ \leq & \Lambda^{-1} ||f-f_m||_{L^\frac{2n}{n+2s}(\Omega^{\prime \prime})} ||w_m||_{L^\frac{2n}{n-2s}(\mathbb{R}^n)} \\ \leq & C_1 ||f-f_m||_{L^\frac{np}{n+(2s-t)p}(\Omega^{\prime \prime})} [w_m]_{W^{s,2}(\mathbb{R}^n)},
	\end{aligned}
	\end{equation}
	where $C_{1}=C_{1}(n,s,t,p,\Lambda,\Omega^{\prime \prime})>0$, so that along with (\ref{appf}), we deduce that
		\begin{align*}
		[w_m]_{W^{s,2}(\mathbb{R}^n)} \leq C_{2} ||f-f_m||_{L^\frac{np}{n+(2s-t)p}(\Omega^{\prime \prime})} \xrightarrow{k \to \infty} 0
		\end{align*}
	and
	\begin{equation} \label{UktoU}
		\lim_{m \to \infty} [u_m]_{W^{s,2}(\mathbb{R}^n)}
		= [u]_{W^{s,2}(\mathbb{R}^n)}.
	\end{equation}
	Next, for any $m \in \mathbb{N}$ let $g_m \in W^{s,2}(\mathbb{R}^n)$ be the unique weak solution of the Dirichlet problem
	\begin{equation} \label{constcof3gk}
		\begin{cases} \normalfont
			(-\Delta)^s g_m = f_m & \text{ in } \Omega^{\prime \prime} \\
			g_m = 0 & \text{ a.e. in } \mathbb{R}^n \setminus \Omega^{\prime \prime}.
		\end{cases}
	\end{equation}
	Then by a similar reasoning as in (\ref{basest}), each function $g_m$ satisfies the estimate
	\begin{equation} \label{fkestimate}
		[g_m]_{W^{s,2}(\mathbb{R}^n)} \leq C_2 ||f_m||_{L^{\frac{np}{n+(2s-t)p}}(\Omega^{\prime \prime})},
	\end{equation}
	where $C_2=C_2(n,s,t,p,\Omega^{\prime \prime})>0$.
	In addition, by the local $H^{2s,p}$ estimates for the fractional Laplacian (see \cite[Theorem 4.4]{MeV}), we have the estimate
	\begin{equation} \label{H2spa}
		||g_m||_{H^{2s,\frac{np}{n+(2s-t)p}}(\Omega_0)} \leq C_3 ||f_m||_{L^{\frac{np}{n+(2s-t)p}}(\Omega^{\prime \prime})},
	\end{equation}
	where $C_3=C_3(n,s,t,p,\Omega_0,\Omega^{\prime \prime})>0$.
	Also, by Proposition \ref{BesselTr}, we have 
	\begin{equation} \label{embend}
		[g_m]_{W^{t,p}(\Omega_0)} \leq C_4 ||g_m||_{H^{2s,\frac{np}{n+(2s-t)p}}(\Omega_0)},
	\end{equation}
	where $C_4=C_4(n,s,t,p,\Omega_0)>0$. In view of (\ref{constcof3k}) and (\ref{constcof3gk}), $u_m$ is a weak solution of the equation
	$$ L_{A} u_m= (-\Delta)^s g_m \text{ in } \Omega^{\prime \prime}.$$
	Since $f_m \in L^\infty(\Omega^{\prime \prime})$, by \cite[Theorem 1.4]{MeV} we have $u_m \in C^{\beta}_{loc}(\Omega_0)$ for any $\beta \in (0,\min\{2s,1\})$ and thus $u_m \in W^{t,p}(\Omega_0)$. Therefore, by Theorem \ref{mainright}, (\ref{fkestimate}), (\ref{embend}) and (\ref{H2spa}), we have
	\begin{align*}
		[u_m]_{W^{t,p}(\Omega^{\prime})}
		\leq & C_5 \left ([u_m]_{W^{s,2}(\mathbb{R}^n)} + [g_m]_{W^{t,p}(\Omega_0)} + [g_m]_{W^{s,2}(\mathbb{R}^n)} \right ) \\
		\leq & C_6 \left ([u_m]_{W^{s,2}(\mathbb{R}^n)} + ||f_m||_{L^{\frac{np}{n+(2s-t)p}}(\Omega^{\prime \prime})} \right ) ,
	\end{align*}
	where all constants depend only on $n,s,t,\Lambda,p,\Omega^\prime,\Omega^{\prime \prime}$ and $\Omega_0$.
	Combining the previous display with Fatou's Lemma (which is applicable after passing to a subsequence if necessary), (\ref{UktoU}) and (\ref{appf}), we conclude that
	\begin{equation} \label{liminfest}
		\begin{aligned}
			[u]_{W^{t,p}(\Omega^{\prime})} \leq & \liminf_{m \to \infty} [u_m]_{W^{t,p}(\Omega^{\prime})} \\
			\leq & C_7 \lim_{m \to \infty} \left ([u_m]_{W^{s,2}(\mathbb{R}^n)} + ||f_m||_{L^{\frac{np}{n+(2s-t)p}}(\Omega^{\prime \prime})} \right ) \\
			= & C_{7} \left ([u]_{W^{s,2}(\mathbb{R}^n)} + ||f||_{L^{\frac{np}{n+(2s-t)p}}(\Omega^{\prime \prime})} \right ) ,
		\end{aligned}
	\end{equation}
	where $C_7=C_7(n,s,t,\Lambda,p,\Omega^{\prime},\Omega^{\prime \prime})>0$.
	This proves the estimate (\ref{Wstest}). \par 
	The assertion that $u \in L^p_{loc}(\Omega)$ now follows by a simple iteration argument for which we refer to the proof of \cite[Theorem 9.1]{MeV}, so that we conclude that $u \in W^{t,p}_{loc}(\Omega)$. This finishes the proof.
\end{proof}

\begin{proof}[Proof of Theorem \ref{mainint5z}]
The case when $t=s$ follows directly from \cite[Theorem 1.1]{MeV}. Next, fix some $p>2$, some $s<t<\min\{2s,1\}$ and consider the corresponding $\delta=\delta(p,n,s,t,\Lambda)>0$ given by Theorem \ref{mainint5}. Since $A$ is assumed to be VMO in $\Omega$, there exists some $R>0$ such that $A$ is $(\delta,R)$-BMO in $\Omega$. Thus, by Theorem \ref{mainint5} we obtain that $u \in W^{t,p}_{loc}(\Omega)$ whenever $f \in L^\frac{np}{n+(2s-t)p}_{loc}(\Omega)$, which finishes the proof.
\end{proof}

\begin{proof}[Proof of Theorem \ref{mainint5zx}]
	Fix some $t$ such that $s \leq t < \min\{2s,1\}$, some $q \in \left (\frac{2n}{n+2(2s-t)},\infty \right )$ and some $f \in L^q_{loc}(\Omega)$. First, we assume that $q<\frac{n}{2s-t}$. Then we have $n>(2s-t)q$ and set $p:=\frac{nq}{n-(2s-t)q}>2$, so that we have $q=\frac{np}{n+(2s-t)p}$ and thus $f \in L^\frac{np}{n+(2s-t)p}_{loc}(\Omega)$.
	Therefore, by Theorem \ref{mainint5z} we obtain $u \in W^{t,p}_{loc}(\Omega) = W^{t,\frac{nq}{n-(2s-t)q}}_{loc}(\Omega)$. \par 
	If on the other hand $q \geq \frac{n}{2s-t}$, then for any $p \in (2,\infty)$ we have $\frac{np}{n+(2s-t)p} \leq q$ and thus $f \in L^\frac{np}{n+(2s-t)p}_{loc}(\Omega)$, so that again by Theorem \ref{mainint5z} we obtain $u \in W^{t,p}_{loc}(\Omega)$.
	In view of Proposition \ref{Sobcont}, the conclusion that $u \in W^{t,p}_{loc}(\Omega)$ for $p \in (2,\infty)$ and for any $t$ in the range $s \leq t < \min\{2s,1\}$ also implies that $u \in W^{t,p}_{loc}(\Omega)$  for any $p \in (1,\infty)$, so that the proof is finished.
\end{proof}

\begin{proof}[Proof of Theorem \ref{higherdiff}]
Fix some $s<t<\min\{2s,1\}$ and some $t^\prime$ such that $t<t^\prime<\min\{2s,1\}$. Since $f \in L^2_{loc}(\Omega)$ and $\frac{2n}{n+2(2s+t^\prime)}<2$, Theorem \ref{mainint5zx} implies that $u \in W^{t^\prime,\frac{2n}{n-2(2s-t^\prime)}}_{loc}(\Omega)$. Since $\frac{2s}{n-2(2s-t^\prime)} >2$, by Proposition \ref{Sobcont} we arrive at $u \in W^{t,2}_{loc}(\Omega)$, so that the proof is finished.
\end{proof}

\begin{rem} \label{endremark} \normalfont
As already indicated in Remark \ref{mainrem}, our main results remain valid for another class of coefficients $A$ that in general might not be VMO in $\Omega$.
Namely, the conclusions of Theorem \ref{mainint5z}, Theorem \ref{mainint5zx} and Theorem \ref{higherdiff} remain true if instead we assume that there exists some small $\varepsilon>0$ such that
\begin{equation} \label{contkernel}
	\lim_{h \to 0} \sup_{\substack{_{x,y \in K}\\{|x-y| \leq \varepsilon}}} |A(x+h,y+h)-A(x,y)| =0 \quad \text{for any compact set } K \subset \Omega.
\end{equation}
In fact, in the present paper we only use the assumption that $A$ is VMO in $\Omega$ in order to ensure that the H\"older estimate for corresponding homogeneous equations given by (\ref{veq}) holds, which in this case is guaranteed by the results from \cite{MeV}. If instead $A$ satisfies the assumption (\ref{contkernel}), then this H\"older estimate actually follows from \cite[Theorem 1.1]{MeH} combined with \cite[Lemma 5.1]{MeV}, so that our proofs and main results remain valid under the assumption (\ref{contkernel}). \par
As mentioned, the condition (\ref{contkernel}) is for example satisfied in the case when $A \in \mathcal{L}_0(\Lambda)$ is translation invariant in $\Omega$, that is, if we have $A(x,y)=a(x-y)$ for all $x,y \in \Omega$ and some measurable function $a: \mathbb{R}^n \to \mathbb{R}$. Since in this case $A$ is otherwise not required to satisfy any additional smoothness assumption, $A$ might not be VMO in $\Omega$ but still satisfies (\ref{contkernel}).
\end{rem}


\bibliographystyle{amsplain}

\end{document}